\documentclass[10pt,twoside,a4paper,english,reqno]{amsart}

\usepackage{graphicx,amsmath,amscd,amssymb,color}
\usepackage[pdftex,colorlinks=false, backref]{hyperref}
\usepackage{upref}

\newcommand{\abs}[1]{\left|#1\right|}
\newcommand{\Set}[1]{\left\{#1\right\}}
\newcommand{\norm}[1]{\left \|#1\right\|}

\def\dist{\operatorname{dist}}
\def\ident{\operatorname{Id}}
\def\ran{\operatorname{Ran}}
\def\dom{\operatorname{Dom}}
\def\div{\operatorname{div}}

\def\sign{\operatorname{sign}}
\def\R{{\mathbb R}}
\def\char{{1\!\mbox{\rm l}}}
\def\eps{{\varepsilon}}
\def\dsp{\displaystyle}
\def\Osc{{\mbox{\rm Osc}}}
\def\NN{{\mathbb{N}}}

\def\Z{{\mathbb Z}}

\def\loc{\mathrm{loc}}

\numberwithin{equation}{section}
\allowdisplaybreaks

\newtheorem{theo}{Theorem}[section]
\newtheorem{rem}[theo]{Remark}
\newtheorem{defi}[theo]{Definition}
\newtheorem{lem}[theo]{Lemma}
\newtheorem{cor}[theo]{Corollary}
\newtheorem{prop}[theo]{Proposition}

\newtheorem{example}[theo]{Example}

\begin{document}

\title[On scalar conservation laws with discontinuous flux]
{A theory of $L^1$-dissipative solvers for scalar conservation laws with discontinuous flux}

\date{\today}

\author[B. Andreianov]{B. Andreianov}
\address[Boris Andreianov]{\newline
         Laboratoire de Math\'ematiques\newline
         Universit\'e de Franche-Comt\'e\newline
         16 route de Gray\newline
         25 030 Besanc$\hspace*{-6pt}_{_{^{\mathsf ,}}}\!\!$on
         Cedex, France}
\email[]{boris.andreianov@univ-fcomte.fr}

\author[K. H. Karlsen]{K. H. Karlsen}
\address[Kenneth Hvistendahl Karlsen]{\newline
         Centre of Mathematics for Applications \newline
         University of Oslo\newline
         P.O. Box 1053, Blindern\newline
         N--0316 Oslo, Norway}
\email[]{kennethk@math.uio.no}
\urladdr{http://folk.uio.no/kennethk/}

\author[N. H. Risebro]{N. H. Risebro}
\address[Nils Henrik Risebro]{\newline
         Centre of Mathematics for Applications \newline
         University of Oslo\newline
         P.O. Box 1053, Blindern\newline
         N--0316 Oslo, Norway}
\email[]{nilsr@math.uio.no}
\urladdr{http://folk.uio.no/nilshr/}



% 35L65 Conservation laws
% 35R05 Partial differential equations with discontinuous coefficients or data

\begin{abstract}
We propose a general framework for the study of $L^1$ contractive
semigroups of solutions to conservation laws
with discontinuous flux:
\begin{equation}\tag{CL}
u_t  + \mathfrak{f}(x,u)_x=0, \qquad
\mathfrak{f}(x,u)=
\begin{cases}
    f^l(u),& x<0,\\
    f^r(u), & x>0,
\end{cases}
\end{equation}
where the fluxes $f^{l},f^{r}$ are mainly assumed to be continuous. 
Developing the ideas of a number of preceding works (Baiti and Jenssen
\cite{BaitiJenssen}, Audusse and Perthame \cite{AudussePerthame},
Garavello et al. \cite{GaravelloAndAl}, Adimurthi et al.~\cite{AdimurthiMishraGowda}, 
B\"urger et al.~\cite{BurgerKarlsenTowers}), we claim 
that the whole admissibility
issue is reduced to the selection of a family 
of ``elementary solutions'', which are piecewise 
constant weak solutions of the form
$$
c(x)=c^l\char_{\Set{x<0}}+c^r\char_{\Set{x>0}}.
$$
We refer to such a family as a ``germ''. It is well known 
that (CL) admits many different $L^1$ contractive 
semigroups, some of which reflects different physical applications.  
We revisit a number of the existing admissibility (or entropy) 
conditions and identify the germs that underly these conditions. 
We devote specific attention to the ``vanishing viscosity''
germ, which is a way to express 
the ``$\Gamma$-condition" of Diehl \cite{Diehl2008}. 
For any given germ, we formulate ``germ-based" admissibility conditions
in the form of a trace condition on the flux discontinuity line
$\Set{x=0}$ (in the spirit of Vol'pert \cite{Volpert}) and 
in the form of a family of global entropy inequalities (following
Kruzhkov \cite{Kruzhkov} and Carrillo \cite{Carrillo}). We
characterize those germs that lead to the $L^1$-contraction property 
for the associated admissible solutions.  Our approach offers a 
streamlined and unifying perspective on many of the 
known entropy conditions, making it possible to 
recover earlier uniqueness results under weaker conditions than before, 
and to provide new results for other less studied problems. 
Several strategies for proving the existence of admissible solutions are discussed,  and 
existence results are given for fluxes satisfying some additional conditions. 
These are based on convergence results either
for the vanishing viscosity method (with standard viscosity or
with specific viscosities ``adapted'' to the choice of a germ), or
for specific germ-adapted finite volume schemes.
\end{abstract}

\maketitle

\newpage

\tableofcontents

\newpage

\section{Introduction}

\subsection{$L^1$-dissipativity for scalar
conservation laws}\label{ssec:L1DissipativityForSCL}

We are interested in the well-posedness of the Cauchy
problem for a general scalar conservation law:
\begin{align}
    & u_t+\div \mathfrak{f}(t,x,u)=0, \label{eq:MultiDConsLaw}
    \\ & u|_{t=0}=u_0. \label{eq:InitialCond}
\end{align}
It is well known that classical solutions to this problem
may not exist globally since discontinuities can develop in finite time; hence 
\eqref{eq:MultiDConsLaw},\eqref{eq:InitialCond} must be interpreted 
in the weak (distributional) sense. However, weak 
solutions are in general not unique, and so additional admissibility criteria 
are needed to single out a unique solution; these so-called 
entropy conditions are usually motivated
by a very careful inspection of the underlying physical 
phenomena in presence of discontinuous or rapidly changing solutions.

Consider solutions which take values in a closed interval $U\subset\R$. 
Whenever the flux function $\mathfrak{f}:\R_+\times\R^N \times U\mapsto\R^N$ is
sufficiently regular in all variables, the classical notion of Kruzhkov
entropy solutions in the $L^\infty$ framework \cite{Kruzhkov}  (cf.~Volpert \cite{Volpert} 
for the $BV$ setting and Ole\u{\i}nik \cite{Oleinik}) and its further 
extensions to bounded domains, measure-valued solutions, 
renormalized solutions, and degenerate parabolic problems, cf., e.g., 
\cite{AmmarCarrilloWittbold,BLN,BenilanCarrilloWittbold,Carrillo,Otto:BC,Panov-precomp-first,Panov-boundary-trace}, have provided a rather complete theory of well-posedness.
An equivalent notion of kinetic solutions was later
formulated in \cite{LionsPerthameTadmor,Perthame-book}, which
allowed for a deeper study of regularity and
compactness properties of admissible solutions. 
By far the most studied case is the autonomous equation
\begin{equation}\label{eq:ConsLaw-autonomous}
    u_t+ \div \mathfrak{f}(u)=0
\end{equation}
Under mild regularity assumptions on the flux $\mathfrak{f}$,  it is well known
that the Kruzhkov entropy solutions of
\eqref{eq:ConsLaw-autonomous},\eqref{eq:InitialCond} (with
$u_0\in L^1\cap L^\infty$) form an $L^1$-contractive
and order-preserving semigroup on $L^1\cap L^\infty$
\cite{Kruzhkov,Keyfitz,Benilan,Crandall,KruzhkovHildebrand,KruzhkovPanov-Dokl,BenilanKruzhkov,AndreianovBenilanKruzhkov,MalikiToure}. 
Moreover, this is the unique semigroup of this kind which 
admits all the trivial \textit{constant solutions} of \eqref{eq:ConsLaw-autonomous}. 
This semigroup is generally viewed as the one representing the 
physically relevant solutions to scalar conservation laws. 
Solution semigroups that are Lipschitz in $L^1$
but not contractive may exist and are of physical 
interest (see \cite{LeFloch}), but
their study will not be addressed in the present paper.

A comparable theory for conservation laws with discontinuous 
flux is still not available, although these equations have received
intense attention in last fifteen years; see
\cite{AdimurthiGowda,AdimurthiJaffreGowda,AdimurthiMishraGowda,AudussePerthame,
BachmannVovelle,BaitiJenssen,BurgerGarciaKarlsenTowers,BurgerKarlsenMishraTowers,BurgerKarlsenTowers,ChenEvanKlingenberg,ColomboGoatin,Diehl1,Diehl2,Diehl3,Diehl2008,GaravelloAndAl,GimseRisebro1,GimseRisebro2,Jimenez,JimenezLevi,Kaasschietter,KarlsenTowers,KarlsenRisebroTowers2002a,KarlsenRisebroTowers2002b,KarlsenRisebroTowers2003,KlingenbergRisebro,MishraThesis,Mitrovic,Ostrov,Panov-precomp-ARMA,Panov-AudussePerthameRevisited,SeguinVovelle,Towers2000,Towers2001}
(and additional references therein) for a number of different admissibility
criteria, existence and/or uniqueness results, which we partially 
revisit in Section \ref{sec:Examples}. Recently, it was
pointed out explicitly by Adimurthi, Mishra, and Veerappa Gowda in
\cite{AdimurthiMishraGowda} that for the 
case $\mathfrak{f}=\mathfrak{f}(x,u)$ with $\mathfrak{f}$
piecewise constant in $x$, there may exist many different
$L^1$-contractive semigroups of solutions to
\eqref{eq:MultiDConsLaw}. Different semigroups correspond to
different physical phenomena modeled by the same equation, but
with different dissipative processes occurring on the
discontinuities of the solutions. We refer to
\cite{BurgerKarlsenMishraTowers} for a comparison of a clarifier-thickener
model and a porous medium model, which leads to the same formal
conservation law, but with two distinct semigroups of
physically relevant solutions. Notice that even for the classical
equation \eqref{eq:ConsLaw-autonomous}, non-Kruzhkov
$L^1$-contractive semigroups can also
be constructed (see \cite{ColomboGoatin,AndrGoatinSeguin}), by an 
approach much similar to the one 
of \cite{AdimurthiMishraGowda,BurgerKarlsenTowers} 
and the present paper.

In this paper  we formulate 
a streamlined, unifying framework encompassing the
different notions of entropy solutions for \eqref{eq:MultiDConsLaw}. 
For the multi-dimensional problem examined
in the sequel paper \cite{AKR-II} we consider Carath\'eodory 
functions $\mathfrak f(t,x,u)$ that are piecewise
Lipschitz continuous in $(t,x)$ and 
locally Lipschitz continuous in $u$. 
The present paper focuses on the model one-dimensional case with
\begin{equation}\label{eq:flux-f-1d-type}
    \mathfrak{f}=\mathfrak{f}(x,u)=
    f^l(u)\char_{\Set{x<0}}+f^r(u)\char_{\Set{x>0}},
\end{equation}
and $f^l,f^r:U\to \R$ being merely continuous (in the existing literature
the fluxes are assumed to Lipschitz continuous). We note that a generalization
to a piecewise constant in $x$ and continuous
in $u$ flux function $\mathfrak{f}(x,u)$ is straightforward.

The initial datum $u_0$ is assumed to belong to $L^\infty(\R;U)$.
The presence of a source term $s(t,x)$ in the conservation
law \eqref{eq:MultiDConsLaw} is easy to take into account, for example in
the case $U=\R$ and  $s(t,\cdot)\in L^\infty(\R^N)$ for a.e.~$t>0$
with $\int_0^T \|s(t,\cdot)\|_{L^\infty}\, dt<\infty$ for all $T>0$. 
For the sake of simplicity, we always take $s\equiv 0$.

\begin{defi}\label{def:L1-solver}
A semigroup $(S_t)_{t\geq 0}$, $S_t:D\rightarrow L^\infty(\R;U)$,
defined on a subset $D\subset L^\infty(\R;U)$, is called an
$L^1$-dissipative solver for \eqref{eq:MultiDConsLaw},\eqref{eq:flux-f-1d-type} if

\noindent $\bullet$  for all $u_0\in D$, the trajectory
$u(t,x):=S_t(u_0)(x)$ of the semigroup $S_t$ gives a
solution\footnote{i.e., a weak solution which is admissible in a
sense specified later.} to the problem
\eqref{eq:MultiDConsLaw} with flux \eqref{eq:flux-f-1d-type} and initial data $u_0$.

\noindent $\bullet$ for all $u_0,\hat u_0\in D$, $u:=S_t(u_0)$ and
$\hat u:=S_t(\hat u_0)$  satisfy the ``Kato inequality''
\begin{equation}\label{eq:L1Dissipativity}
    \begin{split}
        &\forall\,\xi\in \mathcal{D}([0,\infty) \times \R),\; \xi\geq 0,\\
        &-\int_{\R_+} \int_{\R} \Bigl\{ \abs{u-\hat u}\xi_t
        +\sign(u-\hat u)(\mathfrak f(x,u)-\mathfrak f(x,\hat u))\xi_x \Bigr\}
         \leq  \int_{\R} |u_0-\hat u_0|\,\xi(0,x) .
    \end{split}
\end{equation}
\end{defi}
Letting $\xi$ go to the characteristic
function of $(0,t)\times\R$ (cf., e.g., \cite{Benilan,KruzhkovHildebrand}),
we deduce from \eqref{eq:L1Dissipativity} the $L^1$-contractivity property
\begin{equation}\label{eq:L1Contraction}
    \text{whenever $|u_0-\hat u_0|\in L^1(\R)$},
    \quad \int_{\R}  \abs{u-\hat u}(t)
    \leq \int_{\R} |u_0-\hat u_0|
    \quad \text{for a.e.~$t>0$}.
\end{equation}

We could also have required an $L^1$-dissipative solver 
to be order-preserving, in which case $\sign(u-\hat u)$ and $\abs{u-\hat u}$ 
in \eqref{eq:L1Dissipativity},\eqref{eq:L1Contraction}
are replaced by $\sign^+(u-\hat u)$ and $(u-\hat u)^+$, respectively.
It turns out that the  $L^1$-dissipative solvers that we
consider are automatically order-preserving (cf.~\cite{CrandallTartar}).

Notice that in several spatial dimensions, mere continuity of the flux
$\mathfrak{f}$ in $u$ is not sufficient to obtain $L^1$ contractivity
from the Kato inequality; see in particular the counterexample of
Kruzhkov and Panov \cite{KruzhkovPanov-Dokl};
cf.~\cite{Benilan,KruzhkovHildebrand,BenilanKruzhkov,AndreianovBenilanKruzhkov,MalikiToure}
for some sufficient conditions for \eqref{eq:L1Contraction} to hold.

\subsection{Analysis in terms of Riemann solvers}\label{ssec:RiemannAnalysis}

Our primary goal is to formulate a convenient unifying framework for the
study of different $L^1$-dissipative solvers. We will work with the 
basic one-dimensional model equation
\begin{equation}\label{eq:ModelProb}
    u_t + \mathfrak{f}(x,u)_x=0,
    \qquad
    \mathfrak{f}=\mathfrak{f}(x,u)=f^l(u)\char_{\Set{x<0}}+f^r(u)\char_{\Set{x>0}},
\end{equation}
which the existing literature has targeted as the main example
for understanding the admissibility 
issue in the case of a spatially discontinuous flux.

With the admissibility of solutions in the regions
$\{x>0\}$ and $\{x<0\}$ being understood in the sense
of Kruzhkov's entropy solutions, it remains to define the
admissibility of a solution at the discontinuity set
$\Sigma=(0,{\infty}) \times \Set{0}$. 
As suggested by Garavello, Natalini, Piccoli, and Terracina
\cite{GaravelloAndAl}, the admissibility issue should reduce
to the choice of a Riemann solver at $x=0$.
That is, to each pair $(u_-,u_+)\in U \times U$ we have to
assign a weak solution $u:=\mathcal {RS}(u_-,u_+)$
of \eqref{eq:ModelProb} with the initial datum
\begin{equation}\label{eq:RiemannInitialCond}
    u_0(x)=
    \begin{cases}
        u_-, & x<0,\\
        u_+, & x>0.
    \end{cases}
\end{equation}
If \eqref{eq:L1Dissipativity} and \eqref{eq:L1Contraction}
turn out to be true, such a solution should be unique.
If we assume that
\begin{equation*}\tag{A1}
    \begin{split}
        & \text{the notion of solution for \eqref{eq:ModelProb} is invariant}\\
        & \text{under the scaling $(t,x)\mapsto(kt,kx)$, $k>0$},
    \end{split}
\end{equation*}
it follows that the unique solution to \eqref{eq:ModelProb},\eqref{eq:RiemannInitialCond} is
self-similar, i.e., it only depends on the ratio $\xi:=x/t$.
Furthermore, \eqref{eq:L1Dissipativity} should hold 
with $u,\hat u$ being solutions $\mathcal{RS}(u_-,u_+)$,
$\mathcal{RS}(\hat u_-,\hat u_+)$ of two different Riemann problems.

If, in addition, we assume
\begin{equation*}\tag{A2}
    \text{any solution of \eqref{eq:ModelProb} is a Kruzhkov entropy
    solution in $\R_+\times \left(\R\setminus\Set{0}\right)$},
\end{equation*}
it follows that the solution $u=\mathcal {RS}(u_-,u_+)$ is monotone in $\xi$
on each side of the discontinuity; thus, it has strong pointwise left and right traces 
on $\Sigma=(0,{\infty}) \times \Set{0}$. Let us
denote these traces by $\gamma^l u$ and $\gamma^r u$,
respectively. If $u^{l,r}:=\gamma^{l,r}u$, the weak
formulation of \eqref{eq:ModelProb} yields
the Rankine-Hugoniot condition
\begin{equation}\label{eq:RankineHugoniot}
    f^l(u^l)=f^r(u^r).
\end{equation}
Hence, we can associate to the pair $(u^l,u^r)$ the
following stationary weak solution of \eqref{eq:ModelProb}:
\begin{equation}\label{eq:PiecewiseC(x)}
    \begin{cases}
        u^l, & x<0,\\
        u^r, & x>0.
    \end{cases}
\end{equation}
Any solution of the form \eqref{eq:PiecewiseC(x)}, with $(u^l,u^r)\in U \times U$,
will be called an ``\textit{elementary solution}", and be
identified with the pair $(u^l,u^r)$.

Under assumptions (A1) and (A2), we conclude that 
the solution $\mathcal {RS}(u_-,u_+)$ of the Riemann 
problem should consist of the following ingredients:\\[5pt]
(RPb-sol.)
\begin{tabular}{lp{140mm}}
    $\bullet$ & the standard Kruzhkov self-similar solution joining the state $u_-$\\
    & at $t=0, x<0$ to some state $u^l$ at $t>0,x=0^-$;\\

    $\bullet$ & the jump joining the state $u^l$ at $x=0^-$ to a state $u^r$ at $x=0^+$\\
    & such that \eqref{eq:RankineHugoniot} holds;\\

    $\bullet$ & the standard Kruzhkov self-similar solution joining the state $u^r$\\
    & at $t>0,x=0^+$  to the state $u_+$ at $t=0, x>0$.\\
\end{tabular}

\begin{rem}\label{rem:RPb-sol-with-zero-speed}
In {\rm (RPb-sol.)}, it is convenient to consider the situation
in which the wave fans joining $u_-$ to $u^l$, resp.~$u^r$ to $u_+$,
might contain zero-speed shocks. In this case, the same a.e.~defined
solution may correspond to two or more different pairs $(u^l,u^r)$
of intermediate states at $x=0^\pm$, and $\gamma^{l,r} u$
may differ from $u^{l,r}$.
\end{rem}

Whenever solutions of Riemann problems
are seen as trajectories of an $L^1$-dissipative
solver for \eqref{eq:ModelProb}, the different ``elementary
solution'' pairs $(u^l,u^r)$, $(\hat u^l,\hat u^r)$ that can
be used in $\text{(RPb-sol)}$ ought to satisfy
\begin{equation}\label{eq:L1D-States}
    \begin{split}
        q^l(u^l,\hat u^l)
        & := \sign(u^l-\hat u^l)( f^l(u^l)- f^l(\hat u^l))
        \\ & \geq \sign(u^r-\hat u^r)( f^r(u^r)- f^r(\hat u^r))=:q^r(u^r,\hat u^r).
    \end{split}
\end{equation}
Indeed, denote by $\hat u^{l,r}$ the
left and right traces on $\Sigma$ of $\hat u=\mathcal {RS}(\hat u_-,\hat u_+)$ and
let the test function $\xi$ in the Kato inequality \eqref{eq:L1Dissipativity}
converge to the characteristic function of $(0,T) \times \Set{0}$. We then
easily arrive at \eqref{eq:L1D-States}.

\subsection{Admissibility germs and uniqueness}\label{ssec:AxiomaticAdmissibility}

The property in \eqref{eq:L1D-States} of allowed jumps across $\Sigma$
has been recognized in many works as crucial for the uniqueness of admissible solutions;
it was the key ingredient of the uniqueness proofs in for example the works
\cite{AdimurthiMishraGowda,GaravelloAndAl,BurgerKarlsenTowers}.

Our angle of attack is to ``axiomatize'' property \eqref{eq:L1D-States}.
More precisely, we claim that admissibility can be defined directly from the
choice of a set $\mathcal{G}\subset U \times U\subset\R^2$ such
that \eqref{eq:RankineHugoniot} and \eqref{eq:L1D-States} hold for all
$(u^l,u^r),(\hat u^l,\hat u^r)\in \mathcal{G}$. Such a set
$\mathcal{G}$ will be called an \textit{$L^1$-dissipative admissibility germ}
(an \textit{$L^1D$ germ} for short). If an $L^1D$
germ $\mathcal{G}$  admits a unique maximal extension, still satisfying
\eqref{eq:RankineHugoniot} and \eqref{eq:L1D-States}, the
germ is said to be \textit{definite}.

Our results in Section \ref{sec:ModelProblem} can be formulated as follows:
if $\mathcal{G}$ is a definite  germ, it corresponds to a unique $L^1$-dissipative
solver $S_t^{\mathcal{G}}$ for \eqref{eq:ModelProb} such that\\[5pt]
($S_t^\mathcal{G}\!$-sol.)
\begin{tabular}{lp{140mm}}
    $\bullet$ & for any trajectory $u$ of $S_t^\mathcal{G}$, the
    restrictions of $u$ to the domains\\
    & $\R_+ \times (-\infty,0)$ and $\R_+ \times (0,\infty)$ are
    entropy solutions\\
    & in the sense of Kruzhkov;\\

    $\bullet$ & roughly speaking, the pair of left and right traces\\
    & $(\gamma^l u, \gamma^r u)$ of $u$ on $\R_+\times \Set{0}$
    belongs to $\mathcal{G}^*$ pointwise,\\
    & where $\mathcal{G}^*$ is the unique maximal $L^1D$
    extension of $\mathcal{G}$.
 \end{tabular}\\[3pt]

The set $\mathcal{G}$ is exactly the set of ``elementary solutions''
contained within $S_t^{\mathcal{G}}$. Without non-degeneracy
assumptions on the fluxes $f^l$ and $f^r$, the second statement in
($S_t^\mathcal{G}\!$-sol.) should be made precise.
Following the idea of \cite{Szepessy,Vallet}, the pair of (weak)
traces $(\gamma^l u,\gamma^r u)$ can be defined as a Young measure
on $\R^2$, which is the ``nonlinear weak-$\star$ limit'' (see \cite{EGH,GallouetHubert})
of $\Bigl(u(t,-h_n),u(t,h_n)\Bigr)$, with $(h_n)_{n>1}$, $h_n\downarrow 0$,
being a sequence of Lebesgue points of the
map $x\mapsto \Bigl(u(\cdot,-x),u(\cdot,x)\Bigr)$.
What we mean is that the support of this Young measure
is contained in $\mathcal{G}$, for a.e.~$t>0$.
A more practical way to understand the boundary
trace issue is suggested in Definition \ref{def:AdmWithTraces}(ii)
in Section \ref{sec:ModelProblem}.

Proposition \ref{prop:SenseOfG,G*}(i) in Section \ref{sec:ModelProblem}
shows that the elementary solutions on which the admissibility criterion is
based actually belong to $S_t^{\mathcal{G}}$, i.e., they are
admissible\footnote{Notice that if $\mathcal{G}$ is not $L^1D$, then the
admissibility of some of its elementary solutions is contradictory: the pairs
$(u^l,u^r),(\hat u^l,\hat u^r)$ for which \eqref{eq:L1D-States} fails rule out each other!}.
From this viewpoint, our approach generalizes the original one 
of Kruzhkov by defining admissibility in terms of
the local $L^1$-dissipativity property \eqref{eq:L1Dissipativity}
with respect to a ``small'' set of solutions that are judged
admissible \textit{a priori}\footnote{A well-known  alternative
(see B\'enilan \cite{Benilan} and Crandall \cite{Crandall}) is
to define admissibility by the global $L^1$-contraction
property \eqref{eq:L1Contraction} with respect to a dense set
of solutions of the \textit{stationary problem} $ u+f( u)_x= s$; this is
the nonlinear semigroup approach to scalar conservation laws.
We will not develop this second possibility here.
From another point of view, however, the definition of admissibility
on the flux discontinuity line $\Set{x=0}$ by property \eqref{eq:L1D-States} is
a direct generalization of the admissibility approach of Vol'pert
\cite{Volpert}; see Section \ref{ssec:VolpertKruzhkov}.}. In the
Kruzhkov case, these are the constant solutions of $u_t+f(u)_x=0$
(cf.~Section \ref{ssec:VolpertKruzhkov}). In our context,
these are the elementary solutions of the form
\eqref{eq:PiecewiseC(x)} with $(u^l,u^r)\in\mathcal{G}$, and the
constant solutions in the domains $\R_+ \times (-\infty,0)$
and $\R_+\times (0,\infty)$.

Our perspective has the following practical implications:

First, there exist in the literature a number of different admissibility
criteria devised for particular forms of the fluxes $f^l,f^r$. For
some of them, the uniqueness and/or existence issues remained
open; and it is often difficult to judge whether the same criteria
can be useful for fluxes more general than those initially
considered.  On the other hand, if we manage to find some of the elementary
solutions allowed by a given admissibility criterion in a given
configuration of fluxes $f^l,f^r$, it becomes possible to
determine if the criterion is adequate. We give
examples in Section \ref{sec:Examples}.

Second, some important admissibility criteria derive from certain regularized problems. 
For example, for the classical vanishing viscosity method, which consists 
in adding $\eps\Delta u$ to the right-hand 
side of \eqref{eq:MultiDConsLaw} (cf.~Hopf \cite{Hopf} and
Kruzhkov \cite{Kruzhkov}),  a variant of the Kato inequality \eqref{eq:L1Dissipativity} is
well known for each $\eps>0$. Hence, the Kato inequality is easily passed on to
the limiting problem, which in turn gives raise to an implicitly
defined $L^1$-dissipative solver for the limit problem.
Determining the elementary solutions (as $\eps\to 0$ limits 
of standing-wave profiles, for example) opens up for the possibility of 
characterizing the germ of the associated $L^1$-dissipative solver, and
as a consequence describe it explicitly. This can eventually lead to
an existence result. Examples are given 
in Sections \ref{sec:SVVGerm} and \ref{sec:Existence}.
Our analysis utilizes several simple properties of 
germs listed in Sections \ref{sec:ModelProblem} and \ref{sec:Examples}.

Finally, we point out that while the formulation ($S_t^\mathcal{G}\!$-sol.)
(cf.~Definition \ref{def:AdmWithTraces}) is convenient for the uniqueness proof, it is rather
ill-suited for passaging to the limit in a sequence of approximate solutions.
Therefore we supply another formulation in terms of global entropy inequalities,
 inspired by the previous works of Baiti and Jenssen
\cite{BaitiJenssen}, Audusse and Perthame \cite{AudussePerthame},
B\"urger, Karlsen, and Towers \cite{BurgerKarlsenTowers}, and by the
founding papers of Kruzhkov \cite{Kruzhkov}, Otto \cite{Otto:BC},
and Carrillo \cite{Carrillo}. This formulation (cf.~Definition \ref{def:AdmIntegral}
for the precise notion) is also based upon a predefined $L^1D$ germ $\mathcal{G}$ 
associated with the flux discontinuity. Stability of this formulation 
with respect to $L^1_{\loc}$ convergence
of sequences of (approximate) solutions makes it convenient 
for the existence analysis. It is shown in Theorem \ref{th:DefLocal-Global} that the two
formulations are equivalent. We call the associated solutions
$\mathcal{G}$-entropy solutions.

\subsection{Measure-valued  $\mathcal{G}$-entropy solutions 
and convergence results}\label{ssec:EntropyProcessAndExistence}

The definition of $\mathcal{G}$-entropy solutions in terms
of global entropy inequalities can be adapted to define
measure-valued  $\mathcal{G}$-entropy solutions (more
precisely, we use the equivalent device of process solutions
developed in \cite{EGH,GallouetHubert}; the same idea 
appeared previously in \cite{Panov-process-sol}). 
Although we are not able to conceive a direct uniqueness proof for 
$\mathcal{G}$-entropy process solutions, in Section \ref{sec:ModelProblem}
we prove their uniqueness if we already know the existence of
$\mathcal{G}$-entropy solutions. A result like this makes it 
possible to prove strong convergence of merely bounded 
sequences of approximate $\mathcal{G}$-entropy solutions.
Heuristically, on the condition that existence can be established by some 
approximation method enjoying strong compactness
properties (cf.~Section \ref{ssec:Numerics}), we can 
deduce strong convergence of other approximation methods with 
mere $L^\infty$ estimates. We give the details in Section \ref{ssec:Numerics-bis}.

\subsection{Outline of the paper}\label{ssec:Outline}

In Section \ref{sec:Preliminaries}, we recall some parts of
the Kruzhkov theory; we also state and slightly
reformulate some results by Panov on initial and boundary traces 
of Kruzhkov entropy solutions. In Section \ref{sec:ModelProblem}, we 
define admissibility germs and related notions, provide 
several definitions of  $\mathcal{G}$-entropy solutions, discuss 
the relations between them, provide uniqueness, $L^1$-contraction, 
and comparison results, and $L^\infty$ estimates.
In Section \ref{sec:Examples}, we derive a series of important
properties of germs. Moreover, we classify several known
admissibility criteria in terms of their underlying germs, and 
determine their areas of applicability.
An important example is the ``vanishing viscosity'' admissibility
condition. In Section \ref{sec:SVVGerm}, we characterize its germ
for general flux functions $f^{l,r}$.  In Section \ref{sec:Existence}, we
justify the convergence of the standard vanishing viscosity approximations
to the ``vanishing viscosity germ'' entropy solutions obtained in
Section \ref{sec:SVVGerm}. Furthermore, we discuss variants of the
vanishing viscosity approach that lead to different germs; in
particular, we give a simplified proof of existence for the
``$(A,B)$-connection'' admissibility criteria studied by Adimurthi,
Mishra, Veerappa Gowda in \cite{AdimurthiMishraGowda} and by
B\"urger, Karlsen and Towers in \cite{BurgerKarlsenTowers}. Then we give an
existence result for $\mathcal{G}$-entropy solutions corresponding to a general germ. 
The proof is a by-product of  a strong convergence result for a particular 
germ-preserving finite volume scheme. Finally, we discuss convergence of 
$L^\infty$ bounded sequences of approximate 
solutions towards $\mathcal{G}$-entropy process solutions.

\section{Preliminaries}\label{sec:Preliminaries}

First, we recall a localized version of the 
definition due to Kruzhkov \cite{Kruzhkov}.
Let $\Omega\subset \R_+\times \R$ be open and set 
$\partial^0 \Omega:=\partial \Omega\cap \Set{t=0}$.
For $f\in C(U;\R;)$, a function $u\in L^\infty(\Omega;U)$
satisfying the inequalities
\begin{align*}
    & \forall k\in U \;\;\; \forall \xi\in \mathcal{D}(\R_+ \times \R),\; \; \xi \geq 0,
    \;\; \xi|_{\partial\Omega\setminus \partial^0\Omega}=0,\\
    & \iint_\Omega \Bigl\{|u-k| \xi_t+q(u,k) \xi_x \Bigr\}
    -\iint_\Omega \sign(u-k)\,s\xi  \geq  \int_{\partial^0\Omega} |u_0-k|\,\xi,
\end{align*}
is called Kruzhkov entropy solution of the conservation law
\begin{equation}\label{eq:ModelProb-local}
    u_t + f(u)_x=s \quad
    \text{in $\Omega$},
    \quad u|_{\partial^0 \Omega}=u_0.
\end{equation}
Here for $k\in \R$, the function $z\mapsto |z-k|$ is called
a Kruzhkov entropy, and
$$
q(z,k):=\sign(u-k)(f(u)-f(k))
$$
is the associated entropy flux.

An important ingredient of our formulation is
the fact that entropy solutions admit strong traces, in an appropriate sense.
For example, when $\Omega=\R_+\times \R$,
an entropy solution of \eqref{eq:ModelProb-local} admits $u_0$ as
the initial trace in the sense
\begin{equation}\label{eq:InitStrongTraces}
    \forall \xi\in \mathcal{D}(\R)
    \;\;\;
    \lim_{h\to 0^+} \frac 1h \iint_{(0,h) \times \R} \xi(x)\; |u(t,x)-u_0(x)|=0.
\end{equation}
(see Panov \cite{Panov-initial-trace} and the previous
works of Chen, Rascle \cite{ChenRascle} and of
Vasseur \cite{Vasseur}). Thus $t=0$ can be
seen as a Lebesgue point of the map
$t\mapsto u(t,\cdot)\in L^1_{\loc}(\Omega)$.

Most importantly for the present paper, analogous strong trace
results hold for entropy solutions in half-space domains.
For example, with $\Omega=\R_+\times(-\infty,0)$, provided
the non-degeneracy assumption
\begin{equation}\label{eq:non-deg-for-trace}
    \text{for any non-empty interval $(a,b)\subset U$,\;
    $f|_{(a,b)}$ is not constant}
\end{equation}
holds, there exists a measurable function 
$t\mapsto (\gamma^l u)(t)$ on $\R_+$ such that
\begin{equation}\label{eq:BoundStrongTraces}
    \forall \xi\in \mathcal{D}(0,\infty)
    \;\;\;
    \lim_{h\to 0^+} \frac 1h \iint_{\R_+ \times (-h,0)}
    \xi(t)\; \abs{u(t,x)-(\gamma^l u)(t)}=0.
\end{equation}
This result is a particular case of 
Panov \cite[Theorem 1.4]{Panov-boundary-trace} (see also
\cite{Vasseur,KwonVasseur}). Observe that we require from the flux  $f$ only that it is
continuous and satisfies the non-degeneracy condition \eqref{eq:non-deg-for-trace}.
The function $\gamma^l u$ is the strong left trace
of the entropy solution $u$ on $\Set{x=0}$.

We  emphasize that \eqref{eq:BoundStrongTraces} is a
sufficient --- but not a necessary --- property for our purposes.
Actually, strong traces of $f(u)$ and of the entropy fluxes $q(u,k)$ exist
even in the absence of the non-degeneracy assumption \eqref{eq:non-deg-for-trace}.
To state the result, we consider a ``singular mapping'' $\Psi:U\mapsto \R$.
There are different choices of the singular mappings;
in particular, whenever $f$ is of  bounded variation it is convenient to define
$\Psi(z)=\int_0^z |df(s)|$ (cf.~Temple \cite{Temple1982},
Klingenberg and Risebro \cite{KlingenbergRisebro}, and
also \cite{Towers2001,BachmannVovelle,AdimurthiMishraGowda}).
Let us fix the choice
\begin{equation}\label{EqSingMappingFunctions}
    V(z):=\int_0^z \char_{E}(s)\,ds,
\end{equation}
where $E$ is the maximal subset of $U$ such that $f(\cdot)$
is non-constant on any non-degenerate interval contained in $E$.
Clearly, under the non-degeneracy assumption \eqref{eq:non-deg-for-trace}, $V$
becomes the identity mapping.  Now the
result of \cite[Theorem 1.4]{Panov-boundary-trace} can
be restated in the following way.

\begin{theo}\label{TheoPanovNormalTraces}
Let $\mathfrak{f}:U\to\R$ be continuous. Assume $u$ is a
Kruzhkov entropy solution of $u_t+{f}(u)_x=0$ in
$\Omega:=\R_+ \times (-\infty,0)$.
Then $V(u)$ admits a strong
left trace on $\Set{x=0}$. Namely, there exists
a function $t\mapsto (\gamma^l V(u))(t)$ such that
\begin{equation}\label{eq:DefStrongTraces-of-V}
  \forall  \xi\in {\mathcal{D}}(0,\infty), \quad
   \lim_{h\to 0^+}
   \frac 1h \iint_{\R_+\times (-h,0)} \xi(t)
   \abs{V(u)(t,x)-(\gamma^l V(u))(t)}=0.
\end{equation}
Similar statements hold with
$\Omega=\R_+\times(0,\infty) $ and
the right trace operator $\gamma^r$.
\end{theo}

Now let us turn to our problem \eqref{eq:ModelProb}.
Whenever $u$ is a Kruzhkov entropy solution of
\eqref{eq:ModelProb} in the domains $\{x<0\}$ and $\{x>0\}$,
by Theorem \ref{TheoPanovNormalTraces} there exist strong one-sided traces
$\gamma^lV^l(u)$, $\gamma^rV^r(u)$.
Here the singular mappings $V^{l,r}$ are defined
by \eqref{EqSingMappingFunctions} from $f=f^{l,r}$, respectively.

In the two subsequent remarks, we indicate how to use the strong traces
of $V^{l,r}(u)$ to express the strong traces of the fluxes $f^{l,r}(u)$
and of the Kruzhkov entropy fluxes $q^{l,r}(u,k)$ (cf.~\cite{AndrSbihiCRAS}).

\begin{rem}\label{RemFunctionQ}
Define the monotone multivalued functions
$[V^l]^{-1},[V^l]^{-1}$. Notice that the superpositions
$g^{l,r}:=f^{l,r}\circ [V^{l,r}]^{-1}$ are continuous functions.
Moreover, the entropy flux functions $q^{l,r}$,
naturally written in terms of the unknown $z$ and of a parameter $k$,
can be expressed as continuous functions of
$V^{l,r}(z),V^{l,r}(k)$ only. More precisely, we have
\begin{equation*}%\label{EqQreplacesq}
    \begin{split}
        q^l(z,k) & = \sign(z-k)\,(f^l(z)-f^l(k)) \\
        & \equiv  \sign(V^l(z)-V^l(k))\;(g^l\circ V^l(z)-g^l\circ V^l(k)) \\
        & =: Q^l(V^l(z),V^l(k));
    \end{split}
\end{equation*}
a corresponding representation of $q^r(\cdot,\cdot)$
by $Q^r(V^r(\cdot),V^r(\cdot))$ is valid.
\end{rem}

\begin{rem}\label{RemTracesForEntropyfluxes}
It is easily seen that the traces in the sense of \eqref{eq:DefStrongTraces-of-V}
can be composed by any continuous function; namely,
$\gamma^{l,r}(h(V(u))=h(\gamma^{l,r} V(u))$ for $h\in C(U;\R)$.
It follows that in the context of Theorem \ref{TheoPanovNormalTraces}, there
exist strong traces $\gamma^{l,r} q^{l,r}(u,k)$ of the entropy fluxes $q^{l,r}(u,k)$,
and these traces are equal to $Q^{l,r}(\gamma^{l,r}V^{l,r}(u),V^{l,r}(k))$.
This remains true if we replace $k$ by
another entropy solution $\hat u$. In this case,
$$
\gamma^{l,r} q^{l,r}(u,\hat u)
=Q^{l,r}(\gamma^{l,r}V^{l,r}(u),\gamma^{l,r}V^{l,r}(\hat u)).
$$
Recall that whenever $f^{l,r}$ are non-degenerate in the
sense of \eqref{eq:non-deg-for-trace},  $V^{l,r}$
are just the identity mappings and $Q^{l,r}$ coincide with $q^{l,r}$.
Finally, notice that the Rankine-Hugoniot relation for a
weak solution $u$ of \eqref{eq:ModelProb} can be expressed under the
form $g^l(\gamma^lV^l(u))=g^r(\gamma^rV^r(u))$.
\end{rem}

\section{The model one-dimensional problem}\label{sec:ModelProblem}

Consider problem \eqref{eq:ModelProb}, \eqref{eq:InitialCond}. Here
the flux $\mathfrak{f}$ is constant in $x$ on each side of the
discontinuity line $\Sigma:=(0,{\infty}) \times \Set{0}$, on which the
admissibility condition will be defined. As in
Section \ref{sec:Preliminaries}, we denote the Kruzhkov entropy flux by
\begin{equation}\label{eq:EntropyFlux1}
    \mathfrak{q}(x,z,k):=
    \sign(z-k)(\mathfrak{f}(x,z)-\mathfrak{f}(x,k)).
\end{equation}
We write $q^{l,r}(z,k)$ for the left and right entropy fluxes
$\sign(z-k)(f^{l,r}(z)-f^{l,r}(k))$. Whenever it is convenient,
we represent $f^{l,r},q^{l,r}$ by means of the continuous
functions $g^{l,r}$, $Q^{l,r}$ and the singular mappings $V^{l,r}$
introduced in Remark \ref{RemFunctionQ}:
\begin{equation}\label{eq:q-to-Q}
    f^{l,r}(z)=g^{l,r}(V^{l,r}z), \qquad
    q^{l,r}(z,k)=Q^{l,r}(V^{l,r}z,V^{l,r}k),
\end{equation}
where to
simplify the notation we write $V^l z$ instead of
$V^l(z)$, et cetera.

\subsection{Definitions of germs and their basic
  properties}\label{ssec:Germs-defs} 

Related to the left and right fluxes $f^l$ and $f^r$,
we introduce the following definitions.

\begin{defi}\label{def:Germ}
Any set $\mathcal{G}$ of pairs $(u^l,u^r)\in U \times U$
satisfying the Rankine-Hugoniot relation \eqref{eq:RankineHugoniot}
is called an admissibility germ (a germ for short).
If, in addition, \eqref{eq:L1D-States} holds
for all $(u^l,u^r),(\hat u^l,\hat u^r)\in \mathcal{G}$, then
the germ $\mathcal{G}$ is called an $L^1\!$-dissipative
admissibility germ (an $L^1D$ germ for short).
\end{defi}

In the sequel, we focus on $L^1D$ germs, which are those leading
to the $L^1$-dissipativity properties
\eqref{eq:L1Dissipativity}, \eqref{eq:L1Contraction}.
In this case, each pair $(u^l,u^r)\in\mathcal{G}$ corresponds to a
solution of \eqref{eq:ModelProb} of the form \eqref{eq:PiecewiseC(x)} which
will be judged admissible a priori
(see also Proposition \ref{prop:SenseOfG,G*}(i) below).

According to the analysis carried out in the
introduction, if $(u^l,u^r)\in \mathcal{G}$ is judged to be an admissible
elementary solution on \eqref{eq:ModelProb}, then inequality
\eqref{eq:L1D-States} together with the Rankine-Hugoniot relation
should hold for any other admissible
elementary solution $(\hat u^l,\hat u^r)$ of \eqref{eq:ModelProb}.
Therefore we introduce the following definition:

\begin{defi}\label{def:DualGerm}
Let $\mathcal{G}$ be a germ. The dual germ of $\mathcal{G}$, denoted
by $\mathcal{G}^*$, is the set of pairs $(\hat u^l,\hat u^r)\in U \times U$
such that \eqref{eq:L1D-States} holds for all $(u^l,u^r)\in \mathcal{G}$, and the
Rankine-Hugoniot relation $f^l(\hat u^l)=f^r(\hat u^r)$ is satisfied.
\end{defi}

Hence, each elementary solution of \eqref{eq:ModelProb},
which is expected to be admissible, corresponds
to a pair $(\hat u^l,\hat u^r)\in\mathcal{G}^*$
(cf.~Proposition \ref{prop:SenseOfG,G*}(ii) below).

If $\mathcal{G}_1, \mathcal{G}_2$ are two germs such that
$\mathcal{G}_1\subset \mathcal{G}_2$, we say that $\mathcal{G}_2$ is
an extension of $\mathcal{G}_1$. If both $\mathcal{G}_1,\mathcal
G_2$ are $L^1D$ germs, we call $\mathcal{G}_2$ an
$L^1D$ extension of $\mathcal{G}_1$.

\begin{defi}\label{def:Maximal-DefiniteL1DGerm}
If $\mathcal{G}$ is an $L^1D$ germ which
 does not possess a
nontrivial $L^1D$ extension, then $\mathcal{G}$ is called a
maximal $L^1D$ germ. If $\mathcal{G}$ is a germ that possesses a
unique maximal $L^1D$ extension, then $\mathcal{G}$ is called a
definite  germ.
\end{defi}

Notice that any maximal $L^1D$ germ is definite.
As the following proposition shows, the definiteness
of $\mathcal{G}$ is necessary and sufficient for
its dual germ $\mathcal{G}^*$ to be an $L^1D$ germ.

\begin{prop}\label{prop:DualityOfGerms}
Fix a germ $\mathcal{G}$, and let $\mathcal{G}^*$ be the dual germ of $\mathcal{G}$.\\
(i) One has $\mathcal{G}\subset \mathcal{G}^*$ if
and only if  $\mathcal{G}$ is an $L^1D$ germ.\\
(ii) Assume $\mathcal{G}$ is an $L^1D$ germ.
Then $\mathcal{G}^*$ is the union of all $L^1D$ extensions of $\mathcal{G}$.
Specifically, $\mathcal{G}$ is a definite  germ implies
$\mathcal{G}^*$ is $\mathcal{G}$'s unique maximal $L^1D$ extension.\\
(iii) One has $\mathcal{G}^*=\mathcal{G}$ if and only
if  $\mathcal{G}$ is a maximal $L^1D$ germ.\\
(iv) If $\mathcal{G}$ is a definite germ, then $(\mathcal{G}^*)^*=\mathcal{G}^*$.\\
 (v) If $\mathcal{G}^*$ is an $L^1D$ germ, then $\mathcal{G}$ is definite.
\end{prop}

\begin{proof}
Property  (i) follows directly from the definitions.
For a proof of (ii), let $\mathcal{G}'$ be an $L^1D$ extension of
$\mathcal{G}$. Clearly, $\mathcal{G}'\subset \mathcal{G}^*$.
Reciprocally, let $(\hat u^l,\hat u^r)\in \mathcal{G}^*$. Then
 $\mathcal{G}':=\mathcal{G}\cup \{(\hat u^l,\hat u^r)\}$ is an $L^1D$
extension of $\mathcal{G}$ which contains $(\hat u^l,\hat u^r)$.
The second part of the assertion (ii) follows immediately.
Property (iii) follows from (ii). Properties (ii) and (iii) imply (iv).
For a proof of (v), we reason by contradiction. Let $\mathcal{G}_1$, $\mathcal{G}_2$ be two
different maximal $L^1D$ extensions of $\mathcal{G}$. Then
$\mathcal{G}_1\cup \mathcal{G}_2$ is not $L^1D$. By (ii), $\mathcal
G^*$ contains $\mathcal{G}_1\cup \mathcal{G}_2$ and therefore it is
not an $L^1D$ germ either.
\end{proof}

Let $V^l$, $V^r$ be the singular mappings introduced in
Section \ref{sec:Preliminaries}. By $\ran(V^{l,r})$  we mean the images
of $U$ by the functions $V^{l,r}$, respectively. It is clear from
\eqref{eq:q-to-Q} that the validity of the inequality
\eqref{eq:L1D-States} and of the Rankine-Hugoniot condition only
depends on $u^l,\hat u^l, u^r,\hat u^r$ through the corresponding
values $V^lu^l,V^l\hat u^l, V^ru^r,V^r\hat u^r$. This motivates
the following definition.

\begin{defi}\label{def:ReducedGerm}
Let $\mathcal{G}$ be a germ. The reduced germ of $\mathcal{G}$  is
the set
$$
V\mathcal{G}:= \Set{\left( V^lu^l,V^ru^r\right)\, \Big |\,(u^l,u^r)\in \mathcal{G}}.
$$
\end{defi}

\begin{rem}\label{rem:ReducedGermProp}
The following properties are easily derived:

\noindent (i) If $\mathcal{G}$ is a maximal $L^1D$ germ, then
$\mathcal{G}=\Set{(u^l,u^r)\, \Big |\, \left(V^lu^l,V^ru^r\right)\in V\mathcal{G}}$.

\noindent (ii) If $\mathcal{G}^*$ is the dual of a germ $\mathcal{G}$, then
$\mathcal{G}^*=\Set{(u^l,u^r)\, \Big|\, \left(V^lu^l,V^ru^r\right)\in V\mathcal{G}^*}$.
\end{rem}

\begin{rem}\label{rem:L1DGerm-Q}
In accordance with \eqref{eq:q-to-Q}, $V\mathcal{G}^*$ can be
equivalently defined as the set of pairs
$(\hat v^l,\hat v^r)\in \ran(V^l)\times \ran(V^r)$
such that the Rankine-Hugoniot
condition $g^l(\hat v^l)=g^r(\hat v^r)$ holds, and
\begin{equation}\label{eq:L1DStates-Q}
    Q^l(v^l,\hat v^l) \geq Q^r(v^r,\hat v^r),
    \quad
    \text{whenever $(v^l,v^r)\in V\mathcal{G}$}.
\end{equation}
Similarly, if $\mathcal{G}^*$ is an $L^1D$ germ, then
for all $(v^l,v^r),(\hat v^l,\hat v^r)\in V\mathcal{G}^*$
inequality \eqref{eq:L1DStates-Q} is satisfied. Furthermore, if
$\mathcal{G}^*$ is a maximal $L^1D$ germ, then $V\mathcal{G}^*$
does not possess a nontrivial extension $\mathcal V'\subset \ran(V^l)\times
\ran(V^r)$ such that \eqref{eq:L1DStates-Q} still holds for all
$(v^l,v^r),(\hat v^l,\hat v^r)\in \mathcal V'$.
\end{rem}

In the sequel, the reader may assume that $f^{l,r}$ are not
constant on any nontrivial interval, so that in all subsequent
statements we have $V^{l,r}=\ident$, $Q^{l,r}\equiv q^{l,r}$, and the
reduced germs $V\mathcal{G}$, $V\mathcal{G}^*$ can
be replaced with $\mathcal{G}$, $\mathcal{G}^*$, respectively.

\subsection{Definitions and uniqueness of
$\mathbf{\mathcal{G}}$-entropy solutions}\label{ssec:BasicTheory-Defs-and-Uniq.}

We are now in a position to define $\mathcal{G}$-entropy solutions of
\eqref{eq:ModelProb},\eqref{eq:InitialCond} and study their uniqueness.

\begin{defi}\label{def:AdmWithTraces}
Let  $\mathcal{G}$ be an $L^1D$ germ, with dual germ $\mathcal{G}^*$.
A function $u(t,x)$ in $L^\infty(\R_+\times\R;U)$
is called a  $\mathcal{G}$-entropy
solution of \eqref{eq:ModelProb},\eqref{eq:InitialCond} if:\\

\noindent (i) the restriction of $u$ to $\Omega^l:=\R_+\times (-\infty,0)$
is a Kruzhkov entropy solution of the conservation law with
flux $f^l(\cdot)$; the restriction of $u$ to $\Omega^r:=\R_+\times (0,\infty)$ is a Kruzhkov entropy
solution of the conservation law with flux $f^r(\cdot)$;\\

\noindent (ii) $\mathcal H^{1}$-a.e.~ on
$\Sigma=(0,\infty) \times \Set{0}$, the pair of strong traces
$\left( \gamma^l V^l(u),\gamma^rV^r(u)\right)$
on $\Sigma$ belongs to the reduced
germ $V\mathcal{G}^*$ of $\mathcal{G}^*$;\\

\noindent (iii) $\mathcal H^{1}$-a.e.~on $\Set{0}\times \R$,
the trace $\gamma^0 u$ equals $u_0$.
\end{defi}

The above definition fits the expectations of the preliminary analysis carried
out in Section \ref{ssec:RiemannAnalysis}.
 Indeed, notice the following points that we list in a remark.

\begin{rem}\label{rem:Def1-remarks}~

\noindent (i) The existence of strong $L^1$ traces
$\gamma^l V^lu$, $\gamma^rV^ru$, $\gamma^0 u$ in
Definition \ref{def:AdmWithTraces}(ii) is not a restriction: due
to Theorem \ref{TheoPanovNormalTraces} and \eqref{eq:InitStrongTraces}, it
follows from the point (i) of the same definition.
Moreover, Definition \ref{def:AdmWithTraces}(i)
and the result of \cite{Panov-initial-trace} imply that $u$ has a representative
in $C([0,\infty);L^1_{\loc}(\R))$. In the sequel, we
mean that a solution $u$ is defined for all $t$ as an $L^\infty(\R)$ function.

\noindent (ii) A $\mathcal{G}$-entropy solution of
\eqref{eq:ModelProb},\eqref{eq:InitialCond} is a weak (distributional) solution.

\noindent (iii) A  $\mathcal{G}$-entropy solution possesses both the scaling
invariance property $\text{\rm(A1)}$ and property
$\text{\rm(A2)}$ required in Section \ref{ssec:RiemannAnalysis}.

\noindent (iv) In view of Remark \ref{rem:ReducedGermProp}(ii), any
function of the form {\rm (RPb-sol.)}~is a  $\mathcal{G}$-entropy
solution of the Riemann problem \eqref{eq:ModelProb},
\eqref{eq:RiemannInitialCond} if and only if $(u^l,u^r)\in\mathcal{G}^*$.

\noindent (v) One should compare Definition \ref{def:AdmWithTraces}(ii)
with the definition of admissibility of jumps of $BV$ solutions
introduced by Vol'pert in \cite{Volpert}, in the special
case $f^l=f^r$ (see inequality \eqref{eq:VolpertIneq} in
Section \ref{ssec:VolpertKruzhkov}).
\end{rem}

For a proof of Remark \ref{rem:Def1-remarks}(ii), we notice that
Kruzhkov entropy solutions are weak
solutions on their domains; thus $u$ is a weak
solution on $(0,\infty)\times (\R\setminus\Set{0})$. In addition, the
Rankine-Hugoniot relation on $\Sigma=\R_+ \times \Set{0}$ holds, because we have
$\Bigl(\,\gamma^lV^lu\,,\,\gamma^rV^ru\,\Bigr)\in V\mathcal{G}^*$,
and Definitions \ref{def:DualGerm}, \ref{def:ReducedGerm} of $V\mathcal{G}^*$ together with
Remarks \ref{RemFunctionQ}, \ref{RemTracesForEntropyfluxes} imply that
$$
\gamma^{l,r}f^{l,r}(u)=\gamma^{l,r}g^{l,r}(V^{l,r}u)
=g^{l,r}(\gamma^{l,r}V^{l,r}u)
=g^{l,r}(V^{l,r}w^{l,r})=f^{l,r}(w^{l,r}),
$$
for some $(w^l,w^r)\in \mathcal{G}^*$; in particular, $f^l(w^l)=f^r(w^r)$.

From Remark \ref{rem:Def1-remarks}(iv) and
Proposition \ref{prop:DualityOfGerms}(i), we directly deduce
that the germ $\mathcal{G}$ (and, more generally, $\mathcal{G}^*$)
indeed corresponds to a selection of admissible elementary
solutions \eqref{eq:PiecewiseC(x)}, as we point
out in the succeeding proposition.

\begin{prop}\label{prop:SenseOfG,G*}~

\noindent (i) If $(u^l,u^r)\in \mathcal{G}$, then the function
$u:=u^l \char_{\Set{x<0}}+u^r\char_{\Set{x>0}}$  is a $\mathcal{G}$-entropy
solution of \eqref{eq:ModelProb}.

\noindent (ii) A function $u(x)$ of the above form is a
$\mathcal{G}$-entropy solution of \eqref{eq:ModelProb} if and only
if $(u^l,u^r)\in \mathcal{G}^*$.
\end{prop}

Consider the  map $S_t^\mathcal{G}$ which associates to
$u_0\in L^\infty(U)$ the value $u(t,\cdot)$ on the
trajectory of some $\mathcal{G}$-entropy solution $u$ defined globally in time.
 In case such a solution exists, we state that $u_0$
 belongs to $D$, the domain of $S^\mathcal{G}_t$.
 Provided the underlying germ $\mathcal{G}$ is definite, the
 following theorem shows that such a solution $u$ is necessarily unique;
one then deduces that the domain of $S^\mathcal{G}_t$ is
independent of $t$, that $(S^\mathcal{G}_t)_{t\geq 0}$ is a semigroup,
and it is strongly continuous in $L^1_{\loc}(\R)$,
according to Definition \ref{def:AdmWithTraces}(iii).
In fact, the theorem shows that the
map $S^\mathcal{G}_t$ is an $L^1$-dissipative
solver in the sense of Definition \ref{def:L1-solver}.

\begin{theo}\label{theo:LocalESUniqueness}
Assume that $\mathcal{G}$ is a definite  germ.
If $u$ and $\hat u$ are two $\mathcal{G}$-entropy solutions of
problem \eqref{eq:ModelProb},\eqref{eq:InitialCond} corresponding
to the initial data $u_0$ and $\hat u_0$ respectively, then
the Kato inequality \eqref{eq:L1Dissipativity} and the
$L^1$-contractivity property \eqref{eq:L1Contraction} hold.
In particular, there exists at most one $\mathcal{G}$-entropy
solution of problem \eqref{eq:ModelProb},\eqref{eq:InitialCond}.
\end{theo}

\begin{proof}
We only prove the Kato inequality; the $L^1$-contractivity \eqref{eq:L1Contraction}
and the uniqueness will follow by considering in \eqref{eq:L1Dissipativity} the test
functions $\xi_R:=\min\{1,(R-|x|)^+\}$ with $R\to \infty$, using
the continuity of $f^l,f^r$ and the fact that the space of $x$ has
dimension one (see \cite{Benilan,KruzhkovHildebrand}).

Take $\xi\in {\mathcal{D}}([0,\infty)\times \R)$, $\xi\geq 0$. By a
standard approximation argument, for $h>0$ we can take the test
function $\xi_h=\xi\,\min\{1,\frac{(|x|-h)^+}{h}\}$ in the
Kruzhkov entropy formulation, in each of the subdomains
$\Omega^{l,r}$. By the ``doubling-of-variables"
argument of Kruzhkov \cite{Kruzhkov},  we
obtain the standard Kato inequality
$$
-\int_{\R_+} \int_{\R} \Bigl\{\abs{u-\hat u} (\xi_h)_t
+\mathfrak{q}(x,u,\hat u)(\xi_h)_x \Bigr\}
-  \int_{\R} \abs{u_{0}-\hat u_{0}} \xi_h(0,x)  \leq 0.
$$
Clearly, $\xi_h$, $(\xi_h)_t$ converge to $\xi$,
$\xi_t$, respectively, in $L^1(\R_+\times\R)$.
 Using Remark \ref{RemFunctionQ} and calculating $(\xi_h)_x$ explicitly,
 with the Landau notation  $\overline{\overline{o}}_{h\to 0}$, we deduce
\begin{equation}\label{eq:AfterDoubling}
    \begin{split}
        & -\int_{\R_+} \int_{\R}
        \Bigl\{\abs{u-\hat u} \xi_t+\mathfrak{q}(x,u,\hat u) \xi_x \Bigr\}
        - \int_{\R} \abs{u_{0}-\hat u_{0}}\xi(0,x)
        +\overline{\overline{o}}_{h\to 0}(1)\\
        & \qquad\quad
        +\frac{1}{h}\iint_{\R_+\times (-2h,-h)}
        Q^l(V^l(u),V^l(\hat u)) \xi
        \\ & \qquad\qquad\qquad
        -\frac{1}{h}\iint_{\R_+\times (h,2h)}
        Q^r(V^r(u),V^r(\hat u))  \xi  \leq 0.
    \end{split}
\end{equation}
Sending $h\to 0$ in the latter two
terms, keeping in mind the definition of
strong traces of $V^{l,r}(u),V^{l,r}(\hat u)$ and
bringing into service the continuity
of $Q^{l,r}$ (cf.~\eqref{eq:DefStrongTraces-of-V}
and Remark \ref{RemTracesForEntropyfluxes}), we obtain
\begin{equation}\label{eq:IneqProofTh2}
    \int_{\R_+}
    \left( Q^l(v^l(t),\hat v^l(t))-Q^r(v^r(t),\hat v^r(t))\right) \xi(t,0),
\end{equation}
where $v^l=\gamma^l V^lu^l$, $\hat v^l=\gamma^l V^l\hat u^l$, $v^r=\gamma^r V^ru^r$,
and $\hat v^r = \gamma^r V^r\hat u^r$ in the pointwise sense for a.e.~$t>0$.
By Definition \ref{def:AdmWithTraces}(ii), we have $(v^l(t),v^r(t)),
(\hat v^l(t),\hat v^r(t))\in V\mathcal{G}^*$.  Now, because $\mathcal{G}$
is assumed to be definite, by  Proposition \ref{prop:DualityOfGerms} it
follows that $\mathcal{G}^*$ is an $L^1D$ germ. By Remark \ref{rem:L1DGerm-Q}
we conclude that  the term in \eqref{eq:IneqProofTh2} is nonnegative.
Therefore the Kato inequality \eqref{eq:L1Dissipativity}
follows from \eqref{eq:AfterDoubling} at the limit $h\to 0$.
\end{proof}

We collect a few comments in the next two remarks.

\begin{rem}\label{rem:GandG*}~%%%{rem:L1Disneeded}

\noindent (i) The assumption that $\mathcal{G}$ is a definite  germ cannot be
omitted from the statement of Theorem \ref{theo:LocalESUniqueness}, in view of
Proposition \ref{prop:DualityOfGerms}(v). Indeed, in the
above proof, the fact that $\mathcal{G}^*$ is an $L^1D$ germ is crucial.
An example of non-uniqueness  for an $L^1D$ germ
which is not definite is given in Section \ref{ssec:KRT-condition}
(Example \ref{example:KRT-non-uniq}).

\noindent (ii) Let $\mathcal{G}$ be a definite germ.
In view of Proposition \ref{prop:DualityOfGerms}(iv) and Definition \ref{def:AdmWithTraces},
it follows that $\mathcal{G}$- and $\mathcal{G}^*$-entropy solutions coincide.
Certainly, one can choose to work exclusively with maximal
$L^1D$ germs (in which case $\mathcal{G}^*\equiv \mathcal{G}$).
Yet there can be an advantage in using definite
germs $\mathcal{G}$ that are smaller than the
corresponding maximal $L^1D$ extension $\mathcal{G}^*$; see in
particular  Sections \ref{sec:Examples}, \ref{sec:SVVGerm},
\ref{ssec:VV-in-1D}, and \ref{ssec:VV-for-connections}.
\end{rem}

\begin{rem}\label{rem:RPb-G-solutions}
Let $\mathcal{G}$ be a definite  germ, and assume
$(u_-,u_+)\in U \times U$ is such that there exists a
$\mathcal{G}$-entropy solution $u$ of the corresponding Riemann problem
\eqref{eq:ModelProb},\eqref{eq:RiemannInitialCond}. Then $u$ is of
the form {\rm (RPb-sol.)}~with $(u^l,u^r)\in \mathcal{G}^*$. In
particular, $u$ admits strong left and right traces on
$(0,{\infty}) \times \Set{0}$ which are equal to $u^l$ and $u^r$, respectively.
Indeed, the scaling invariance $\text{\rm(A1)}$ (see
Remark \ref{rem:Def1-remarks}(iii)) and the uniqueness statement in
Theorem \ref{theo:LocalESUniqueness} imply that $u$ is
self-similar; by Definition \ref{def:AdmWithTraces}(i), it follows
that $u$ is monotone in the variable $x/t$ on $(-\infty,0)$ and on $(0,\infty)$.
Hence the traces $u^{l,r}:=\gamma^{l,r}u$ exist, and
by Remark \ref{rem:Def1-remarks}(iv), the solution $u$ is of
the form {\rm (RPb-sol.)}~with $(u^l,u^r)\in \mathcal{G}^*$.
\end{rem}

We want next to provide an equivalent definition of solution in which the trace condition of
Definition \ref{def:AdmWithTraces}(ii) is incorporated into the
``global'' entropy inequalities.  As a preparational step, we state
the following elementary lemma.

\begin{lem}\label{lem:remainder-term-forms}
Let $U$ be a closed interval in $\R$ and $f\in C(U;\R)$; for
$z,c\in U$, set $q(z,c)=\sign(z-c)\Bigl(f(z)-f(c)\Bigr)$. Then for
all $a,b,c\in U$,
\begin{equation}\label{eq:q-estim}
    \abs{q(a,c)-q(a,b)}\leq R(f(\cdot);c,b),
\end{equation}
where $R(f(\cdot);c,b)$ takes on one of the following forms:

$\bullet$ $R(f(\cdot);c,b)=2\Osc\Bigl(f(\cdot);c,b\Bigr)$, where
$\Osc$ denotes the oscillation of $f$ on the interval between $c$ and $b$:
\begin{equation}\label{eq:Osc-definition}
    \Osc\Bigl(f(\cdot);c,b\Bigr)
    := \max\left\{\abs{f(z)-f(s)}\, \Big |\,\min\{c,b\}\leq z\leq s \leq \max\{c,b\}\right\};
\end{equation}

$\bullet$ $R(f(\cdot);c,b):=2\,\omega(|b-c|)$, where $\omega$
is the modulus of continuity of $f$ on the interval between $c$ and $b$:
$$
\omega(h):=\max\left\{\abs{f(z)-f(s)}\, \Big |\,
\min\{c,b\}\leq z\leq s \leq \max\{c,b\},\, |z-s| \leq h,\right\};
$$

$\bullet$ provided $f\in BV_{\loc}(U)$,
$R(f(\cdot);c,b):=2 \abs{\int\nolimits_c^b \abs{f'(s)}\,ds}$ (this is the
variation of $f$ on the interval between $c$ and $b$);
\end{lem}

\begin{defi}\label{def:AdmIntegral} Let  $\mathcal{G}$ be an $L^1D$ germ.
A function $u\in L^\infty(\R_+\times\R;U)$ is called a
$\mathcal{G}$-entropy solution of \eqref{eq:ModelProb},\eqref{eq:InitialCond}
if it is a weak solution of this problem and
for all $(c^l,c^r)\in U \times U$,
\begin{equation}\label{eq:EntropySolDefi}
    \begin{split}
        & \int_{\R_+}\int_\R \Bigl\{\abs{u(t,x)-c(x)} \xi_t
        +\mathfrak{q}(x,u(t,x),c(x)) \xi_x \Bigr\}
        -\int_{\R} \abs{u_0(x)-c(x)} \xi(0,x) \\
        & \qquad\quad
        +\int_{\R_+} R_{\mathcal{G}}\Bigl((c^l,c^r)\Bigr) \xi(t,0)
        \geq  0, \quad
        \text{$\forall \xi\in {\mathcal{D}}([0,\infty)\times \R)$, $\xi \geq 0$,}
    \end{split}
\end{equation}
where $c(x)$ is the step function
\begin{equation*}%\label{eq:C(x)}
    c(x)=
    c^l\char_{\Set{x<0}}+c^r\char_{\Set{x>0}}=
    \begin{cases}
        c^l, & x<0, \\
        c^r, & x>0,\\
     \end{cases}
\end{equation*}
the entropy flux $\mathfrak{q}$ is defined by
\eqref{eq:EntropyFlux1}, and the
remainder term $R_{\mathcal{G}}$ is given by
\begin{equation*}%\label{eq:rem-oscillation}
    R_{\mathcal{G}}\Bigl((c^l,c^r)\Bigr):=2\inf_{(b^l,b^r)\in 
    \mathcal{G}}\:\Bigl( \Osc(f^l(\cdot);c^l,b^l)
    +\Osc(f^r(\cdot);c^r,b^r)\Bigr),
\end{equation*}
with the oscillation function $\Osc$ defined in \eqref{eq:Osc-definition}.
\end{defi}

\begin{rem}\label{rem:Def2-remarks}
The explicit requirement that $u$ should be a weak solution is only
needed to ensure that the Rankine-Hugoniot condition holds on
$\Sigma=(0,\infty) \times \Set{0}$. However, for many particular
choices of $\mathcal{G}$, a bounded function $u$ satisfying
\eqref{eq:EntropySolDefi} is automatically a
weak solution of \eqref{eq:ModelProb}, \eqref{eq:InitialCond}.
\end{rem}

Let us explain the choice of the penalization (remainder)
term $R_{\mathcal{G}}\Bigl((c^l,c^r)\Bigr)$
in Definition \ref{def:AdmIntegral}, which is inspired
by an idea of Otto \cite{Otto:BC}. The remainder
$R_{\mathcal{G}}\Bigl((c^l,c^r)\Bigr)$ is chosen to satisfy
the two following properties:
\begin{equation}\label{eq:prop1-remainder}
    (c^l,c^r)\in \mathcal{G}  \Rightarrow
    R_{\mathcal{G}}\Bigl((c^l,c^r)\Bigr)=0,
\end{equation}
and (because $\mathcal{G}$ is an $L^1D$ germ)
\begin{equation}\label{eq:prop2-remainder}
    (c^l,c^r) \in  U \times U, \,
    (a^l,a^r) \in \mathcal{G} \Rightarrow
    q^r(a^r,c^r) -q^l(a^l,c^l) \leq R_{\mathcal{G}}\Bigl((c^l,c^r)\Bigr).
\end{equation}
Property \eqref{eq:prop2-remainder} follows from
\eqref{eq:q-estim} in Lemma \ref{lem:remainder-term-forms}, via the bound
\begin{equation*}
    \begin{split}
        & \text{for all $\,(a^l,a^r),(c^l,c^r) \in U\times U$},\\
        & \inf\nolimits_{(b^l,b^r)\in \mathcal{G}}
        \abs{\left[q^l(a^l,b^l)-q^r(a^r,b^r)\right]-\left[q^l(a^l,c^l)-q^r(a^r,c^r)\right]}
        \leq R_{\mathcal{G}}\Bigl((c^l,c^r)\Bigr),
    \end{split}
\end{equation*}
and the inequality  $q^l(a^l,b^l)-q^r(a^r,b^r)\geq 0$, which is valid
for all $(a^l,a^r),(b^l,b^r) \in\mathcal{G}$.

According to Lemma \ref{lem:remainder-term-forms}, we can
take the remainder term $R_{\mathcal{G}}$ in
\eqref{eq:EntropySolDefi} under a few different forms, using, e.g., the
moduli of continuity of the functions $f^{l,r}$ or their variation functions
instead of the oscillation functions \eqref{eq:Osc-definition}.
If $f^{l,r}$ are globally Lipschitz continuous on
$U$, the simplest choice is to take
\begin{align*}
    R_{\mathcal{G}}\Bigl((c^l,c^r)\Bigr)
    & := 2\norm{(f^{l,r})'}_{L^\infty}
    \inf \Set{\abs{b^l-c^l}+\abs{b^r-c^r} \Big |\, (b^l,b^r) \in \mathcal{G}}
    \\  & \equiv C \dist \Bigl((c^l,c^r),\mathcal{G} \Bigr),
\end{align*}
where $\dist$ is the euclidean distance on $\R^2$ and $C$ is a
sufficienly large constant. With this choice, the properties
\eqref{eq:prop1-remainder} and \eqref{eq:prop2-remainder} remain
true. See  \cite{AndrGoatinSeguin} for one application.

The above Definition \ref{def:AdmIntegral} is suitable for many generalizations,
including the multi-D setting as in \cite{AKR-II,AKR-ParisNote}
and the case of time-dependent families of germs as in \cite{AndrGoatinSeguin}.
However, for the precise model case \eqref{eq:ModelProb}, one
can avoid the use of the remainder term $R_{\mathcal{G}}$ in
\eqref{eq:EntropySolDefi}, thanks to the following result.

\begin{prop}\label{prop:Carrillo-type-defs}
The inequalities in \eqref{eq:EntropySolDefi} hold for all
$(c^l,c^r)\in U\times U$ and for all nonnegative
test functions $\xi$, if and only if
they hold (with zero remainder term) for the choices
\begin{align*}
    &\Set{\text{$c^l=c^r=c$ with $c$ arbitrary and $\xi|_{x=0}=0$}}
    \\ & \quad
    \cup
    \Set{\text{$(c^l,c^r)\in \mathcal{G}$ and $\xi\geq 0$ is arbitrary}}.
\end{align*}
\end{prop}

\begin{proof}
According to \eqref{eq:prop1-remainder}, whenever $(c^l,c^r)\in \mathcal{G}$, the term
$\int_{\R_+} R_{\mathcal{G}}\Bigl((c^l,c^r)\Bigr) \xi(t,0)$
in \eqref{eq:EntropySolDefi} vanishes; clearly, it also vanishes if $\xi|_{x=0}=0$.
Conversely, following the proof of Theorem \ref{th:DefLocal-Global} below, using the
test functions $\xi_h^\pm$ in \eqref{eq:test-functions}
and property \eqref{eq:prop2-remainder}  of the
remainder $\mathcal R_{\mathcal{G}}$, it is easy  to establish
\eqref{eq:EntropySolDefi} for all choices of $(c^l,c^r)$, $\xi\ge0$.
 \end{proof}

Using Proposition \ref{prop:Carrillo-type-defs}, with specific choices of
$\mathcal{G}$, from Definition \ref{def:AdmIntegral} we recover the
formulations of Baiti-Jenssen \cite{BaitiJenssen},
Audusse-Perthame \cite{AudussePerthame} (restricted
to the model case \eqref{eq:ModelProb}), and B\"urger, Karlsen,
and Towers \cite{BurgerKarlsenTowers}; see
Section \ref{sec:Examples} for details.
Namely, for test functions $\xi$ which are zero on the
interface $\Set{x=0}$, the Kruzhkov entropy
inequalities with any entropy $\eta(z)=|z-c|$, $c\in \R$, are
required; and for general $\xi$, up-to-the-interface entropy
inequalities are required only for a careful selection of
``adapted'' entropies\footnote{These formulations are very
similar to the one introduced by Carrillo in \cite{Carrillo} for
homogeneous Dirichlet boundary value problems for degenerate
parabolic problems including conservation laws.
In our context, we can use the standard Kruzhkov
entropies $|z-c|$ in the place of the semi-Kruzhkov entropies
$(z-c)^\pm$ used by Carrillo.}.
The adapted entropies are $x$-dependent functions of the form
$$
\eta(z,c):=|z-c(x)|, \qquad
c(x)=c^l\,\char_{\Set{x<0}}+c^r\,\char_{\Set{x>0}},
\quad (c^l,c^r)\in \mathcal{G}.
$$

The notions of solution introduced in
Definitions \ref{def:AdmWithTraces} and \ref{def:AdmIntegral} bear
the same name. Indeed, we have

\begin{theo}\label{th:DefLocal-Global}
For any $L^1D$ germ $\mathcal{G}$, Definitions \ref{def:AdmWithTraces}
and \ref{def:AdmIntegral} are equivalent.
\end{theo}

\begin{proof}
Consider a solution $u$ of \eqref{eq:ModelProb},\eqref{eq:InitialCond} in the sense of
Definition \ref{def:AdmWithTraces} and fix a pair 
$(c^l,c^r)\in U \times U$; consider now the function
$c(x)=c^l\,\char_{\Set{x<0}}+c^r\,\char_{\Set{x>0}}$.
Take  $0\le \xi\in {\mathcal{D}}((0,\infty)\times \R)$ and
consider the compactly supported in $\Omega^{l,r}$ test functions
\begin{equation}\label{eq:test-functions}
    \xi^\pm_h:=\xi\,\min\Set{1,\frac{(x^\pm-h)^+}{h}}
\end{equation}
in the Kruzhkov entropy formulations of Definition \ref{def:AdmWithTraces}(i).
Since $c(x)$ is constant in each of the domains $\Omega^{l,r}$,
setting $\xi_h=\xi^-_h+\xi^+_h$ we have
$$
\int_{\R_+} \int_\R
\Bigl\{|u-c(x)|\;(\xi_h)_t
+\mathfrak{q}(x,u,c(x)) (\xi_h)_x \Bigr\} \geq  0.
$$
Sending $h>0$ to zero, using Definition \ref{def:AdmWithTraces}(i) for
the existence of strong traces and
Remarks \ref{RemFunctionQ}, \ref{RemTracesForEntropyfluxes}
for some of their properties, we deduce as in \eqref{eq:AfterDoubling} that
\begin{equation}\label{eq:intermediate-global-formul}
    \begin{split}
        & \int_{\R_+}\int_\R \Bigl\{|u(t,x)-c(x)| \xi_t
        +\mathfrak{q}(x,u(t,x),c(x)) \xi_x \Bigr\}
        \\ & \quad
        -\int_{\R_+}\left( Q^l\Bigl((\gamma^lV^lu)(t),V^lc^l\Bigr)
        -Q^r\Bigl((\gamma^rV^ru)(t),V^rc^r\Bigr)\right) \xi(t,0) \geq  0.
    \end{split}
\end{equation}
Set $v^{l,r}(t):=(\gamma^{l,r}V^{l,r}u)(t)$ and $z^{l,r}:=V^{l,r}c^{l,r}$.
By Definition \ref{def:AdmWithTraces}(ii),
$$
(v^l(t),v^r(t)) \in V\mathcal{G}^*,
\quad \text{for a.e.~$t>0$.}
$$
Hence,  by Remark \ref{rem:L1DGerm-Q}, $\forall (w^l,w^r)\in V\mathcal{G}$ we
obtain $Q^l(v^l,w^l)\geq Q^r(v^r,w^r)$, and so
\begin{align*}
    Q^l(v^l,z^l) -Q^r(v^r,z^r)
    & =Q^l(v^l,w^l) -Q^r(v^r,w^r)+Q^l(v^l,z^l) -Q^r(v^r,z^r)
    \\ & \qquad -Q^l(v^l,w^l) +Q^r(v^r,w^r)
    \\ & \geq \Bigl[Q^l(v^l,z^l) -Q^r(v^r,z^r)\Bigr]
    -\Bigl[Q^l(v^l,w^l)-Q^r(v^r,w^r)\Bigr];
\end{align*}
coming back to the definitions of $V^{l,r}$ and $Q^{l,r}$. From
\eqref{eq:prop2-remainder} we then easily deduce that
$$
Q^l(v^l,z^l) -Q^r(v^r,z^r)
\geq -R_{\mathcal{G}}\Bigl((c^l,c^r)\Bigr);
$$
therefore \eqref{eq:intermediate-global-formul} yields
\eqref{eq:EntropySolDefi}, at least for test
functions $\xi$ that vanish on $\{t=0\}$.

We deduce \eqref{eq:EntropySolDefi} for all $0\le \xi\in \mathcal{D}([0,\infty),\R)$, by
truncating $\xi$ in a neighbourhood of
$\{t=0\}$ and using Definition \ref{def:AdmWithTraces}(iii). Thanks to
Remark \ref{rem:Def1-remarks}(ii),  $u$ is also a weak solution
solution of \eqref{eq:ModelProb}, \eqref{eq:InitialCond}; thus it
is a solution in the sense of Definition \ref{def:AdmIntegral}.

Reciprocally, take a solution $u$ of
\eqref{eq:ModelProb}, \eqref{eq:InitialCond} in the sense of
Definition~\ref{def:AdmIntegral}. First, it is clear that
Definition~\ref{def:AdmWithTraces}(i) holds. Therefore it
follows from \eqref{eq:InitStrongTraces} that also
Definition \ref{def:AdmWithTraces}(iii) is fulfilled. Similarly,
it follows from Theorem~\ref{TheoPanovNormalTraces} that strong
left and right traces $\gamma^{l,r}V^{l,r}u$ on
$\Sigma=(0,\infty) \times \Set{0}$ exist. It remains to show that
$\Bigl( (\gamma^l V^lu)(t),(\gamma^rV^ru)(t)\Bigr)\in
V\mathcal{G}^*$ pointwise on $\Sigma$.
To this end, with $\xi\in \mathcal{D}([0,\infty)\times\R)$, $\xi\geq 0$,
and $h>0$, take $\psi_h=\frac{(h-|x|)^+}{h}\xi$ as test function in
\eqref{eq:EntropySolDefi} and pass to the limit as $h\to 0$. Using
Remarks~\ref{RemFunctionQ} and \ref{RemTracesForEntropyfluxes},
we deduce
\begin{equation*}%\label{eq:GlobalToPiecewise}
    \int_{\R_+}\left(Q^l\Bigl((\gamma^lV^lu)(t),V^lc^l\Bigr)
    -Q^r\Bigl((\gamma^rV^ru)(t),V^rc^r\Bigr)\right) \xi(t,0) \geq  0.
\end{equation*}
Now we take $(c^l,c^r)\in \mathcal{G}$ and utilize that
$R_{\mathcal{G}}\Bigl((c^l,c^r)\Bigr)=0$, cf.~property \eqref{eq:prop1-remainder}.
Since $\xi(\cdot,0) \in \mathcal{D}([0,\infty))$, $\xi\ge 0$,
is arbitrary, we get
$$
Q^l\Bigl((\gamma^lV^lu)(t),V^lc^l\Bigr)
\geq Q^r\Bigl((\gamma^rV^ru)(t),V^rc^r\Bigr),
\quad \text{for a.e.~$t>0$},
$$
for any $(c^l,c^r)\in \mathcal{G}$. In addition, recall that, $u$ is a
weak solution of \eqref{eq:ModelProb}; by
Remark \ref{RemTracesForEntropyfluxes}, the Rankine-Hugoniot
relation $g^l(\gamma^lV^lu)=g^r(\gamma^rV^ru)$ holds. 
By Remark \ref{rem:L1DGerm-Q}, we infer that
$$
\Bigl((\gamma^l V^lu)(t),(\gamma^rV^ru)(t)\Bigr)\in V\mathcal{G}^*.
$$
This justifies Definition \ref{def:AdmWithTraces}(ii)
and concludes the proof.
\end{proof}

\subsection{Comparison and continuous dependence results
and $L^\infty$ estimates}\label{ssec:BasicTheory-Compar-ContDep-MaxPrinc}

In this section we provide a comparison result and some $L^\infty$
estimates for $\mathbf{\mathcal{G}}$-entropy solutions.  Moreover, we
introduce a ``distance'' between two germs, which is used to state a
continuous dependence result with respect to the choice of the germ
$\mathcal{G}$.  This result prepares the ground for a study of
parametrized families of germs in \cite{AKR-II}.

\begin{theo}\label{th:Contraction-Comparison}
  Assume that $\mathcal{G}$ is a definite germ. Let $u$ and $\hat u$ be two
  $\mathcal{G}$-entropy solutions of problem \eqref{eq:ModelProb},
  \eqref{eq:InitialCond} corresponding to initial data $u_0$ and $\hat
  u_0$, respectively, such that $(u_0-\hat u_0)^+\in L^1(\R)$. Then
  \begin{equation*}%\label{eq:L1Contraction-Comparison}
    \int_{\R}  (u-\hat u)^+(t)
    \leq \int_{\R} (u_0-\hat u_0)^+,
    \quad \text{for a.e.~$t>0$}.
  \end{equation*}
  In particular, $u_0\leq \hat u_0\Rightarrow u\leq \hat u$.
\end{theo}

\begin{proof}
The proof is the same as the one used
to conclude Theorem~\ref{theo:LocalESUniqueness}.
The doubling-of-variables arguments is used to
derive the Kato inequality with $(u-\hat u)^+$
instead of $\abs{u-\hat u}$, and we then derive the analogue of inequality
\eqref{eq:L1D-States} for the associated entropy fluxes $q^{l,r}_\pm$.
As in Remark \ref{RemFunctionQ}, these fluxes are
expressed by means of the functions
$$
Q^{l,r}_\pm:(z,k)\mapsto \sign^\pm(z-k)(g^{l,r}(z)-g^{l,r}(k)).
$$
Take $(v^l,v^r),(\hat v^l,\hat v^r)\in V\mathcal{G}^*$; it suffices to show that
\begin{equation}\label{eq:Half-entropiesIneq}
    Q^l_\pm(v^l,\hat v^l) \geq Q^r_\pm (v^r,\hat v^r).
\end{equation}
By Remark \ref{rem:L1DGerm-Q}, we have \eqref{eq:L1DStates-Q} and,
moreover, $g^l(v^l)=g^r(v^r)$, $g^l(\hat v^l)=g^r(\hat v^r)$.
Consider for example the case $g^l(v^l)-g^l(\hat v^l)>0$, in which
case \eqref{eq:L1DStates-Q} is equivalent to the
inequality $\sign(v^l-\hat v^l)\geq \sign(v^r-\hat v^r)$.
This implies the following two inequalities: $\sign^\pm(v^l-\hat v^l)
\geq \sign^\pm(v^r-\hat v^r)$, which in turn imply
\eqref{eq:Half-entropiesIneq}.  The other cases are similar.
\end{proof}

As a corollary, we have  the following maximum principle:
Let $u$ be a  $\mathcal{G}$-entropy solution of
\eqref{eq:ModelProb} with an initial function $u_0$ satisfying
for some $(c^l,c^r),(C^l,C^r)\in\mathcal{G}^*$
the following lower and upper bounds:
$$
\text{$c^l \leq u_0(x)\leq C^l$ for a.e.~$x\in \R_-$}, \quad
\text{$c^r\leq u_0(x)\leq  C^r$ for a.e.~$x\in \R_+$}.
$$
Then for a.e.~$t>0$ the same
inequalities are satisfied by $u(t,\cdot)$. The proof is immediate
from Proposition \ref{prop:SenseOfG,G*}(ii)
and Theorem \ref{th:Contraction-Comparison}.

In general, $\norm{u(t,\cdot)}_{L^\infty}$ can be greater than
$\norm{u_0}_{L^\infty}$; see e.g.~the examples
in \cite{BurgerKarlsenMishraTowers}.  But we have

\begin{prop}\label{prop:LinftyEstimate}
Let $\mathcal{G}$ be a definite  germ. Let $m_-\leq M_-$,
$m_+\leq M_+$ be real numbers such that there exist
$\mathcal{G}$-entropy solutions of the Riemann problems
\eqref{eq:ModelProb},\eqref{eq:RiemannInitialCond} with data
$(u_-,u_+)=(m_-,m_+)$ and with data $(u_-,u_+)=(M_-,M_+)$. By
Remark \ref{rem:RPb-G-solutions}, strong left and right traces of
these solutions on $\Set{x=0}$ exist; denote
these traces by $c^{l,r}$,$C^{l,r}$, respectively.
If $m_\pm\leq u_0\leq M_\pm$ for a.e.~in $\R^\pm$ and $u$ is the
$\mathcal{G}$-entropy solution of \eqref{eq:ModelProb}
with the initial function $u_0$, then
\begin{align*}
    &\min\{m_-,c^{l}\}\,\char_{\Set{x<0}}+\min\{m_+,c^{r}\}\, \char_{\Set{x>0}}
    \\ & \qquad \qquad\qquad
    \leq  u(t,x) \leq
    \\ & \max\{M_-,C^{l}\}\,\char_{\Set{x<0}}+\max\{M_+,C^{r}\}\,\char_{\Set{x>0}},
    \qquad \text{a.e.~in $\R_+\times\R$.}
\end{align*}
\end{prop}

\begin{proof}
The claim follows from the comparison principle in
Theorem \ref{th:Contraction-Comparison} and from the
monotonicity property of the solutions of
the Riemann problems in the domains
$\{x<0\}$, $\{x>0\}$, cf.~Remark \ref{rem:RPb-G-solutions}.
\end{proof}

Finally, we state a simple result regarding
the continuous dependence of $\mathcal{G}$-entropy
solutions on the choice of the $L^1D$ germ $\mathcal{G}$.
For a pair of fixed functions $f^{l,r}\in C(U;\R)$
and the associated Kruzhkov entropy
fluxes $q^{l,r}$, consider two $L^1D$
germs $\mathcal{G}$, $\hat{\mathcal{G}}$.
Define the ``distance" function
\begin{equation}\label{eq:Germs-Topology}
    \rho(\mathcal{G};\hat{\mathcal{G}})
    :=\max\Set{0,\sup_{(b^l,b^r)\in \mathcal{G}, (\hat b^l,\hat b^r)\in \hat{\mathcal{G}}}
    \Bigl(q^r(b^r,\hat b^r)-q^l(b^l,\hat b^l)\Bigr)}.
\end{equation}
Observe that $\rho(\mathcal{G};\hat{\mathcal{G}})=0$ implies
$\mathcal{G}\subset \hat {\mathcal{G}}^*$ and $\hat{\mathcal{G}}\subset \mathcal{G}^*$.
So if the distance $\rho$ between two
maximal $L^1D$ germs is zero, then the two germs coincide.

\begin{prop}\label{prop:contdep-on-germ}
With the notation above, let $u,\hat u$ be
$\mathcal{G}$- and $\hat{\mathcal{G}}$-entropy
solutions of \eqref{eq:ModelProb}, respectively, with initial data
$u_0,\hat u_0$, respectively. Then
\begin{equation}\label{eq:compar-with-different-germs}
     \int_{\R}  \abs{u-\hat u}(t)
     \leq  \int_{\R}  |u_0-\hat u_0|
     +t\,\rho(\mathcal{G},\hat {\mathcal{G}}),
     \quad \text{for a.e.~$t>0$.}
\end{equation}
\end{prop}
The proof of Proposition \ref{prop:contdep-on-germ} represents an 
obvious modification of the last argument of
the proof of Theorem \ref{theo:LocalESUniqueness} (namely, 
the term \eqref{eq:IneqProofTh2} is controlled by the last term on right-hand side 
of \eqref{eq:compar-with-different-germs} when $0\leq \xi \leq 1$).

\subsection{Complete germs and $\mathcal{G}$-entropy
process solutions}\label{ssec:ProcessSolutionsAndComparison}

In this section we will discuss briefly the question of
existence of  $\mathcal{G}$-entropy solutions
for a general definite germ $\mathcal{G}$, and introduce
the weaker notion of $\mathcal{G}$-entropy process solutions
as a tool to establish convergence of
approximate solutions equipped with mere $L^\infty$ bounds.
For some particular cases, existence will be shown in detail
in Sections \ref{sec:Examples} and \ref{ssec:VV-in-1D}, while more
general results will be given in Section \ref{sec:Existence}.

\begin{defi}\label{def:CompleteGerm}
A germ $\mathcal{G}$ is said to be complete if for all
$(u_-,u_+)\in U \times U$, the Riemann problem
\eqref{eq:ModelProb},\eqref{eq:RiemannInitialCond}
admits a weak solution of the form {\rm(RPb-sol.)}~with
$(u^l,u^r)\in\mathcal G$ (here and in the sequel, we
adopt the convention of Remark \ref{rem:RPb-sol-with-zero-speed}).
\end{defi}

Observe that by Proposition \ref{prop:DualityOfGerms}(i) and
Remark \ref{rem:Def1-remarks}(iv), if $\mathcal{G}$ is an $L^1D$ germ,  then
the aforementioned weak solution of problem
\eqref{eq:ModelProb},\eqref{eq:RiemannInitialCond} is also the unique
$\mathcal{G}$-entropy solution of the problem.

\begin{prop}\label{prop:ExistenceAndCompleteness}
Let $\mathcal{G}$ be a definite  germ. If there exists
a $\mathcal{G}$-entropy solution of \eqref{eq:ModelProb},\eqref{eq:InitialCond} for any
$u_0\in L^\infty(\R;U)$, then $\mathcal{G}^*$ must be complete.
\end{prop}

\begin{proof}
Consider general Riemann initial data
of the form \eqref{eq:RiemannInitialCond}.
By assumption and by Remark \ref{rem:RPb-G-solutions}, there
exists exactly one $\mathcal{G}$-entropy solution $u$ to
\eqref{eq:ModelProb},\eqref{eq:RiemannInitialCond}; moreover, $u$ is of the
form {\rm(RPb-sol.)}, and the corresponding pair of states
$(u^l,u^r)$ belongs to $\mathcal{G}^*$.  Since $u$ is necessarily
a weak solution, $\mathcal{G}^*$ is complete.
\end{proof}

The assumption that $\mathcal{G}^*$ is complete
is expected to yield existence of  $\mathcal{G}$-entropy solutions.
Indeed, starting from an appropriate Riemann solver (a classical
solver for $x\neq 0$, and the solver $\mathcal {RS}^\mathcal{G}$ at
$x=0$), it is possible to construct numerical approximations of
$\mathcal{G}$-entropy solutions, e.g., by the Godunov finite volume scheme or
by the front tracking scheme; see Section  \ref{ssec:Numerics}.
Unfortunately, currently we can deduce existence only
in the presence of good pre-compactness properties
(see, e.g., \cite{GaravelloAndAl} and \cite{BurgerKarlsenTowers} for
uniform $BV$ estimates for front tracking and
adapted Engquist-Osher schemes, respectively).

In passing, we mention that ``completeness" makes
available uniform $L^\infty$ bounds. If $\mathcal{G}$ is definite  and
$\mathcal{G}^*$ is complete, Proposition \ref{prop:LinftyEstimate}
yields an  $L^\infty$ bound on any $\mathcal{G}$-entropy
solution $\mathcal{S}_t^{\mathcal{G}}u_0$ of the form
$$
\norm{(\mathcal{S}_t^{\mathcal{G}}u_0)(\cdot)}_{L^\infty(\R)}
\leq \mathrm{Const}\left(\|u_0\|_{L^\infty},\mathcal{G}\right),
\quad \text{uniformly in $t$}.
$$

If similar $L^\infty$ estimates are available for sequences
produced by an approximation procedure,
and if the existence of the $\mathcal{G}$-entropy solution
is already known by other means, then Definition \ref{def:EntropyProcess} and
Theorem \ref{th:UniqProcessSol} below provide a ``propagation of compactness" approach
to proving convergence of such sequences.

To this end, following \cite{Szepessy,Panov-process-sol,EGH,GallouetHubert} (see
also references cited therein and the previous works of Tartar and DiPerna), we
first extend the notion of solution given in Definition \ref{def:AdmIntegral}
to account for so-called process solutions (higher dimensional $L^\infty$ objects
related to the distribution function of the Young measure).

\begin{defi}\label{def:EntropyProcess}
Let  $\mathcal{G}$ be an $L^1D$ germ.
A function $\mu\in L^\infty(\R_+\times\R \times (0,1);U)$ is
called a  $\mathcal{G}$-entropy process solution of
\eqref{eq:ModelProb},\eqref{eq:InitialCond} if the following conditions hold:

1 (weak process formulation). For all $\xi\in  \mathcal{D}([0,\infty) \times \R)$,
\begin{equation}\label{eq:weak-proc-formulation}
    \int_0^1\int_{\R_+} \int_\R
    \Bigl\{\mu(t,x,\alpha) \xi_t
    +\mathfrak{f}(x,\mu(t,x,\alpha)) \xi_x \Bigr\}
    -\int_{\R} u_0 \xi(0,x)=0.
\end{equation}

2 (penalized entropy process inequalities). For all pairs $(c^l,c^r)\in U \times U$ and for all
$\xi\in {\mathcal{D}}([0,\infty)\times \R)$, $\xi \geq 0$,
\begin{equation}\label{eq:EntropyProcessSolDefi}
    \begin{split}
        & \int_0^1\int_{\R_+} \int_\R
        \Bigl\{|\mu(t,x,\alpha)-c(x)| \xi_t
        +\mathfrak{q}(x;\mu(t,x,\alpha),c(x)) \xi_x \Bigr\}
        \\ & \qquad \qquad\quad
        -\int_{\R} |u_0-c(x)| \xi(0,x)
        +\int_{\R_+}R_{\mathcal{G}}\Bigl((c^l,c^r)\Bigr) \xi(t,0) \geq  0,
    \end{split}
\end{equation}
where $c(x)=c^l\,\char_{\Set{x<0}}+c^r\,\char_{\Set{x>0}}$,
and the remainder term $R_{\mathcal{G}}\Bigl((c^l,c^r)\Bigr)$ has the same
meaning as in Definition \ref{def:AdmIntegral}.
\end{defi}

As for Definition \ref{def:AdmIntegral}, alternative (equivalent)
forms of the remainder term can be chosen, provided the two
properties \eqref{eq:prop1-remainder}
and \eqref{eq:prop2-remainder} are fulfilled.

\begin{rem}[cf.~Proposition \ref{prop:Carrillo-type-defs}]\label{rem:simpler-entropyproc-def}
A $\mathcal{G}$-entropy process solution is equivalently
characterized by the following three requirements:
\begin{itemize}
    \item [(a)] The weak process formulation \eqref{eq:weak-proc-formulation} holds;
    \item [(b)] For all $\xi\in {\mathcal{D}}([0,\infty)\times \R)$, $\xi \geq 0$,
    such that $\xi=0$ on the interface $\Set{x=0}$, for all pairs $(c,c)$, $c\in U$,
    the entropy inequalities \eqref{eq:EntropyProcessSolDefi}
    hold (with zero remainder term);
    \item [(c)] For all $\xi\in {\mathcal{D}}([0,\infty)\times \R)$, $\xi \geq 0$, for
    all pairs $(c^l,c^r)\in\mathcal{G}$, the entropy
    inequalities \eqref{eq:EntropyProcessSolDefi} hold
    (with zero remainder term).
\end{itemize}
\end{rem}

Indeed, it is evident that Definition \ref{def:EntropyProcess}
implies the three properties above.
Clearly, in the cases $(b)$ and $(c)$ above, the remainder
in \eqref{eq:EntropyProcessSolDefi} vanishes, cf.~\eqref{eq:prop1-remainder}.
For the proof of the converse implication, as in the proof
of Theorem \ref{th:DefLocal-Global}, we utilize the test functions $\xi^\pm_h$
given in \eqref{eq:test-functions} and the remainder
term property \eqref{eq:prop2-remainder}.

\begin{rem}\label{rem:TracesOfProcesses}
Fix a pair $(c^l,c^r)\in U\times U$, and let $\mu$ be a
$\mathcal{G}$-entropy process solution.
We want to consider the functions
$$
I^{l,r}:(t,x)\mapsto \dsp \int_0^1
\abs{\mu(t,x,\alpha)-c^{l,r}}\,d\alpha, \quad
J^{l,r}:(t,x)\mapsto \int_0^1
q^{l,r}(\mu(t,x,\alpha),c^{l,r})\,d\alpha.
$$
The inequalities in \eqref{eq:EntropyProcessSolDefi}
imply that the $\R^2$ vector field
$$
(t,x)\mapsto \Bigl(I^l(t,x),J^{\,l}(t,x)\Bigr)
$$
is an $L^\infty$ divergence-measure field on $\R_+ \times (-\infty,0)$ (locally in $t$).
According to the results in Chen and Frid \cite{ChenFrid}, such a
vector field admits a weak normal trace on the boundary
$\Sigma=\R_+\times \Set{0}$. This means that $J^{\,l}$ admits a
weak trace from the left on $\Sigma$. Similarly, we conclude that
$J^{\,r}$ admits a weak trace from the right on $\Sigma$.
\end{rem}

In what follows, the weak trace operators
$t\mapsto \gamma^{l,r}_wJ^{\,l,r}(t,\cdot)$ from
Remark \ref{rem:TracesOfProcesses} will be denoted by $\gamma^{l,r}_w$.

\begin{lem}\label{ProcessTraceProperty}
Let $\mu$ be a $\mathcal{G}$-entropy process solution. Then for
all $(v^l,v^r)\in V\mathcal{G}$,
\begin{equation}\label{eq:ProcessTracesIneq}
    \gamma^l_w \int_0^1 Q^{l}\Bigl(V^l\mu(t,\cdot,\alpha),v^{l}\Bigr)\, d\alpha
    \geq \gamma^r_w \int_0^1 Q^{r}\Bigl(V^r\mu(t,\cdot,\alpha),v^{r}\Bigr)\, d\alpha,
\end{equation}
for a.e.~$t>0$.
\end{lem}

\begin{proof}
The proof is very similar to the proof of the second part of
Theorem \ref{th:DefLocal-Global}. We take $c^{l,r}$ such that
$v^{l,r}=V^{l,r}c^{l,r}$. For $\xi\in  \mathcal{D}([0,\infty)\times\R)$,
$\xi(\cdot,0)\geq 0$, and $h>0$, we take $\frac{(h-|x|)^+}{h}\xi$ as
test function in \eqref{eq:EntropyProcessSolDefi}, use
the Fubini theorem, and pass to the limit
as $h\to 0$ to finally arrive at \eqref{eq:ProcessTracesIneq}.
\end{proof}

Notice that, unlike Remark \ref{rem:GandG*}(ii), it is not
clear whether $\mathcal{G}$- and $\mathcal{G}^*$- entropy process solutions
coincide for a definite germ $\mathcal{G}$. Therefore we require
that $\mathcal{G}=\mathcal{G}^*$ (i.e., we require that
$\mathcal{G}$ is a maximal $L^1D$ germ) while working with
$\mathcal{G}$-entropy process solutions.

\begin{theo}\label{th:UniqProcessSol}
Let $\mathcal{G}$ be a maximal $L^1D$ germ.
Suppose $u_0$ is such that there exists an
$\mathcal{G}$-entropy solution $u$ of \eqref{eq:ModelProb}
with initial data \eqref{eq:InitialCond}.
Then there exists a unique $\mathcal{G}$-entropy process
solution $\mu$ with initial data $u_0$, and
$\mu(\alpha)=u$ for a.e.~$\alpha\in (0,1)$.
\end{theo}

\begin{proof}
Being a  $\mathcal{G}$-entropy solution, $u$
fulfills Definition \ref{def:AdmWithTraces}(ii). Since $\mathcal
G=\mathcal{G}^*$ by Proposition \ref{prop:DualityOfGerms}(iii), the pair
$(v^l(t),v^r(t)):=\Bigl((\gamma^lV^lu)(t),(\gamma^rV^ru)(t)\Bigr)$ belongs to
$V\mathcal{G}^*$ for a.e.~$t>0$.  Hence, inequality
\eqref{eq:ProcessTracesIneq} at a point $t>0$ holds
in particular with the choice $v^{l,r}=v^{l,r}(t)$. Further, notice that since the
traces of $V^{l,r}u$ are strong, the weak traces
$\gamma_w^{l,r} \int_0^1 Q^{l,r}\Bigl(V^{l,r}\mu(\alpha),V^{l,r}u\Bigr)\,d\alpha$
on $\Sigma$ exist and are equal to $\gamma_w^{l,r} \int_0^1 Q^{l,r}
\Bigl(V^{l,r}\mu(\alpha),\gamma^{l,r}V^{l,r}u\Bigr)\,d\alpha$.
Therefore
\begin{equation}\label{eq:IneqProofTh4}
    \begin{split}
        & \int_{\R_+}\Biggl(
        \gamma^l_w \int_0^1 Q^{l}\Bigl(V^l\mu(t,\cdot,\alpha),V^{l}u(t,\cdot)\Bigr)\,d\alpha
        \\ & \qquad \qquad
        -\gamma^r_w \int_0^1 Q^{r}\Bigl(V^r\mu(t,\cdot,\alpha),V^{r}u(t,\cdot)\Bigr)
        \,d\alpha\Biggr) \xi(t,0) \geq 0,
    \end{split}
\end{equation}
for all $\xi(\cdot,0)\in \mathcal{D}([0,\infty))$, $\xi(\cdot,0)\geq 0$.
Now we repeat the arguments in the proof of
Theorem \ref{theo:LocalESUniqueness}, which leads to the analogue of
inequality \eqref{eq:AfterDoubling}; after that, in view
of \eqref{eq:IneqProofTh4}, we deduce
$$
\int_0^1\int_{\R_+} \int_\R
\Bigl\{|u-\mu(\alpha)| \xi_t+\mathfrak{q}(x,u,\mu(\alpha))\xi_x\Bigr\} \geq 0,
$$
because $u$ and $\mu$ are solutions with the same data.
As in \cite{Kruzhkov} (for Lipschitz fluxes $f^{l,r}$) or in
\cite{Benilan,KruzhkovHildebrand} (for merely continuous
$f^{l,r}$), we work out finally that
$$
\text{$\mu(\cdot,\cdot,\alpha)=u(\cdot,\cdot)$ a.e.~on $\R_+\times\R$, for
any $\alpha\in (0,1)$.}
$$
\end{proof}

We have the following open problem: \textit{Is it true that for a
  definite germ $\mathcal{G}$ there is at most one
  $\mathcal{G}$-entropy process solution $\mu$ with given initial data
  $u_0$ (in which case $\mu$ is independent of $\alpha$) ?} The answer
is known to be positive in some situations. Firstly, according to
\cite{EGH,GallouetHubert}, this is the case for the Kruzhkov entropy
process solutions of $u_t+f(u)_x=0$, which can be equivalently defined
by Definition \ref{def:EntropyProcess} starting from the germ
$\mathcal{G}_{VKr}$ (see Section \ref{ssec:VolpertKruzhkov}).
Secondly, if $f^{l}$ is non-decreasing and $f^r$ is non-increasing,
then every $L^1D$ germ $\mathcal{G}$ is a singleton $\{(v^l,v^r)\}$
(see Section \ref{ssec:AntiMonotone}); in this case, from
\eqref{eq:ProcessTracesIneq} we readily derive that
$V^{l,r}\mu(\alpha)=v^{l,r}$ for a.e.~$\alpha$, and thus uniqueness of
a $\mathcal{G}$-entropy solution follows.

\subsection{Conclusions}\label{ssec:ModelProblemConclusions}
Let us review the results of Section \ref{sec:ModelProblem}.
Admissibility of entropy solutions to the model equation
\eqref{eq:ModelProb} is most naturally defined starting from a germ
$\mathcal{G}$ possessing a unique maximal $L^1D$ extension
$\mathcal{G}^*$ that is complete.  In this case, uniqueness,
$L^1$-contractivity, and comparison properties hold for solutions in
the sense of (the equivalent) Definitions \ref{def:AdmWithTraces} and
\ref{def:AdmIntegral}, and we can hope for the existence of solutions
for general data.  Furthermore, uniform $L^\infty$ estimates of
solutions can be obtained (see, e.g., Proposition
\ref{prop:LinftyEstimate}) and measure-valued solutions (see
Definition \ref{def:EntropyProcess}) can be considered, but for the
latter we need to know already the existence of solutions in the sense
of Definitions \ref{def:AdmWithTraces}, \ref{def:AdmIntegral} and also
the additional assumption $\mathcal{G}=\mathcal{G}^*$ must hold; under
these conditions we are able to conclude that Definition
\ref{def:EntropyProcess} is in fact equivalent to Definitions
\ref{def:AdmWithTraces}, \ref{def:AdmIntegral}.  In Sections
\ref{sec:Examples}, \ref{sec:SVVGerm}, we illustrate the theory of
germs by a number of examples.

\begin{rem}\label{rem:CharactSemigroups}
Each maximal $L^1D$ germ $\mathcal{G}$ gives rise
to an $L^1$-contractive  semigroup for \eqref{eq:ModelProb}.
Conversely, by the analysis in Section \ref{ssec:RiemannAnalysis}, any
semigroup defined on the whole space $L^\infty(\R;U)$ that
satisfy ${\rm (A1)}$ and ${\rm (A2)}$ corresponds to a 
complete maximal $L^1D$ germ. If we drop the semigroup 
property and the  scaling invariance $\text{\rm(A1)}$, then 
different $L^1$-dissipative solvers can be
constructed, e.g.,~starting from a family $\Set{\mathcal{G}(t)}_{t>0}$ of
maximal $L^1D$ germs; we refer to Colombo and Goatin \cite{ColomboGoatin}
for an example of such a solver, and to \cite{AndrGoatinSeguin}
for an analysis of the Colombo-Goatin solver in terms of admissibility germs.
Notice that in order to have existence we must assume
that $\mathcal{G}(\cdot)$ is measurable in an appropriate sense. We
will  develop an approach to measurability for
time-dependent families of germs in \cite{AKR-II}.
\end{rem}

\section{Examples and analysis of known admissibility criteria}\label{sec:Examples}

In this section we first discuss the completeness of germs, a
``closure'' operation on germs, and relations between
completeness and maximality of $L^1D$ germs.
We continue by reviewing a number of known
admissibility criteria for problem \eqref{eq:ModelProb} and
make explicit the underlying germs. In some cases, using the
results of Section \ref{sec:ModelProblem} we refine the known
uniqueness theory. Some known and some new existence results for these
criteria are contained in Sections \ref{ssec:VV-for-connections}, \ref{ssec:Numerics}.
A similar study of the important ``vanishing viscosity'' germ is postponed
to Sections \ref{sec:SVVGerm} and \ref{ssec:VV-in-1D}.

\subsection{More about complete germs and closed germs}\label{ssec:GermProperties}

For a given $u^l\in U$, we denote by $\theta^l(\cdot,u^l)$ the graph that
contains all points $(u_-,f^l(u_-))\in U\times \R$ such that
there exists a Kruzhkov solution to the Riemann problem joining
the state $u_-$ at $x<0$, $t=0$ to the state $u^l$ at $x=0^-$,
$t>0$ (i.e., this solution contains only waves of nonpositive
speed). Similarly, for $u^r\in U$, $\theta^r(\cdot,u^r)$ is the graph of all
points $(u_+,f^r(u_+))$ such that there exists a Kruzhkov solution
to the Riemann problem joining $u^r$ at $x=0^+$, $t>0$, to $u_+$ at $x>0$, $t=0$
(i.e., this solution contains only waves of nonnegative speed)

Notice that in concrete situations, the completeness of a given germ
can be checked as follows.

\begin{rem}\label{rem:complete-caract}
A germ $\mathcal{G}$ is complete if and only if
$$
\bigcup\limits_{(u^l,u^r)\in \mathcal{G}}
\dom \theta^l(\cdot,u^l)\times
\dom \theta^r(\cdot,u^r)=U \times U.
$$
\end{rem}

According to the convention of
Remark \ref{rem:RPb-sol-with-zero-speed}, the pair of traces
$(\gamma^l u,\gamma^r u)$ of a solution of the form
(RPb.-sol) does not necessarily coincide with $(u^l,u^r)$.
The reason is that zero-speed shocks may be, at least in principle, a
part of the wave fans joining $u_-$ to $u^l$ and $u^r$ to $u^+$.
In order to cope with such situations, we
introduce the following definition.

\begin{defi}\label{def:ContactShock}
A  left (resp., right) contact shock is a pair
$(u_-,u^l)\in U \times U$ (resp., $(u^r,u_+)\in U \times U$)
such that
$u(x)=
\begin{cases}
    u_- , &x<0,\\
    u^l, & x>0
\end{cases}$ is a Kruzhkov entropy solution
of $u_t+f^l(u)_x=0$ in $\R_+\times\R$ (resp., such that
$u(x)=
\begin{cases}
    u^r, & x<0,\\
    u_+, & x>0,
\end{cases}$
is a Kruzhkov entropy solution of
$u_t+ f^r(u)_x=0$ in $\R_+\times\R$).
\end{defi}

In other words,
$$
\text{$(u_-,u^l)$ is a left contact shock $\Longleftrightarrow$ $u_-\in \dom\,
\theta^l(\cdot,u^l)$ and $f^l(u_-)=f^l(u^l)$};
$$
an analogous statement is true for the right contact shocks.

Clearly, any pair $(c,c)\in U \times U$ is both a left
and a right contact shock.

\begin{defi}\label{def:ClosedGerm}
A germ $\mathcal{G}$  is said to be closed, if $\mathcal{G}$ is a closed
subset of $U \times U$ and, moreover, for all pairs
$(u^l,u^r)\in \mathcal{G}$, $\mathcal{G}$ also contains all pairs
$(u_-,u_+)$ such that $(u_-,u^l)$ is a left contact shock,
$(u^r,u_+)$ is a right contact shock.
The smallest closed extension of $\mathcal{G}$
is called the closure of $\mathcal{G}$
and is denoted by $\overline{\mathcal{G}}$.
\end{defi}

It is clear that for any solution $u$ of the form (RPb.-sol)
such that $(u^l,u^r)\in \mathcal{G}$, we have $(\gamma^lu,\gamma^r u)\in
\overline{\mathcal{G}}$ (cf.~the convention of
Remark \ref{rem:RPb-sol-with-zero-speed}).

The above definition is consistent in the sense that the
closure $\overline{\mathcal{G}}$ of a germ $\mathcal{G}$
is indeed a germ. Indeed, the Rankine-Hugoniot
condition for $(u_-,u_+)\in \overline{\mathcal{G}}$
holds because it holds for the
corresponding pair $(u^l,u^r)\in \mathcal{G}$.

Let us list some properties of the closure operation and of closed
germs (to be utilized in Section \ref{sec:SVVGerm}).

\begin{prop}\label{prop:Closure}~

\noindent (i) $\mathcal{G}$ is an $L^1D$ germ $\Rightarrow$
$\overline{\mathcal{G}}$ is an $L^1D$ germ. Furthermore,
$\overline{\mathcal{G}}\subset \mathcal{G}^*$,
$\left(\overline{\mathcal{G}}\right)^*=\mathcal{G}^*$.

\noindent  (ii) A maximal $L^1D$ germ $\mathcal{G}$ is closed.

\noindent  (iii) Any maximal $L^1D$ extension of $\mathcal{G}$ contains
$\overline{\mathcal{G}}$. In particular, $\mathcal{G}$
is a definite  germ if and only if $\overline{\mathcal{G}}$ is a definite  germ.

\noindent (iv) If $\mathcal{G}$ is a definite germ,
then $\mathcal{G}$ and $\overline{\mathcal{G}}$
entropy solutions coincide.
\end{prop}

\begin{proof}
(i) Fix $(u^l,u^r),(\hat u^l,\hat u^r)\in \mathcal{G}$, and
let $(u_-,u^l),(\hat u_-,\hat u^l)$ be left contact
shocks and $(u^r,u_+),(\hat u^r,\hat u_+)$ be
right contact shocks. The Kruzhkov
admissibility of these shocks implies that
$$
q^l(u_-,\hat u_-) \geq q^l(u^l,\hat u^l), \qquad
q^r(u^r,\hat u^r)\geq q^r(u_+,\hat u_+).
$$
Moreover, we have $q^l(u^l,\hat u^l)\geq q^r(u^r,\hat u^r)$
because $\mathcal{G}$ is an $L^1D$ germ. Therefore we infer
$q^l(u_-,\hat u_-)\geq q^r(u_+,\hat u_+)$.
From the definition, it is not difficult to
see that $\overline{\mathcal{G}}$ is obtained by
the following two operations:
\begin{itemize}
  \item[--] firstly, taking the topological closure in $\R^2$, and

  \item[--] secondly, adjoining (possibly trivial) left and right contact
  shocks to the elements of $\mathcal{G}$.
\end{itemize}
The result of the second operation is topologically closed.  Because
each of the two operations preserves inequality \eqref{eq:L1D-States},
by the argument above and by the continuity of $q^{l,r}$,
$\overline{\mathcal{G}}$ is an $L^1D$ germ.  By Proposition
\ref{prop:DualityOfGerms}(ii), $\overline{\mathcal{G}}$ being an
$L^1D$ extension of $\mathcal{G}$, $\overline{\mathcal{G}}$ is
contained within $\mathcal{G}^*$.

Finally, by Definition \ref{def:DualGerm}, $\mathcal{G}\subset
\overline{\mathcal{G}}$ implies $(\mathcal{G})^*\supset
\left(\overline{\mathcal{G}}\right)^*$.
The reciprocal inclusion is also true.
Indeed, by Proposition \ref{prop:DualityOfGerms}(ii), each pair
$(u^l,u^r)\in\mathcal{G}^*$ belongs to an $L^1D$ extension
$\mathcal{G}'$ of $\mathcal{G}$. By Definition \ref{def:ClosedGerm},
$\mathcal{G}\subset \mathcal{G}'$ implies $\overline{\mathcal
G}\subset \overline{\mathcal{G}'}$. By our first claim,
$\overline{\mathcal{G}'}$ is an $L^1D$ extension of
$\overline{\mathcal{G}}$; using once more
Proposition \ref{prop:DualityOfGerms}(ii), we infer that
$(u^l,u^r)\in \overline{\mathcal{G}'}\subset \left(\overline{\mathcal{G}}\right)^*$.
Thus $\mathcal{G}^*=(\overline{\mathcal{G}})^*$.

\noindent (ii) By Definition \ref{def:ClosedGerm} and by (i), one has
$\mathcal{G}\subset \overline{\mathcal{G}}\subset \mathcal{G}^*$.
Since $\mathcal{G}=\mathcal{G}^*$ by
Proposition \ref{prop:DualityOfGerms}(iii), it follows that
$\mathcal{G}$ coincides with its closure.

\noindent (iii) If $\mathcal{G} \subset \mathcal{G}'$ and $\mathcal{G}'$ is
a maximal $L^1D$ germ, then $\overline{\mathcal{G}} \subset \overline{\mathcal{G}'}$;
besides, $\overline{\mathcal{G}'}=\mathcal{G}'$ by (ii).

\noindent (iv) This follows from (i),(iii) and Remark \ref{rem:GandG*}(ii).
\end{proof}

Let us list separately the properties related to complete  germs.

\begin{prop}\label{prop:Completeness}~

\noindent (i) If $\mathcal{G}$ is a complete $L^1D$ germ,
then $\overline{\mathcal{G}}$ is a maximal $L^1D$ germ.

\noindent (ii) If $\mathcal{G}$ is an $L^1D$ germ such that
$\overline{\mathcal{G}}$ is complete, then $\mathcal{G}$
is definite and $\mathcal{G}^*=\overline {\mathcal{G}}$.
\end{prop}

Notice that in case (ii), $\mathcal{G}$ is a definite $L^1D$
germ and $\mathcal{G}^*$ is complete; as it is
pointed out in Section \ref{ssec:ModelProblemConclusions},
such germs are expected to lead to a
well-posedness theory for
$\mathcal{G}$-entropy solutions.

\begin{proof}

  \noindent (i) Let $\mathcal{G}'$ be a nontrivial maximal $L^1D$
  extension of $\overline{\mathcal{G}}$, and pick a pair
  $(u_-,u_+)\in\mathcal{G}'\setminus \overline{\mathcal{G}}$.  Now the
  Riemann problem \eqref{eq:ModelProb},\eqref{eq:RiemannInitialCond}
  possesses two solutions of the form (RPb.-sol), namely

  \begin{itemize} 
    \item[--] a solution $u=\mathcal {RS}^\mathcal{G}(u_-,u_+)$, which
    exists because $\mathcal{G}$ is complete;

    \item[--] the elementary solution $\hat u=
    u_-\,\char_{\Set{x<0}}+u_+\,\char_{\Set{x>0}}$.
  \end{itemize}
  Both are $\mathcal{G}'$-entropy solutions, by Remark
  \ref{rem:Def1-remarks}(iv) and because $(\gamma^lu,\gamma^ru)\in
  \overline{\mathcal{G}}\subset \mathcal{G}' =
  \left(\mathcal{G}'\right)^*$ and $(\gamma^l\hat u,\gamma^r\hat
  u)=(u_-,u_+) \in \mathcal{G}'=\left(\mathcal{G}'\right)^*$.  In
  addition, $u$ and $\hat u$ do not coincide a.e.~because their traces
  are different: $(\gamma^l\hat u,\gamma^r\hat u) =(u_-,u_+)\notin
  \overline{\mathcal{G}}\ni (\gamma^l u,\gamma^r u)$.  Since
  $\mathcal{G}'$ is maximal $L^1D$, this contradicts the uniqueness
  result of Theorem \ref{theo:LocalESUniqueness}. This contradiction
  proves that $\overline{\mathcal{G}}$ is itself maximal $L^1D$.

  \noindent (ii) By Proposition \ref{prop:Closure}(i),
  $\overline{\mathcal{G}}$ is a complete $L^1D$ germ. By (i), it is a
  maximal $L^1D$ germ and thus it is definite. By Proposition
  \ref{prop:Closure}(iii), also $\mathcal{G}$ is definite. Finally, by
  Propositions \ref{prop:Closure}(i) and
  \ref{prop:DualityOfGerms}(iii), it follows that
  $\mathcal{G}^*=\left(\overline{\mathcal G}\right)^*
  =\overline{\mathcal{G}}$.
\end{proof}

Finally, let us point out a situation where the existence of a specific pair
in an $L^1D$ germ $\mathcal{G}$ immediately excludes
some other pairs from $\mathcal{G}$.

\begin{prop}\label{PropIsolatedGermPoints}
Let $\mathcal{G}$ be an $L^1D$ germ, and fix a pair
$(u^l,u^r)\in \mathcal{G}$.
Consider $u_-\in \dom \theta^l(\cdot,u^l)$ and
$u_+\in \dom\, \theta^r(\cdot,u^r)$.
Then  $(u_-,u_+)\notin \overline{\mathcal{G}}$, except in
the case with $f^l(u_-)=f^r(u_+)=f^l(u^l)=f^r(u^r)$.
\end{prop}

\begin{proof}
By Proposition \ref{prop:Closure}(i), $\overline{\mathcal{G}}$ is
an $L^1D$ germ. Since $(u^l,u^r)\in \overline{\mathcal{G}}$,  in case
$(u_-,u_+)\in \overline{\mathcal{G}}$ we must have
$f^l(u_-)=f^r(u_+)$ and $q^l(u_-,u^l) \geq q^r(u_+,u^r)$.
But the monotonicity of $\theta^{l,r}(\cdot,u^{l,r})$ (which is easy to get
from the description of the Kruzhkov solutions
of the Riemann problem, see, e.g.,
\cite{HoldenRisebro,GoritskyChechkin}) yields
\begin{align*}
    q^l(u_-,u^l) & = \sign(u_- - u^l)(f^l(u_-)-f^l(u^l))
    = -\abs{f^l(u_-)-f^l(u^l)}, \\
    q^r(u_+,u^r) & = \sign(u_+ - u^r)(f^r(u_+)-f^r(u^r))
    = \abs{f^r(u_+)-f^r(u^r)}.
\end{align*}
These conditions are not compatible, unless
$f^l(u_-)=f^r(u_+)=f^l(u^l)=f^r(u^r)$. In addition, if the latter
condition holds, then $(u_-,u^l)$ (resp., $(u^r,u_+)$ is a left
(resp., right) contact shock; in which case
$(u_-,u_+)\in \overline{\mathcal{G}}$.
\end{proof}

\subsection{The case $f^l\equiv f^r$; the
Volpert-Kruzhkov germ}\label{ssec:VolpertKruzhkov}

We will now illustrate the abstract ``discontinuous flux" theory on
a well-known example, namely the case without a flux discontinuity
at $\Set{x=0}$, so $f^l=f^r$ in \eqref{eq:ModelProb}; let us
write $f$ instead of $f^{l,r}$ and take $U=\R$.

 In this case, the set of weak solutions of \eqref{eq:ModelProb} includes
 all constant in $\R_+\times\R$ solutions.
 Let us postulate that  all constants are admissible
elementary solutions of $u_t+f(u)_x=0$; in other words,
consider the ``Volpert-Kruzhkov" germ
$$
\mathcal{G}_{\mathit{VKr}}:=\Set{(c,c)\,|\,c\in\R}.
$$
One checks immediately that $\mathcal{G}$ is an $L^1D$ germ.
According to the definition, the dual germ $\mathcal{G}_{\mathit{VKr}}^*$
is determined by the family of inequalities:
\begin{equation}\label{eq:VolpertIneq}
    (u^l,u^r)\in \mathcal{G}_{\mathit{VKr}}^*
    \Longleftrightarrow
    q(u^l,c)\geq q(u^r,c),
    \qquad \forall c\in\R
\end{equation}
(here $q(u,c)=\sign(u-c)(f(u)-f(c))$).
Indeed, the Rankine-Hugoniot condition $f(u^l)=f(u^r)$, which
enters Definition \ref{def:DualGerm}, follows
by taking $\pm c$ large enough in \eqref{eq:VolpertIneq}.
Following Volpert \cite{Volpert}, we check
that $\mathcal{G}_{\mathit{VKr}}^*$ is an $L^1D$ germ:
$$
q(u^l,\hat u^l)\geq q(u^r,\hat u^l)=q(\hat u^l,u^r)
\geq q(\hat u^r,u^r)=q(u^r,\hat u^r),
$$
for all $(u^l,u^r),(\hat u^l,\hat u^r)\in\mathcal{G}_{\mathit{VKr}}^*$. By
Proposition \ref{prop:DualityOfGerms}(v), $\mathcal{G}_{\mathit{VKr}}$
is definite. Therefore uniqueness of  
$\mathcal{G}_{\mathit{VKr}}$-entropy solutions holds.

Moreover, it is well known (see, e.g., \cite{Gelfand} and
\cite{HoldenRisebro,GoritskyChechkin}) that there
exists a Kruzhkov entropy solution of $u_t+f(u)_x=0$ with the Riemann data
\eqref{eq:RiemannInitialCond}, and that this solution is
of the form (RPb.-sol) with $(u^l,u^r)$ satisfying
\eqref{eq:VolpertIneq}. Therefore the germ
$\mathcal{G}_{Kr}^*$ is also complete, which raises the
hope for existence of $\mathcal{G}_{\mathit{VKr}}$-entropy 
solutions for all bounded data. Not surprisingly, of course, it easily follows that
Definition \ref{def:AdmIntegral} with $\mathcal{G}=\mathcal{G}_{\mathit{VKr}}$ is
equivalent to the definition of Kruzhkov entropy
solutions. Hence, the existence of
$\mathcal{G}_{\mathit{VKr}}$-entropy solutions
is a consequence of \cite{Kruzhkov}.

In conclusion, the germ $\mathcal{G}_{\mathit{VKr}}$ leads to a complete
well-posedness theory for solutions in the sense of
Definitions \ref{def:AdmWithTraces} and \ref{def:AdmIntegral}, which
turns out to be precisely the theory of Kruzhkov entropy solutions.
Another description of the associated maximal $L^1D$
germ $\mathcal{G}_{\mathit{VKr}}^*$ is given in Section \ref{ssec:Gelfand} below.

\begin{rem}
As a final remark, this example is somewhat misleading. Indeed, we
have just investigated the question of admissibility of
discontinuities of solutions of $u_t+f(u)_x=0$ on the line
$\Set{x=0}$; and we took for granted that the admissibility of $u$
on $\Set{x<0}$ and $\Set{x>0}$ is understood in the sense of Kruzhkov.
The focus on the line $\Set{x=0}$ is ``artificial'' in this example.
The fact that we finally arrive at notion of Kruzhkov entropy
solutions on the whole of $\R_+ \times \R$ is of
course expected; moreover, our existence
analysis relies on the knowledge of the Riemann solver associated
to Kruzhkov entropy solutions! But we believe that this example and
its sequel in Section \ref{ssec:Gelfand} illustrate well the general approach
and some techniques related to the use of admissibility germs.
Also see Section \ref{ssec:ColomboGoatin} and
Remark \ref{rem:FantasqueSolvers} for examples of
non-Kruzhkov $L^1\!$-dissipative solvers for $u_t+f(u)_x=0$.
\end{rem}

\subsection{The case $f^l\equiv f^r$; the Gelfand germ}\label{ssec:Gelfand}
We will continue with the above example but apply a different approach,
which will be extensively investigated in Section \ref{sec:SVVGerm}
in the case of a discontinuous flux.

Recall that existence of Kruzhkov entropy solutions for
\begin{equation}\label{eq:KruzhkovConsLaw}
    u_t+f(u)_x=0, \qquad u|_{t=0}=u_0
\end{equation}
was first obtained in \cite{Kruzhkov}
through the limit, as $\eps\downarrow 0$, of the vanishing
viscosity approximations
\begin{equation}\label{eq:GelfandViscosity}
    u_t+f(u)_x=\eps u_{xx}.
\end{equation}
Indeed, it is well known that there
exists a unique solution $u^\eps$ to \eqref{eq:GelfandViscosity}
for all bounded data, that the solutions satisfy an analogue of the
$L^1$-dissipativity property \eqref{eq:L1Dissipativity} for each
$\eps>0$, and that for any $u_0\in L^\infty$, $(u^\eps)_{\eps>0}$ is
relatively compact in $L^1_{\loc}$.

Using diagonal, density, and comparison arguments
(see Section \ref{ssec:VV-in-1D} below for details), at
the limit $\eps\to 0$ we arrive at an $L^1$-dissipative solver for
\eqref{eq:KruzhkovConsLaw}, defined for all $u_0\in L^\infty$.
By Remark \ref{rem:CharactSemigroups}, this implicitly defined {\it
``vanishing viscosity solver''} is characterized by some complete
maximal $L^1D$ germ  that we now want to determine.
Its dual germ $\mathcal{G}^*$ consists of all elementary solutions
\eqref{eq:PiecewiseC(x)} that are trajectories of this solver.
According to Gelfand \cite{Gelfand} (cf.~\cite{HoldenRisebro,GoritskyChechkin}),
for all $(u^l,u^r)$ satisfying
\begin{equation}\label{eq:GelfandGerm}
    \left|
    \begin{split}
        & \text{$f(u^l)=f(u^r)$ and either $u^l=u^r$, or}\\
        & \text{$\sign(u^r-u^l)(f(c)-f(u^l))>0$
        for all $c$ between $u^l$ and $u^r$},
    \end{split}\right.
\end{equation}
there exist standing-wave profiles $W:\R\to\R$ such that
$\lim_{\xi\to -\infty} W(\xi)=u^l$,  $\lim_{\xi\to \infty}
W(\xi)=u^r$, and $W(\frac x\eps)$ solves
\eqref{eq:GelfandViscosity} for all $\eps>0$.
Thus solutions of the form \eqref{eq:PiecewiseC(x)} with
$(u^l,u^r)$ as in \eqref{eq:GelfandGerm} are explicit limits, as
$\eps\to 0$, of solutions of \eqref{eq:GelfandViscosity}.
Therefore these solutions are trajectories of the vanishing
viscosity solver, and we can define the ``Gelfand'' germ
$\mathcal{G}_{\mathit{G}}$ as the set of all pairs $(u^l,u^r)$ satisfying
\eqref{eq:GelfandGerm}; $\mathcal{G}_{G}$ is $L^1D$ because the
vanishing viscosity solver is $L^1$ dissipative by construction.

Let us look at the closure $\overline{\mathcal{G}_G}$
of $\mathcal{G}_G$. One checks that
$\overline{\mathcal{G}_G}\supset \mathcal{G}'$, where $\mathcal{G}'$
is the set of all pairs $(u^l,u^r)$ such that
\eqref{eq:GelfandGerm} is satisfied with the
inequality ``$\geq 0$'' instead of ``$>0$''.
But $\mathcal{G}'$ is known to be
complete (see \cite{Gelfand,HoldenRisebro,GoritskyChechkin}),
therefore $\overline{\mathcal
G_G}$ is a complete germ. By
Proposition \ref{prop:Completeness}(ii), $\mathcal{G}_G$ is a
definite germ with complete closure; more exactly, $\mathcal
G_G^*=\overline{\mathcal{G}_G}=\mathcal{G}'$, because $\mathcal{G}'$
turns out to be an $L^1D$ germ. Therefore we can identify 
the ``vanishing viscosity'' germ with the 
unique maximal $L^1D$ extension of $\mathcal
G_G$, i.e. with its closure $\overline{\mathcal{G}_G}$.

Now $\mathcal{G}_G$- and $\mathcal{G}_{\mathit{VKr}}$-entropy solutions must coincide,
because $\mathcal{G}_{\mathit{VKr}}\subset \mathcal{G}_G$ (see
their definitions) and $\mathcal{G}_{\mathit{VKr}}$ was shown to be definite;
in addition, we infer that $\mathcal{G}_{\mathit{VKr}}^*$
coincides with $\overline{\mathcal{G}_G}$.
In conclusion, the germs $\mathcal{G}_{\mathit{VKr}}$ and $\mathcal{G}_G$
have different motivation, but they both
lead to the classical Kruzhkov entropy solutions.

\begin{rem}\label{rem:FantasqueSolvers}
In closing, we mention that non-classical $L^1$-dissipative
solvers exist for \eqref{eq:KruzhkovConsLaw}.
For example, for the Hopf-Burgers flux $f(u)=\frac{u^2}{2}$,
$\overline{\mathcal{G}_G}\neq \R \times \R$. Take any pair
$(u^l,u^r)\notin \overline{\mathcal{G}_G}$,
which satisfies the Rankine-Hugoniot condition, and then
set $\mathcal{G}''=\{(u^l,u^r)\}$. Take an arbitrary
maximal $L^1D$-extension $\mathcal{G}'''$ of $\mathcal{G}''$; the
associated $\mathcal{G}'''$-entropy  solutions are, in general,
different from the Kruzhkov entropy solutions! Notice that since
$(u^l,u^r)\notin \overline{\mathcal{G}_G}=\mathcal{G}_{\mathit{VKr}}^*$, some
of the constant elementary solutions in $\mathcal{G}_{\mathit{VKr}}$ will
not be trajectories for a non-classical solver. Finally, we recall
that physically motivated---but not $L^1$-dissipative---non-classical
solvers for, say, the nonconvex conservation law $u_t+(u^3)_x=0$ have been
extensively studied in the recent years, cf.~the
book of LeFloch \cite{LeFloch}.
\end{rem}

\subsection{Increasing surjective fluxes $f^{l,r}$}\label{ssec:BaitiJenssen}

Baiti and Jenssen \cite{BaitiJenssen} have proposed an entropy
formulation for \eqref{eq:ModelProb} with strictly
increasing and uniformly Lipschitz continuous fluxes
$f^l,f^r$ (strictly decreasing functions can be treated similarly).

In what follows, we assume that there is a
set $W\subset \R$ such that $f^{l,r}:U\mapsto W$ are
non-decreasing and onto. The corresponding germ $\mathcal{G}_{RH}$
consists of all pairs $(u^l,u^r)\in U \times U$ that satisfy
the Rankine-Hugoniot condition \eqref{eq:RankineHugoniot}.

Due to the monotonicity of $f^{l,r}$, \eqref{eq:L1D-States} holds for all
$(u^l,u^r),(\hat u^l,\hat u^r)\in \mathcal{G}_{RH}$.
Thus $\mathcal{G}_{RH}$ is an $L^1D$ germ; it is a maximal $L^1D$ germ because
$\mathcal{G}_{RH}$ admits no nontrivial extension. Finally,
$\mathcal{G}_{RH}$ is complete. Indeed, taking $(u_-,u_+)\in U \times U$, there
exists $u^r\in U$ such that $(u_-,u^r)\in \mathcal{G}_{RH}$; take
any $u^r \in (f^r)^{-1}(f^l(u_-))$.
Furthermore, the graph $\theta^r(\cdot,u^r)$ coincides with the
graph of $f^r$. Therefore $u_+ \in \dom \theta^r(\cdot,u^r)$, and the
Riemann problem \eqref{eq:ModelProb},\eqref{eq:RiemannInitialCond}
admits a weak solution of the form {\rm(RPb-sol.)}~with
the intermediate states $(u_-,u_+)\in\mathcal{G}_{RH}$.
Definition \ref{def:AdmIntegral} of $\mathcal{G}_{RH}$-entropy solutions
then turns into the definition of an admissible
solution in the sense of Baiti and Jenssen \cite{BaitiJenssen},
adapted to the model equation \eqref{eq:ModelProb}.

In passing, notice that a simpler equivalent choice of the
admissibility germ is possible.  Indeed, we can take
$\mathcal{G}_{\emptyset}=\emptyset$.  It possesses a unique maximal
$L^1D$ extension, namely
$\left(\mathcal{G}_{\emptyset}\right)^*=\mathcal{G}_{RH}$.  Thus
$\mathcal{G}_{\emptyset}$ is a definite germ; by Remark
\ref{rem:GandG*}, the $\mathcal{G}_{RH}$- and
$\mathcal{G}_{\emptyset}$-entropy solutions coincide.

The entropy dissipation property \eqref{eq:EntropySolDefi} of an
$\mathcal{G}_{\emptyset}$-entropy solution $u$ imposes no restriction on
the jump of $u$ across the flux discontinuity, so the only restriction is the
Rankine-Hugoniot condition: $\gamma^l f^l(u)=\gamma^r f^r(u)$
(which is equivalent to $(\gamma^l V^lu, \gamma^r V^ru)\in V\mathcal{G}_{RH}
=V(\mathcal{G}_{\emptyset})^*$ of Definition \ref{def:AdmWithTraces}(ii),
since $f^{l,r}$ are monotone).

\subsection{Other cases with monotone fluxes $f^{l,r}$}\label{ssec:AntiMonotone}
We have two different situations.
\begin{itemize}
\item If we assume that on $U\subset \R$,  $f^l$ is
non-decreasing and $f^r$ is non-increasing, then the germ $\mathcal{G}_{RH}$
of Section \ref{ssec:BaitiJenssen} is still a maximal $L^1D$ germ, but it
is not complete. Indeed, for all $u^{l,r}\in U$ the graphs
$\theta^{l,r}(\cdot,u^{l,r})$ are constant; they are the sections
of the graphs of $f^{l,r}$ at the levels $f^{l,r}(u^{l,r})$, respectively.
Therefore, unless the Riemann data $(u_-,u_+)$
belong to $\mathcal{G}_{RH}$, there exist no solution to
\eqref{eq:ModelProb},\eqref{eq:RiemannInitialCond} under the form
{\rm(RPb-sol.)}~with $(u^l,u^r)\in\mathcal{G}_{RH}$.

\item Consider the case where $f^l$ is non-decreasing and $f^r$ is
  non-increasing, on $U\subset \R$. If $f^l(U)\cap f^r(U)=\emptyset$,
  then the Rankine-Hugoniot condition \eqref{eq:RankineHugoniot} never
  holds, i.e., there exist no weak solution of \eqref{eq:ModelProb}.
\end{itemize}
Let us assume in addition that the ranges of $f^l$ and $f^r$ are
not disjoint. We claim that in this case, for any pair $(A,B)$
such that the Rankine-Hugoniot condition $f^l(A)=f^r(B)$ holds,
the germ $\mathcal{G}_{(A,B)}=\Set{(A,B)}$ is a definite complete
germ. The corresponding maximal $L^1D$ extension of
$\mathcal{G}_{(A,B)}$ is its closure $\overline{\mathcal{G}}_{(A,B)}$.
Indeed, it is evident that $\overline{\mathcal{G}}_{(A,B)}$ is an
$L^1D$ germ. If $\mathcal{G}$ is a maximal
$L^1D$ extension of $\mathcal{G}_{(A,B)}$, then
$(A,B)\in \mathcal{G}$. Notice that the graph
$\theta^{l}(\cdot,A)$ coincides with the graph of
$f^{l}$, and the graph $\theta^{r}(\cdot,B)$
coincides with the graph of $f^{r}$.
In particular, both graphs have $U$ for their domain.
It follows by Propositions \ref{PropIsolatedGermPoints}
and \ref{prop:Closure}(iii) that for all $(u_-,u_+)\in U \times U$,
either $(u_-,u_+)\notin \overline{\mathcal{G}}=\mathcal{G}$
or $(u_-,u_+)\in \overline{\mathcal{G}}_{(A,B)}$.
Thus $\mathcal{G}=\overline{\mathcal{G}}_{(A,B)}$.
Finally, for all $(u_-,u_+)\in U \times U$,  the Riemann problem
\eqref{eq:ModelProb},\eqref{eq:RiemannInitialCond} admits a weak
solution of the form {\rm(RPb-sol.)}~with $(u^l,u^r)=(A,B)$.

Notice that the germs $\mathcal{G}_{(A,B)}$, $\mathcal{G}_{(A',B')}$
described above have different maximal $L^1D$ extensions if and
only if $f^l(A)=f^r(B)\neq f^l(A')=f^r(B')$. Hence, whenever
$f^l(U)\cap f^r(U)$ contains more then one point, there exist
infinitely many two by two non-equivalent definitions of
$\mathcal{G}$-entropy solutions, each definition
being based on a complete definite admissibility germ.

\subsection{The Audusse-Perthame adapted entropies}\label{ssec:AudussePerthame}

Now we consider the case of two fluxes $f^{l,r}$ such that 
for some $u_o^{l,r}\in \R$, $f^{l,r}:(-\infty, u_o^{l,r}]\mapsto \R_+$ are bijective
decreasing functions and $f^{l,r}:[ u_o^{l,r},\infty)\mapsto \R_+$ are 
bijective increasing functions, cf.~Figure \ref{fig:audusse-perthame} for 
an example. This is the context of Audusse and Perthame in 
\cite{AudussePerthame}, reduced to our model case \eqref{eq:ModelProb}. 
The work \cite{AudussePerthame} applies in more 
generality; it was recently revisited by 
Panov \cite{Panov-AudussePerthameRevisited}, who showed
that a suitable change of variables reduces the 
Audusse and Perthame definition to the definition of 
Kruzhkov \cite{Kruzhkov}. Here we do not take 
advantage of Panov's observation.

\begin{figure}[htbp] %  figure placement: here, top, bottom, or page
\centering
\includegraphics[width=0.55\linewidth]{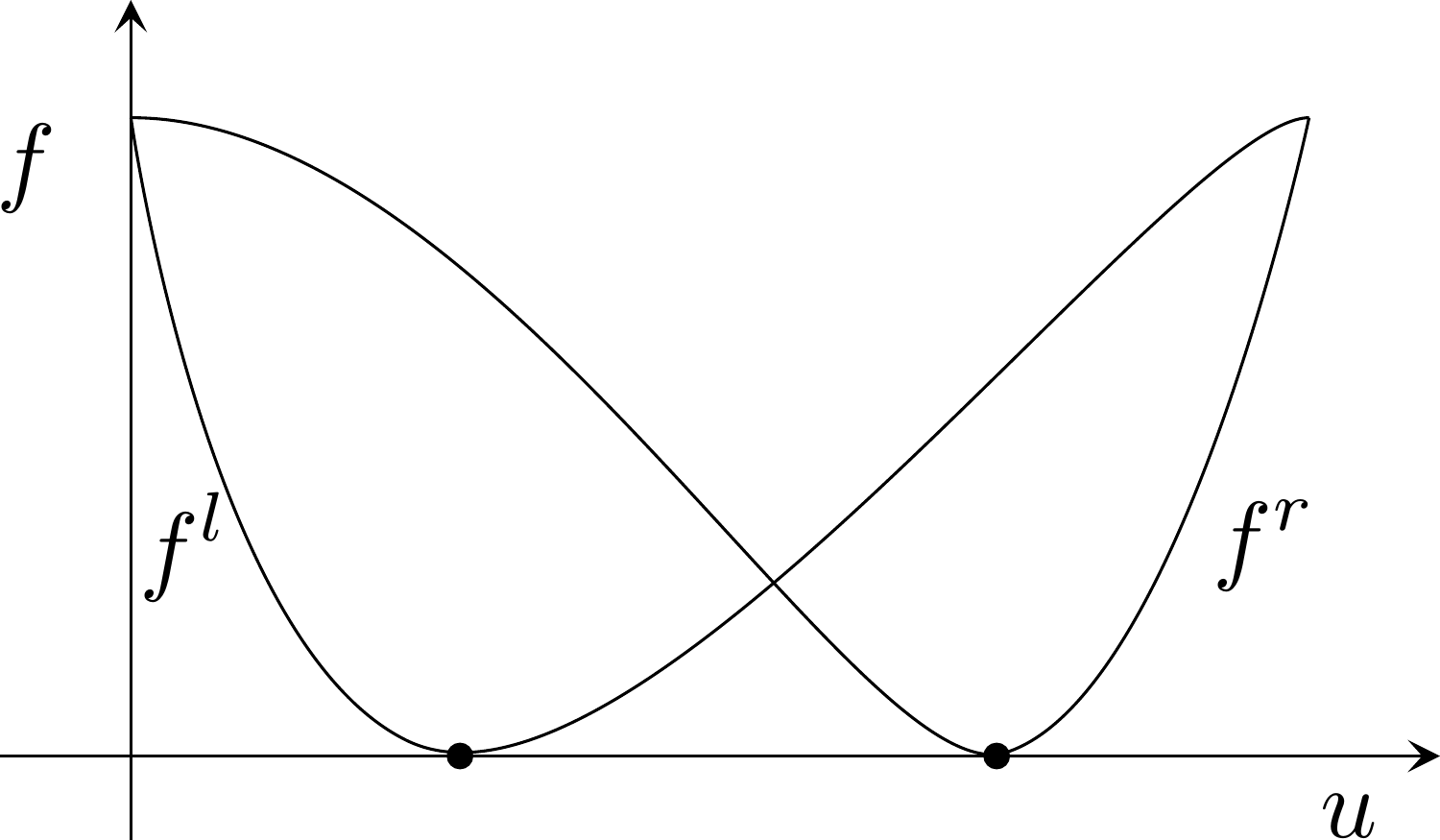}
\caption{Typical left and right fluxes covered by the
framework of Audusse and Perthame; Bold dots represent $u^{l,r}_o$.}
\label{fig:audusse-perthame}
\end{figure}

The notion of entropy solutions in \cite{AudussePerthame} is 
based on ``partially adapted entropies''. These
are, for our model case, entropies built from a step 
function $c(x)$ associated with a pair $(c^l,c^r)\in
\mathcal{G}_{AP}$: $z\mapsto \eta_{c(x)}(z)=\abs{z-c(x)}$. 
Here $\mathcal{G}_{AP}$ is the ``Audusse-Perthame" germ given by
\begin{equation}\label{eq:AudussePerthame-germ}
    (u^l,u^r)\in \mathcal{G}_{AP}  \Longleftrightarrow
    f^l(u^l)=f^r(u^r),\, \, \sign(u^l-u_o^l)=\sign(u^r-u_o^r).
\end{equation}

Definition \ref{def:AdmIntegral} (under the form
of Proposition \ref{prop:Carrillo-type-defs})
of $\mathcal{G}_{AP}$-entropy solutions reduces to
the definition of Audusse and Perthame, for which uniqueness
is shown in \cite{AudussePerthame,Panov-AudussePerthameRevisited}.
Thus it is not surprising that $\mathcal{G}_{AP}$ is an $L^1D$ germ;
property \eqref{eq:L1D-States} is checked directly from the definition
\eqref{eq:AudussePerthame-germ} of $\mathcal{G}_{AP}$.

\begin{prop}\label{lem:AP-germ-complete}
The germ $\mathcal{G}_{AP}$ is complete.
\end{prop}

\begin{proof}
Denote by $h^{l,r}_-$ the inverses of $f^{l,r}|_{(-\infty,u_o^{l,r}]}$;
denote by $h^{l,r}_+$ the inverses of $f^{l,r}|_{[u_o^{l,r},\infty)}$.
The domains of $\theta^{l,r}(\cdot,u^{l,r})$ are easy to calculate:
\begin{align*}
\dom \theta^l(\cdot,u^l)
&=
\begin{cases}
    \Bigl(-\infty,h^l_+(f^l(u^l))\Bigr],
    & \text{if $u^l\leq u_o^l$},\\
    \Set{u^l}, & \text{if $u^l>u_o^l$};
 \end{cases}
 \\
\dom \theta^r(\cdot,u^r)
&=
\begin{cases}
    \Set{u^r} , & \text{if $u^r<u_o^r$},\\
    \Bigl[h^r_-(f^r(u^r)),\infty\Bigr),
    & \text{if $u^r\geq u_o^r$}.
\end{cases}
\end{align*}
According to Remark \ref{rem:complete-caract}, we have to show that
$\R\times\R$ is a subset of
$$
\bigcup_{(u^l,u^r)\in \mathcal{G}_{AP}}
\dom \theta^l(\cdot,u^l) \times \dom\theta^r(\cdot,u^r).
$$
Indeed, take $(u_-,u_+)\in\R \times \R$. We have five cases:

\begin{itemize}
\item in the case $u_-\leq u^l_o$, $u_+\geq u^r_o$,
$(u_-,u_+)\in \dom  \theta^l(\cdot,u^l_o)
\times\dom \theta^r(\cdot,u^r_o)$, and $(u^l_o,u^r_o)\in
\mathcal{G}_{AP}$ (notice that $f^{l,r}(u^{l,r}_o)=0$, by the assumptions on $f^{l,r}$);
 
\item in the case $u_-\leq u^l_o$, $u_+ < u^r_o$, set $s=f^r(u_+)$;
then the pair $(u_-,u_+)$ belongs to $\dom 
\theta^l(\cdot,h^l_-(s)) \times \dom \theta^r(\cdot,u_+)$,
and $(h^l_-(s),u_+)\in \mathcal{G}_{AP}$;

\item the same is true if $u_-> u^l_o$, $u_+< u^r_o$ and $f^l(u_-)\leq f^r(u_+)=s$;

\item in the case $u_-> u^l_o$, $u_+ \geq u^r_o$, set $s=f^l(u_-)$;
then the pair $(u_-,u_+)$ belongs to
$\dom \theta^l(\cdot,u_-) \times\dom \theta^r(\cdot,h^r_+(s))$,
and $(u_-,h^r_+(s))\in \mathcal{G}_{AP}$;

\item the same is true if $u_-> u^l_o$, $u_+< u^r_o$ and
$s=f^l(u_-)\geq f^r(u_+)$.
\end{itemize}

Hence, $(u_-,u_+)\in  \dom \theta^l(\cdot,u^l)\times \dom \theta^r(\cdot,u^r)$
for some $(u^l,u^r)\in \mathcal G_{AP}$.
\end{proof}

Notice that Propositions \ref{lem:AP-germ-complete} and
\ref{prop:Completeness}(i) imply that the maximal
$L^1D$ extension $(\mathcal{G}_{AP})^*$ of the germ
$\mathcal{G}_{AP}$ coincides with its closure $\overline{\mathcal{G}}_{AP}$.
The above analysis gives
\begin{align*}
    \overline{\mathcal{G}}_{AP} & = \mathcal{G}_{AP}
    \cup \Set{(u^l,u^r) \, \Big | \, f^l(u^l)=f^r(u^r), u^l>u^l_o,u^r<u^r_o}
    \\ & \equiv \Set{(u^l,u^r)\, \Big | \, f^l(u^l)=f^r(u^r),\,
    \sign^-(u^l-u^l_o) \sign^+(u^r-u^r_o)\leq 0}.
\end{align*}

Now, consider the germ $\mathcal{G}_{(u^l_o,u^r_o)}:=\Set{(u^l_o,u^r_o)}$.
We find that $\left(\mathcal{G}_{(u^l_o,u^r_o)}\right)^*
=\overline{\mathcal{G}}_{AP}$, so
$\left(\mathcal{G}_{(u^l_o,u^r_o)}\right)^*$ is an $L^1D$ germ; by
Proposition \ref{prop:DualityOfGerms}(v), $\mathcal{G}_{(u^l_o,u^r_o)}$
is a definite germ.  By Remark \ref{rem:GandG*}(ii),
$\mathcal{G}_{(u^l_o,u^r_o)}$-entropy solutions
of \eqref{eq:ModelProb} coincide with
$\mathcal{G}_{AP}$-entropy solutions, thus they
coincide with entropy solutions in the sense of
Audusse and Perthame \cite{AudussePerthame}.
The completeness of $\left(\mathcal{G}_{(u^l_o,u^r_o)}\right)^*$ follows
from Proposition \ref{lem:AP-germ-complete}.
The existence of $\mathcal{G}_{(u^l_o,u^r_o)}$-entropy
solutions can be obtained from Theorem \ref{th:exist-by-numerics} or from the
adapted viscosity method of Theorem \ref{th:AB-connections-VV-1D-converges} in
Section \ref{sec:Existence} (see also
\cite{Panov-AudussePerthameRevisited}).

Moreover, different notions of entropy solution can be considered
for the same configuration of fluxes. Indeed, take any pair
$(A,B)\in (-\infty,u^l_o]\times[u^r_o,\infty)$ which fulfills the
Rankine-Hugoniot condition; set
$$
s_{(A,B)}:=f^l(A)=f^r(B).
$$
The germ $\mathcal{G}_{
(A,B)}:=\{(A,B)\}$ turns out to be definite, and its unique
maximal extension $\left(\mathcal{G}_{(A,B)}\right)^*$ is given by
\begin{equation}\label{eq:Modified-AP-germ}
    \begin{split}
        & (u^l,u^r)\in \mathcal (G_{(A,B)})^*
        \\ & \Longleftrightarrow  f^l(u^l)=f^r(u^r)\geq
        s_{(A,B)},\,
        \;\; \sign^-(u^l-A)\;\sign^+(u^r-B)\leq 0.
    \end{split}
\end{equation}
In the same way as in Proposition \ref{lem:AP-germ-complete}, we check
that $\left(\mathcal{G}_{(A,B)}\right)^*$ is complete.
Clearly, the choice $(A,B)=(u^l_o,u^r_o)$ discussed above is a
particular case. By analogy with the case  considered by
Adimurthi, Mishra, and Veerappa Gowda in
\cite{AdimurthiMishraGowda} (see also \cite{BurgerKarlsenTowers}
and Section \ref{ssec:Bell-shaped}), such pairs $(A,B)$ will be
called $(A,B)-$connections.

We easily check that \eqref{eq:Modified-AP-germ} accounts for 
all maximal $L^1D$ germs, for the fluxes
considered in this paragraph. The monotonicity 
condition on $f^{l,r}$ can be relaxed, allowing for non-strictly
monotone on $(-\infty, u^{l,r}_o]$ and on $[u^{l,r}_o,\infty)$ fluxes $f^{l,r}$.

We see that the definition of entropy solutions due to
Audusse and Perthame \cite{AudussePerthame},
as applied to the model case \eqref{eq:ModelProb}, corresponds to
\textit{one} among infinitely many choices  of $L^1$-contractive
semigroups of solutions. Each semigroup is most concisely
determined by a choice of  an $(A,B)-$connection, i.e., by a
singleton $\Set{(A,B)}$ from the set $(-\infty,u^l_o]\times[u^r_o,\infty)$;
any such singleton is a definite germ with complete dual.

\subsection{The Karlsen-Risebro-Towers entropy condition}\label{ssec:KRT-condition}
In \cite{Towers2000,Towers2001,KarlsenTowers,KarlsenRisebroTowers2003}, the
definition of an entropy solution is similar to Definition \ref{def:AdmIntegral}, with
\eqref{eq:EntropySolDefi} replaced by the inequality
\begin{equation}\label{eq:KRT-entropy-inequality}
    \begin{split}
        &\int_{\R_+} \int_\R
        \Bigl\{|u-c|\,\xi_t+\mathfrak{q}(x,u,c)
        \,\xi_x \Bigr\}
        \\ & \qquad -\int_{\R}\! |u_0-c|\,\xi(0,x)
        +\int_{\R_+} |f^r(c)-f^l(c)| \,\xi(t,0) \geq  0,
    \end{split}
\end{equation}
that should hold for all $c\in\R$ and for all
$\xi\in \mathcal{D}([0,\infty)\times \R)$, $\xi \geq 0$. Uniqueness for
\eqref{eq:ModelProb} under this admissibility condition is shown
for fluxes $f^{l,r}$ satisfying the
so-called \textit{crossing condition}, cf.~Figure \ref{fig:crossing}:
\begin{equation}\label{eq:BasicFluxesGoodCrossingCase}
    \text{$\exists u_\chi\in\R$ such that
    $\sign(z-u_\chi)(f^r(z)-f^l(z))\geq 0$ $\forall z\in U$.}
\end{equation}

 \begin{figure}[htbp] %  figure placement: here, top, bottom, or page
\centering
\begin{tabular}[h]{lr}
\includegraphics[width=0.4\linewidth]{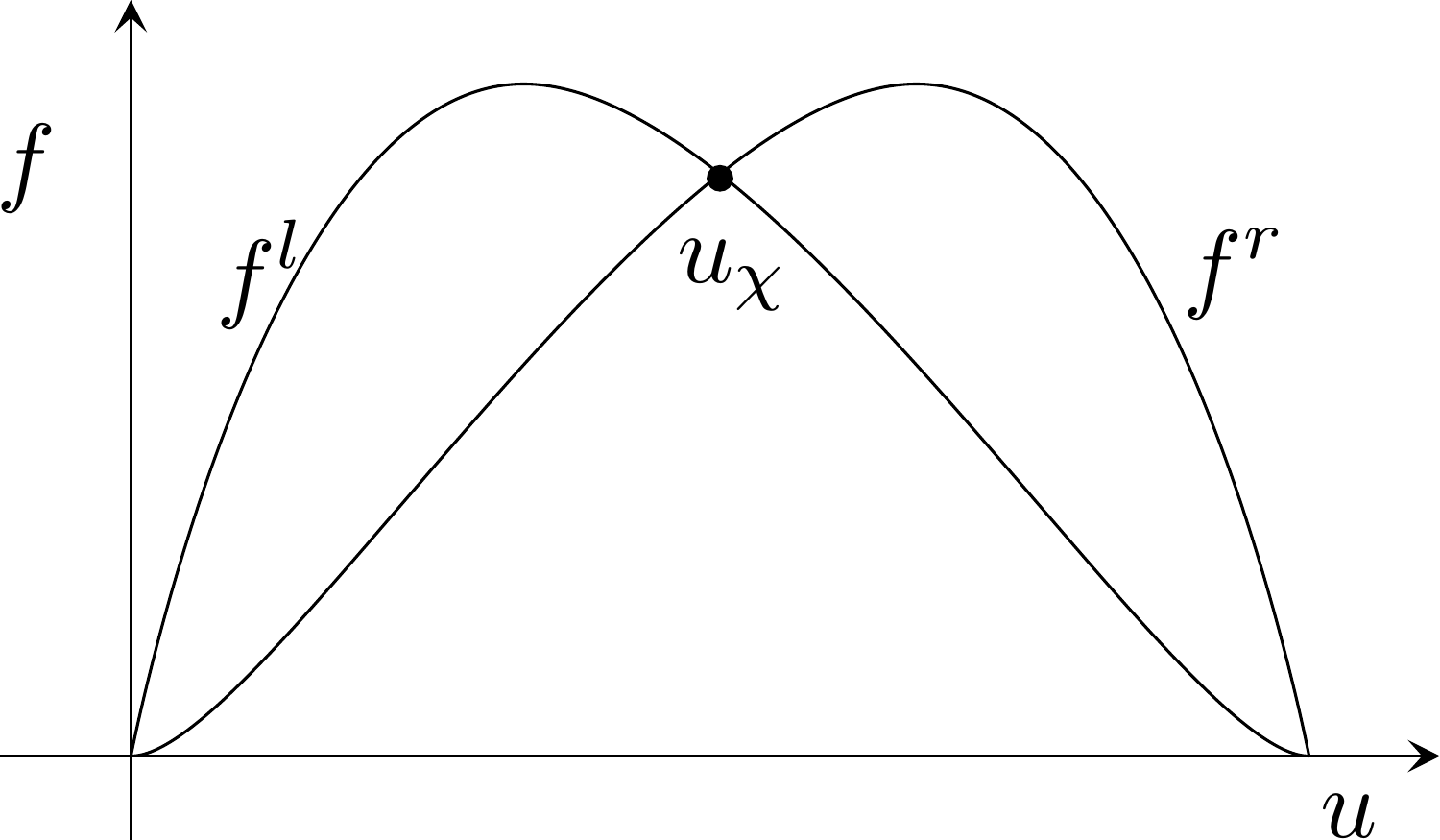}
&\includegraphics[width=0.4\linewidth]{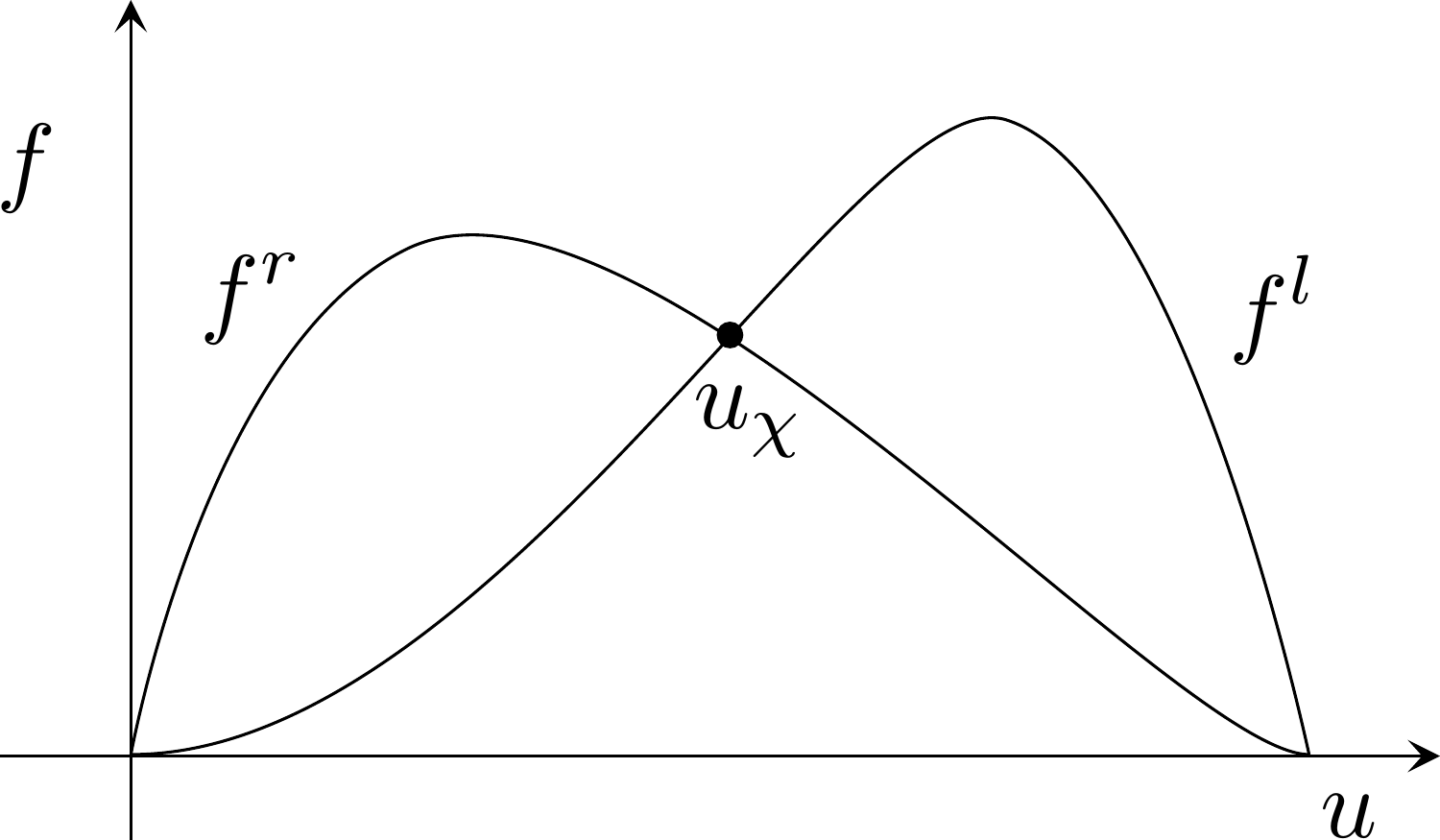}
\end{tabular}
\caption{Left: Crossing condition \eqref{eq:BasicFluxesGoodCrossingCase}
satisfied. Right: Crossing condition violated.}
\label{fig:crossing}
\end{figure}

Let us analyze the underlying germ. For the sake of simplicity,
assume that $f^{l,r}$ are non-constant on non-degenerate
intervals. Then the singular mappings $V^{l,r}$ are strictly
monotone. In this case, for all $u$ satisfying \eqref{eq:KRT-entropy-inequality}
there exist strong traces $t\mapsto (\gamma^{l,r} u)(t)$; in
addition, the Rankine-Hugoniot condition holds for these left and right traces.
As in the proof of Theorem \ref{th:DefLocal-Global}, for $h>0$ and $\xi\in
 \mathcal{D}([0,{\infty})\times\R)$, $\xi\geq 0$, we take
$\psi_h=\frac{(h-|x|)^+}{h}\xi$ as the test function in
\eqref{eq:KRT-entropy-inequality} and send $h\to 0$. The final result reads
\begin{equation}\label{eq:KRT-reduction}
     \int_{\R_+}
     \biggl( q^l\Bigl((\gamma^l u)(t),c\Bigr)-q^r\Bigl((\gamma^r u)(t),c\Bigr)
     + \abs{f^r(c)-f^l(c)} \biggr) \xi(t,0) \geq  0.
\end{equation}
Setting $u^{l,r}(t):=(\gamma^{l,r} u)(t)$, the inequalities
\eqref{eq:KRT-reduction} yield the condition
\begin{equation}\label{eq:KRT-reduction-2}
    \forall c\in \R,
    \quad
    q^l\Bigl((u^l)(t),c\Bigr) -q^r\Bigl((u^r)(t),c\Bigr) +\abs{f^r(c)-f^l(c)}\geq 0,
    \quad \text{for a.e.~$t>0$}.
\end{equation}

A case study shows that \eqref{eq:KRT-reduction-2}
defines the following germ:
\begin{equation}\label{EqKarlsenRisebroTowersGerm}
    \begin{split}
        &\text{$(u^l,u^r)\in \mathcal{G}_{KRT}$
        if and only if $f^l(u^l)=f^r(u^r)=:s$}
        \\ & \text{and} \quad
        \begin{cases}
        \text{either $u^l=u^r$}; \\
        \text{or $u^l<u^r$ and $\forall z\in[u^l,u^r]$,
        $\max\{f^l(z),f^r(z)\}\geq s$};\\
        \text{or $u^l>u^r$ and $\forall z\in [u^r,u^l]$,
        $\min\{f^l(z),f^r(z)\}\leq s$.}
    \end{cases}
    \end{split}
\end{equation}
This definition of $\mathcal{G}_{KRT}$ is valid also if
$f^{l,r}$ are allowed to degenerate (become constant on intervals).
The germ $\mathcal{G}_{KRT}$ is closed.

Proceeding as in the proof of Theorem \ref{th:DefLocal-Global}, we have
\begin{rem}\label{lem:KRT-def-equivalence}
A function $u$ is an entropy solution of \eqref{eq:ModelProb} in
the sense of \cite{KarlsenRisebroTowers2003} if and only if it is a
solution in the sense of Definition \ref{def:AdmWithTraces} with
$\mathcal{G}^*$ replaced by $\mathcal{G}_{KRT}$.
Thus, any function of the form {\rm(RPb-sol.)}~with
intermediate states $(u^l,u^r)\in\mathcal{G}_{KRT}$
is a solution of the Riemann problem
\eqref{eq:ModelProb},\eqref{eq:RiemannInitialCond}
in the sense of \cite{KarlsenRisebroTowers2003}.
\end{rem}

Let us show that the crossing condition
\eqref{eq:BasicFluxesGoodCrossingCase} is sufficient
for the germ $\mathcal{G}_{KRT}$ to be $L^1D$
and, moreover, maximal $L^1D$.

\begin{prop}\label{prop:G-KRT-L1D}~

\noindent (i) The germ $\mathcal{G}_{KRT}$ admits
no non-trivial $L^1D$ extension.

\noindent (ii) If the fluxes $f^{l,r}$ satisfy the crossing
condition \eqref{eq:BasicFluxesGoodCrossingCase}, then the
germ $\mathcal{G}_{KRT}$ is $L^1D$
(and therefore it is maximal $L^1D$).
\end{prop}

\begin{proof}
(i) Assume that there exists
$(\hat u^l,\hat u^r)\notin \mathcal{G}_{KRT}$ such
that $\mathcal{G}_{KRT}\cup\Set{(\hat u^l,\hat u^r)}$
is $L^1D$. We can assume that $\hat u^l<\hat u^r$ and that there exists
$z\in (\hat u^l,\hat u^r)$ such that both $f^l(z)$ and $f^r(z)$ are
strictly smaller than $\hat s:=f^{l,r}(\hat u^{l,r})$ (the other case is symmetric).
We claim the
\begin{equation}\label{eq:RKT-is-maximal}
    \text{$\exists (u^l,u^r)\in \mathcal{G}_{KRT}$
    such that $\hat u^l<u^l\leq u^r<\hat u^r$ and $f^{l,r}(u^{l,r})<\hat s$.}
 \end{equation}
Indeed, we can assume that $f^l(z)\leq f^r(z)<\hat s$ (the other
case is symmetric). We set $\bar s:=f^r(z)$ and examine the three possibilities.
\begin{itemize}

\item[--] Either  $f^{l}, f^{r}$ have a crossing point $u_\chi\in [\hat u^l,z]$
such that $f^{l,r}(u_\chi)\leq \bar s$. Then it is clear
that $\hat u^l<u_\chi\leq z<\hat u^r$, so
that $u^{l,r}:=u_\chi$ satisfies\eqref{eq:RKT-is-maximal}.

\item[--] Or $f^r>f^l$ on $[\hat u^l,z]$.
Set $u^{l,r}:=\inf\Set{\zeta\in [\hat u^l,z]\,\Big | \, f^{l,r}(\zeta)=\bar s}$.
Since $f^r(z)=\bar s$, $u^r$ is well defined and $u^r\leq z<\hat u^r$.
Since $f^l(u^l)=\hat s> \bar s\geq f^l(z)$, $u^l$ is well
defined and $\hat u^l<u^l$. Because $f^r>f^l$ and by the choice of $u^{l,r}$,
we have $u^l<u^r$, $f^r(\zeta)\geq \bar s=f^{l,r}(u^{l,r})$
$\forall \zeta\in [u^l,u^r]$. Thus $(u^l,u^r)\in \mathcal{G}_{KRT}$
and \eqref{eq:RKT-is-maximal} holds.

\item[--] Or else, $f^{l,r}$ have a crossing point $u_\chi\in [\hat u^l,z]$ 
with $f^{l,r}(u_\chi)> \bar s$. Then we choose the closest to 
$z$ crossing point $u_\chi$ with the above property, and set
$u^{l,r}:=\Set{\zeta\in [u_\chi,z]\, \Big | \, f^{l,r}(\zeta)=\bar s}$.
As in the previous case, we check that $u^{l,r}$ are well
defined, moreover, $u^l<u^r$ and  $f^r(\zeta)\geq \bar s=f^{l,r}(u^{l,r})$
for all $\zeta\in [u^l,u^r]$. Thus, $(u^l,u^r)\in \mathcal{G}_{KRT}$
and \eqref{eq:RKT-is-maximal} holds.
\end{itemize}

Now the pair $(\hat u^l,\hat u^r)$, and the pair $( u^l, u^r)$ which
satisfies \eqref{eq:RKT-is-maximal}), both violate \eqref{eq:L1D-States}; thus
$\mathcal{G}_{KRT}\cup\Set{(\hat u^l,\hat u^r)}$ is not $L^1D$.
This contradiction shows that $\mathcal{G}_{KRT}$ has
no non-trivial $L^1D$ extension.

\noindent (ii)  Let $(u^l,u^r)$, $(\hat u^l,\hat u^r)$ be two
pairs satisfying the Rankine-Hugoniot condition; set
$s:=f^{l,r}(u^{l,r})$, $\hat s:=f^{l,r}(\hat u^{l,r})$.
Upon exchanging the roles of $u^{l,r}$ and $\hat u^{l,r}$, we
easily find that \eqref{eq:L1D-States} fails if and
only if we have simultaneously $s<\hat s$, $u^l>\hat u^l$,
and $u^r<\hat u^r$. These conditions are
realized if and only if $s<\hat s$ and
\begin{equation}\label{eq:six-cases}
    \begin{split}
        & \text{either $\hat u^l<u^l\leq u^r<\hat u^r$,
        or $\hat u^l\leq u^r\leq u^l\leq \hat u^r$,
        or $u^r\leq \hat u^l<u^l\leq \hat u^r$},
        \\
        & \text{or $\hat u^l\leq u^r<\hat u^r\leq u^l$, or
        $u^r\leq \hat u^l\leq \hat u^r\leq u^l$,
        or $u^r<\hat u^r\leq \hat u^l<u^l$.}
    \end{split}
\end{equation}
We check that in all these cases except for the first one and the 
last one, \eqref{EqKarlsenRisebroTowersGerm}
never holds for both $(u^l,u^r)$ and $(\hat u^l,\hat u^r)$. 
In the two cases that were excepted, 
\eqref{EqKarlsenRisebroTowersGerm} can hold for both $(u^l,u^r)$, 
$(\hat u^l,\hat u^r)$ only if the crossing 
condition \eqref{eq:BasicFluxesGoodCrossingCase} fails.

This shows that if \eqref{eq:BasicFluxesGoodCrossingCase} holds, 
then \eqref{eq:L1D-States} holds for all 
$(u^l,u^r), (\hat u^l,\hat u^r)\in\mathcal{G}_{KRT}$. 
We conclude that $\mathcal{G}_{KRT}$ is 
an $L^1D$ germ; by (i), it is maximal $L^1D$.
\end{proof}

Now let us look at one case where the crossing condition
\eqref{eq:BasicFluxesGoodCrossingCase} fails. Consider the fluxes in
Section \ref{ssec:AudussePerthame} (i.e., of the type in
Figure~\ref{fig:audusse-perthame}) with $u^l_o< u^r_o$. We claim that
$\mathcal{G}_{KRT}$ is not an $L^1D$ germ. Indeed, with the notation
of Section \ref{ssec:AudussePerthame}, let
$$
\delta=\inf\Set{s\, \Big |\, 
\exists\, u\in(u_o^l,u_o^r)\; \text{such that}\, f^l(u)=f^r(u)=s};
$$
then for all $s\in(0,\delta]$, the pair $(u^l,u^r):=(h^l_-(s),h^r_+(s))$
belongs to $\mathcal{G}_{KRT}$. Also $(\hat u^l,\hat u^r):=(u^l_o,u^r_o)
\in \mathcal{G}_{KRT}$, but then \eqref{eq:L1D-States} is violated.
Another way to prove that $\mathcal{G}_{KRT}$ is 
not $L^1D$ is to notice that the definite germ 
$\mathcal{G}_{(u^l_o,u^r_o)}=\{(u^l_o,u^r_o)\}$ is contained
within $\mathcal{G}_{KRT}$, but, according 
to \eqref{eq:Modified-AP-germ}, 
$(u^l,u^r)$ does not belong to the unique maximal $L^1D$ extension 
$\left(\mathcal{G}_{(u^l_o,u^r_o)}\right)^*$ 
of $\mathcal{G}_{(u^l_o,u^r_o)}$.

Based on this fact, we can construct a non-uniqueness example
for one simple case where the crossing
condition \eqref{eq:BasicFluxesGoodCrossingCase} fails.

\begin{example}\label{example:KRT-non-uniq}
  Consider the fluxes $f^{l,r}:z\mapsto |z-u^{l,r}_o|$ with $u^l_o<
  u^r_o$. Then the entropy condition of
  \cite{KarlsenRisebroTowers2003} is not sufficient to single out a
  unique solution to the Riemann problem
  \eqref{eq:ModelProb},\eqref{eq:RiemannInitialCond} with
  $(u_-,u_+)=(u^l_0-s,u^r_0+s)$, for all $s$ in
  $\left(0,\frac12(u^r_o-u^l_o)\right]$.  Indeed, the result follows
  by the above analysis and by Remark
  \ref{lem:KRT-def-equivalence}. We have $(u_-,u_+)\in
  \mathcal{G}_{KRT}$; thus the function
  $u(t,x)=u_-\,\char_{\Set{x<0}}+u_+\,\char_{\Set{x>0}}$ is one
  solution. Because $u^{l,r}\in \dom \theta^{l,r}(\cdot,u^{l,r}_o)$,
  there exists another solution of the form {\rm(RPb-sol.)}~with
  intermediate states $(u^l_o,u^r_o)\in \mathcal{G}_{KRT}$.
\end{example}

In contrast, if we consider the fluxes as in
Section \ref{ssec:AudussePerthame} with $u^l_o>u^r_o$ such that the
crossing condition \eqref{eq:BasicFluxesGoodCrossingCase} holds,
then there exists a unique crossing point $u_\chi$ of $f^l$ and
$f^r$, and $u_\chi\in (u^r_o,u^l_o)$. In this case, one easily
checks that $\mathcal{G}_{KRT}$ coincides with the complete maximal
$L^1D$ germ $\left((\mathcal{G}_{(u_\chi,u_\chi)}\right)^*$.
Another very similar example is given
in Section \ref{ssec:Bell-shaped}.

Let us also look at the configuration of fluxes considered in
Section \ref{ssec:BaitiJenssen}. In this case, irrespective of the
crossing condition \eqref{eq:BasicFluxesGoodCrossingCase}, we
always have $\mathcal{G}_{KRT}=\mathcal{G}_{RH}$; thus $\mathcal
G_{KRT}$ is a complete maximal $L^1D$ germ. Also in the case
$f^l\equiv f^r$, $\mathcal{G}_{KRT}$ is a complete maximal $L^1D$
germ, because it coincides with the Volpert-Kruzhkov germ
$\mathcal{G}_{\mathit{VKr}}$ (here, the crossing condition is satisfied).

Finally, we point out that the germ $\mathcal{G}_{KRT}$ is
closely related to the ``vanishing viscosity'' germ
$\mathcal{G}_{VV}$ which we consider in Section \ref{sec:SVVGerm};
more precisely, we prove there that
${\mathcal{G}}_{VV}=\mathcal{G}_{KRT}$ whenever the crossing condition
\eqref{eq:BasicFluxesGoodCrossingCase} holds.

In conclusion, the definition \eqref{eq:KRT-entropy-inequality} of 
an entropy solution given in \cite{KarlsenRisebroTowers2003} can be 
equivalently reformulated in terms of the 
germ $\mathcal{G}_{KRT}$ (see Remark \ref{lem:KRT-def-equivalence}).
Whenever the crossing condition \eqref{eq:BasicFluxesGoodCrossingCase} holds, the
germ $\mathcal{G}_{KRT}$ is maximal
$L^1D$ (see Proposition \ref{prop:G-KRT-L1D}) and the unique
entropy solution in the sense \eqref{eq:KRT-entropy-inequality}
coincides with the unique $\mathcal{G}_{KRT}$-entropy
solution (see \cite{KarlsenRisebroTowers2003} or
Theorem \ref{theo:LocalESUniqueness} for the uniqueness proof).
For general fluxes, uniqueness may fail, as in
Example \ref{example:KRT-non-uniq}; nevertheless, the crossing
condition \eqref{eq:BasicFluxesGoodCrossingCase} is not necessary
for the uniqueness of $\mathcal{G}_{KRT}$-entropy solutions.

\subsection{Bell-shaped fluxes}\label{ssec:Bell-shaped}
Many works have been devoted to the case of bell-shaped fluxes.
To fix the ideas,
\begin{equation}\label{eq:Bell-shaped-fluxes}
    \left|
    \begin{split}
        &\text{let $U=[0,1]$; assume $f^{l,r}:U\!\mapsto \R_+$
        are such that $f^{l,r}(0)=0=f^{l,r}(1)$},\\
        & \text{$f^{l,r}$ are non-de\-cre\-asing on $[0,u_o^{l,r}]$
        and non-increasing on $[u_o^{l,r},1]$},\\
        & \text{the graphs of $f^l,f^r$ have at most one
        crossing point with abscissa in $(0,1)$},\\
        & \text{and if there is a crossing point $u_\chi$, then it
        lies between the points $u_o^{l},u_o^r$}.
    \end{split}\right.
\end{equation}

We refer to \cite[Figure 1.1]{BurgerKarlsenTowers} (cf.~also Figure \ref{fig:crossing}
in the present paper) for a catalogue of different configurations.
More general cases with multiple crossing points
and oscillating fluxes are considered in \cite{MishraThesis}.

Here the situation is quite similar to the one described in
Section \ref{ssec:AudussePerthame} (we just have to take into account
the fact that $u^{l,r}_o$ are now points of maximum (not minimum)
for $f^{l,r}$. There exist infinitely many complete maximal
$L^1D$ germs. Each of these germs contains one and only one pair
$(A,B)\in [u^l_o,1]\times[0,u^r_o]$ (cf.~Proposition \ref{PropIsolatedGermPoints}).
The germ is thus uniquely determined by a choice of the connection $(A,B)$ (see Adimurthi,
Mishra, Veerappa Gowda~\cite{AdimurthiMishraGowda} and B\"urger,
Karlsen, Towers~\cite{BurgerKarlsenTowers}); each maximal $L^1D$ germ
is the dual $\left(\mathcal{G}_{(A,B)}\right)^*$ of the definite germ
$$
\mathcal{G}_{(A,B)}
:=\Set{\text{$(A,B)$: $A\in [u^l_o,1]$, $B\in [0,u^r_o]$, and
$f^l(A)=f^r(B)=:s_{(A,B)}$}.}
$$
We have
\begin{equation}\label{eq:AB-germ}
    \begin{split}
        &(u^l,u^r)\in \left(\mathcal{G}_{(A,B)}\right)^*
        \Longleftrightarrow \\
        & f^l(u^l)=f^r(u^r)\leq s_{(A,B)},\, \,
        \sign^+(u^l-A)\; \sign^-(u^r-B)\leq 0.
    \end{split}
\end{equation}

As in Proposition \ref{lem:AP-germ-complete}, one shows that
$(G_{(A,B)})^*$ is complete. Another way is to
use Proposition \ref{prop:ExistenceAndCompleteness} together with
the existence result of \cite{BurgerKarlsenTowers}.
Indeed, in view of Proposition \ref{prop:Carrillo-type-defs}, it is clear that
Definition \ref{def:AdmIntegral} of $G_{(A,B)}$-entropy
solution is equivalent to the definition given by
B\"urger, Karlsen, Towers in \cite{BurgerKarlsenTowers}, which
requires that inequality \eqref{eq:EntropySolDefi} holds

--- for all $\xi\in {\mathcal{D}}([0,{\infty})\times \R)$, $\xi \geq 0$,
such that $\xi(0,t)=0$; with $(c^l,c^r)$ arbitrary;

--- for all $\xi\in {\mathcal{D}}([0,{\infty})\times \R)$, $\xi \geq 0$, and
$(c^l,c^r)=(A,B)$.

The  technical genuine nonlinearity restriction on $f^{l,r}$
imposed in \cite{BurgerKarlsenTowers} can be
bypassed\footnote{In works on conservation laws with interface
coupling, such as the discontinuous flux problems, there are
two reasons to ask for the genuine nonlinearity of the fluxes.
Firstly, this assumption simplifies the manipulation of the interface
fluxes (cf.~Remarks \ref{RemFunctionQ}, \ref{RemTracesForEntropyfluxes}, where
these restrictions are bypassed). Secondly, this assumption yields strong
compactness properties (see \cite{LionsPerthameTadmor}
and \cite{Panov-precomp-first,Panov-precomp-ARMA}).
In \cite{BurgerKarlsenTowers}, the compactness is achieved
with the help of local $BV$ estimates; thus the nonlinearity
assumption on $f^{l,r}$ can be dropped.}.
In Section \ref{sec:Existence}, we give alternative proofs of existence of
$G_{(A,B)}$-entropy solutions, using an adapted viscosity
approach and a specific numerical scheme.

Let us point out three specific choices of a connection $(A,B)$.
The first one is $(A,B)=(1,0)$, which corresponds to the extremal choice
$s_{(A,B)}=0$.  Furthermore, we have that
$\mathcal{G}_{(1,0)}=\{(1,0),(0,0),(1,1)\}$ (this example
is mentioned, e.g., by Panov in \cite{Panov-AudussePerthameRevisited}).
The other extremal choice is to pick $(A,B)$
corresponding to the maximum of $s_{(A,B)}$; in
this case, we have either $A=u_o^l$ (whenever $f^l(u_o^l)\leq f^r(u_o^r)$)
or $B=u_o^r$ (whenever $f^l(u_o^l)\geq f^r(u_o^r)$)\footnotemark.
This is the germ identified by Kaasschietter \cite{Kaasschietter}, Ostrov \cite{Ostrov},
and Adimurthi and Veerappa Gowda \cite{AdimurthiGowda}; the
corresponding solutions (which are $\mathcal{G}_{(A,B)}$-entropy
solutions in our vocabulary) have been constructed as
limits of various physically relevant regularizations. In the
case where $f^l$ and $f^r$ have no crossing point with abscissa in
$(0,1)$, the definition of Bachmann and Vovelle
\cite{BachmannVovelle} also leads to this germ, as well as the one
of Karlsen, Risebro, Towers \cite{KarlsenRisebroTowers2003} (see
Section \ref{ssec:KRT-condition}).

Finally, consider the case where the crossing condition
\eqref{eq:BasicFluxesGoodCrossingCase} holds with $u_\chi\in(0,1)$.
Then $(u_\chi,u_\chi)\in[u^l_o,1]\times[0,u^r_o]$.
The germ $\mathcal{G}_{(u_\chi,u_\chi)}$ has a unique
maximal $L^1D$ extension $\left(\mathcal{G}_{(u_\chi,u_\chi)}\right)^*$
which coincides with the germ $\mathcal{G}_{KRT}$
described in Section \ref{ssec:KRT-condition}. As it
was demonstrated by Adimurthi, Mishra, Veerappa Gowda in
\cite{AdimurthiMishraGowda}, this germ also comes out of the
vanishing viscosity approaches of Gimse and Risebro
\cite{GimseRisebro1,GimseRisebro2} and Diehl
\cite{Diehl1,Diehl2,Diehl3}. Section \ref{sec:SVVGerm} below is
devoted to a detailed study of the ``vanishing viscosity''
germ(s). \footnotetext{We autorize $f^{l,r}$ to be constant on
non-degenerate intervals, therefore $u^{l,r}_o$ are not uniquely
defined. But all the germs $\mathcal{G}_{(u^l_o,B)}$
(or $\mathcal{G}_{(A,u^r_o)}$) corresponding
to different choices of $u^{l,r}_o$ have the same closure.}

\subsection{The case $f^l\equiv f^r$ and the
Colombo-Goatin germs}\label{ssec:ColomboGoatin}

In \cite{ColomboGoatin}, Colombo and Goatin consider the
conservation law of the form \eqref{eq:ModelProb}
with $f^l=f=f^r$, which is a bell-shaped function in the sense of
Section \ref{ssec:Bell-shaped}; for some $F\in [0,\max f]$,  the
restriction on the value of the flux, $|f(u(t,0))|\leq F$, is imposed.
Starting from a regularization of the constraint and using the entropy formulation
of Karlsen, Risebro, Towers \cite{KarlsenRisebroTowers2003}, the
authors obtain a definition of entropy solution which leads to
well-posedness in the $BV$ framework.

Analyzing the definition of \cite{ColomboGoatin}, we easily determine
the pairs $(u^l,u^r)$ for which the elementary solution \eqref{eq:PiecewiseC(x)}
is an entropy solution in the sense of \cite{ColomboGoatin}.
The set of all these pairs is in fact a maximal $L^1D$ germ; moreover,
this germ is the unique maximal extension
$\left(\mathcal{G}_{CG}^F\right)^*$ of the definite
``Colombo-Goatin" germ $\mathcal{G}_{CG}^F$:
$$
\mathcal{G}_{CG}^F
=\Set{\text{$( A,B )$: $A$, $B$ are determined by $A\leq B$, $f(A)=f(B)=F$}}.
$$
This can be seen as  an $(A,B)$-connection
of \cite{BurgerKarlsenTowers} and Section \ref{ssec:Bell-shaped}, except
that we are in the particular ``degenerate'' situation with $f^l\equiv f^r$.

In fact, in \cite{ColomboGoatin} the constraint $F$ depends on
$t$. This case can be included in our consideration. We refer to
Remark \ref{rem:CharactSemigroups} and
the forthcoming papers \cite{AKR-II,AndrGoatinSeguin}.

\section{The vanishing viscosity germ}\label{sec:SVVGerm}

There are many possibilities for introducing a vanishing viscosity
regularization of \eqref{eq:ModelProb}. Gimse and Risebro
\cite{GimseRisebro1,GimseRisebro2} recast \eqref{eq:ModelProb} as
a $2\times 2$ system and add the simplest viscosity terms to both
equations. Diehl \cite{Diehl1,Diehl2,Diehl3} introduces both the
viscosity $\eps u_{xx}$ and a smoothing $\mathfrak f^\delta$ of
the discontinuity in $\mathfrak f$; the same approach is pursued
in \cite{BurgerKarlsenTowers}. The regularization term of
Kaasschietter \cite{Kaasschietter} is more involved, making
appear a capillarity pressure function; the admissibility condition
he obtains is different from the ones
of \cite{GimseRisebro1,GimseRisebro2,Diehl1,Diehl2,Diehl3}.

We shall first study the pure vanishing viscosity regularization
of \eqref{eq:ModelProb}, i.e.,
\begin{equation}\label{eq:ModelProb-Visco}
    u_t + \mathfrak{f}(x,u)_x=\eps u_{xx}, \qquad
    \mathfrak{f}(x,z)=
    \begin{cases}
        f^l(z), & x<0,\\
        f^r(z), & x>0.
    \end{cases}
\end{equation}
We pursue the traveling-wave approach of Gelfand \cite{Gelfand}
(cf.~Section \ref{ssec:Gelfand}), which is, in our context, a ``standing-wave'' approach.
Given $f^{l,r}$, define the ``standing-wave vanishing viscosity" germ by
\begin{equation}\label{eq:VV-germ}
    (u^l,u^r)\in \mathcal{G}^{s}_{VV}
    \;
    \Longleftrightarrow
    \;
    \left\{\begin{split}
        &\text{there exists a function $W:\R\to\R$ such that}\\
        &\text{$\lim_{\xi\to -\infty} W(\xi)=u^l$,\; $\lim_{\xi\to \infty} W(\xi)=u^r$, and}\\
        &\text{$u^\eps(t,x)=W(x/\eps)$ solves
        \eqref{eq:ModelProb-Visco} in $ \mathcal{D}'(\R_+\times\R)$.}
    \end{split}\right.
\end{equation}
Clearly, if $(u^l,u^r)\in \mathcal{G}^{s}_{VV}$,
then the elementary stationary solution \eqref{eq:PiecewiseC(x)}
can be obtained as the almost everywhere and $L^1_{\loc}$ limit of a sequence of
solutions of the regularized equation \eqref{eq:ModelProb-Visco}.

It turns out that $\mathcal{G}^{s}_{VV}$ can be replaced by a
smaller germ which is described explicitly;
moreover, we show that $\mathcal{G}^{s}_{VV}$ is definite.
The proposition below leads to a definition of the maximal 
$L^1D$ germ $\mathcal{G}_{VV}$, which has already 
appeared in the work of Diehl \cite{Diehl2008}.

\begin{prop}\label{prop:G-VV-versus-G-VV-0}
Consider the germ $\mathcal{G}^{o}_{VV}$ defined
as the set of all the pairs
$(u^l,u^r)$ such that
\begin{equation}\label{eq:ViscGerm}
    \begin{split}
        & f^l(u^l)=f^r(u^r)=:s, \;\; \text{and} \\
        &
        \begin{cases}
            \text{either $u^l=u^r$}; & \\
            \text{or $u^l<u^r$ and} & \text{either $f^l(z)>s$
            for all $z\in (u^l,u^r]$},\\ &
            \text{or $f^r(z)>s$ for all $z\in [u^l,u^r)$};\\[0.5pt]
            \text{or $u^l>u^r$ and} & \text{either $f^l(z)<s$
            for all $z\in [u^r,u^l)$},\\ &
            \text{or $f^r(z)<s$ for all $z\in (u^r,u^l]$}.
        \end{cases}
    \end{split}
\end{equation}

(i) The germ $\mathcal{G}^{o}_{VV}$ is $L^1D$.\\

(ii) The closure $\overline{\mathcal{G}^{o}_{VV}}$ of $\mathcal
G^{o}_{VV}$ is a maximal $L^1D$ germ, and
$(u^l,u^r)\in \overline{\mathcal{G}^{o}_{VV}} \Longleftrightarrow$
\begin{equation}\label{eq:ClosureViscGerm}
    \begin{split}
        & f^l(u^l)=f^r(u^r)=:s, \;\; \text{and} \\
        &
        \begin{cases}
            \text{either $u^l=u^r$}; & \\
            \text{or $u^l<u^r$ and} & \text{there exists $u^o\in[u^l,u^r]$}\\ &
            \text{such that $f^l(z)\geq s$ for all $z\in [u^l,u^o]$} \\ &
            \text{and $f^r(z)\geq s$ for all $z\in [u^o,u^r]$};\\
            \text{or $u^l>u^r$ and} & \text{there exists $u^o\in[u^r,u^l]$}\\&
            \text{such that $f^r(z)\leq s$ for all $z\in[u^r,u^o]$}\\ &
            \text{and $f^l(z)\leq s$ for all $z\in [u^o,u^l]$}.
        \end{cases}
    \end{split}
\end{equation}

(iii) We have $\mathcal{G}^{o}_{VV}\subset
\mathcal{G}^{s}_{VV}\subset \overline{\mathcal{G}^{o}_{VV}}$.\\

(iv) Both $\mathcal{G}^{s}_{VV}$ and $\mathcal
G^{o}_{VV}$ are definite germs with the same
unique maximal $L^1D$ extension $\overline{\mathcal{G}^{o}_{VV}}$.
The $\mathcal{G}^{s}_{VV}$, $\mathcal{G}^{o}_{VV}$,
and $\overline{\mathcal{G}^{o}_{VV}}$ entropy solutions coincide.
\end{prop}

\begin{defi}\label{def:VV-germ}
In the sequel, we denote by $\mathcal{G}_{VV}$ the maximal $L^1D$
germ $\overline{\mathcal{G}^{o}_{VV}}$ given by
\eqref{eq:ClosureViscGerm}.  We call it the ``vanishing viscosity'' germ
associated with the pair of functions $f^{l,r}$.
\end{defi}

Whenever the dependence on $f^{l,r}$ should be indicated
explicitly, we denote the corresponding vanishing viscosity germ
by $\mathcal{G}_{VV}^{f^{l,r}}$. In \cite{AKR-II}, we will study
more carefully the dependence of $\mathcal{G}_{VV}^{f^{l,r}}$ on  $f^{l,r}$.

\begin{rem}\label{rem:Diehl-formulation}~\\
  \noindent (i) Definition~\ref{def:VV-germ} allows for general
  continuous fluxes $f^{l,r}$. Yet it is clear that the so defined
  germ $\mathcal G_{VV}$ need not be complete. For instance, if the
  ranges of $f^{l,r}$ do not intersect, $\mathcal G_{VV}$ is empty and
  no $\mathcal G_{VV}$ entropy solution exists.

\noindent (ii) Basing themselves on different viscosity approaches,
Gimse and Risebro, and then Diehl, have implicitly introduced the set
of all admissible elementary solutions of \eqref{eq:ModelProb}. These
are the elementary solutions used to construct a Riemann solver based
on the ``minimal jump condition'' \cite{GimseRisebro1,GimseRisebro2}
or on the ``$\Gamma$-condition'' \cite{Diehl1,Diehl2,Diehl3}. In fact,
both approaches lead to the same germ which coincides with
$\mathcal{G}_{VV}$; the closure operation accounts for the difference
between the minimal jump condition and the $\Gamma$-condition. For
certain configurations of the fluxes $f^{l,r}$, a Riemann solver was
defined for all data in
\cite{GimseRisebro1,GimseRisebro2,Diehl1,Diehl2,Diehl3,Diehl2008}. In
all these cases we conclude that $\mathcal{G}_{VV}$ is complete.

\noindent (iii) In \cite{Diehl2008}, Diehl reformulated the
$\Gamma$-condition from his previous works \cite{Diehl1,Diehl2,Diehl3}
in the following form: the pair $(u^l,u^r)$ satisfies the
$\Gamma$-condition if
\begin{equation}\label{eq:Gamma-condition}
    \begin{split}
        & \text{$f^l(u^l)=f^r(u^r)$ and there exists
        $u^o\in \text{ch}(u^l,u^r)$ such that}\\
        & (u^r-u^o)\,(\,f^r(z)-f^r(u^r)) \geq 0,
        \quad \forall z\in \text{ch}(u^r,u^o)\\
        & (u^o-u^l) (f^l(z)-f^l(u^l)) \geq 0,
        \quad \forall z\in \text{ch}(u^l,u^o),
    \end{split}
\end{equation}
where for $a,b\in\R$, $\text{ch}(a,b)$ denotes the interval $[\min\{a,b\},\max\{a,b\}]$. Clearly,
\eqref{eq:ClosureViscGerm} coincides with \eqref{eq:Gamma-condition}. The descriptions
\eqref{eq:ClosureViscGerm} and \eqref{eq:Gamma-condition} are reminiscent 
of the Ole\u{\i}nik admissibility condition (for convex fluxes $f^{l,r}$) and 
of the ``chord condition'' (see \cite{Gelfand,HoldenRisebro,GoritskyChechkin}); as 
is the case with the chord condition, \eqref{eq:ClosureViscGerm} 
and \eqref{eq:Gamma-condition} are derived from 
the travelling-waves approach of \cite{Gelfand}.
\end{rem}

\begin{proof}[Proof of Proposition \ref{prop:G-VV-versus-G-VV-0}]
~

(i) The proof is entirely similar to the one of
Proposition \ref{prop:G-KRT-L1D}(ii). Because $f^{l,r}$ are
continuous, each of the six cases in \eqref{eq:six-cases} turns
out to be incompatible with the assumption that both $(u^l,u^r)$
and $(\hat u^l,\hat u^r)$ fulfill \eqref{eq:ViscGerm}.

Notice that a less tedious proof can be obtained by the passage to
the limit as $\eps\to 0$ in the Kato inequality which holds for
equation \eqref{eq:ModelProb-Visco} (indeed, recall that
\eqref{eq:L1D-States} follows from the Kato inequality
\eqref{eq:L1Dissipativity}). Such a proof is
given in \cite{AKR-ParisNote}, for the case of Lipschitz continuous $f^{l,r}$.
The general case follows by approximation.

(ii) Recall that we denote by $\mathcal{G}_{VV}$ the set of all
pairs satisfying \eqref{eq:ClosureViscGerm}. In a first step,
we show that $\mathcal{G}_{VV} \subset \overline{\mathcal{G}^{o}_{VV}}$.
Since $\overline{\mathcal{G}^{o}_{VV}}$ is an $L^1D$ germ
by (i) and Proposition \ref{prop:Closure}(i), this implies that
$\mathcal{G}_{VV}$ is an $L^1D$ germ.

Then in a second step, we show that $\mathcal{G}_{VV}$ contains
the dual $\left(\mathcal{G}^{o}_{VV}\right)^*$ of $\mathcal{G}^{o}_{VV}$.
According to Propositions \ref{prop:Closure} and \ref{prop:DualityOfGerms}, this
yields the reverse inclusion $\mathcal{G}_{VV}
\supset \overline{\mathcal{G}^{o}_{VV}}$
and hence the maximality of the $L^1D$
germ $\mathcal{G}_{VV}=\overline{\mathcal{G}^{o}_{VV}}$.

\smallskip
We will make repeated use of the continuity of $f^{l,r}$, without mentioning it.

\smallskip

\noindent \underline{Step 1}. Let $(u^l,u^r)\in\mathcal{G}_{VV}$.
If $u^l=u^r$, then $(u^l,u^r)\in \mathcal{G}^{o}_{VV}\subset
\overline{\mathcal{G}^{o}_{VV}}$.
The remaining cases are symmetric; let us treat the one with $u^l<u^r$.

Take $s$ and $u^o$ as introduced in \eqref{eq:ClosureViscGerm}.
Consider the function $f^l(\cdot)-s$ on the interval $[u^o,u^r)$.
If it has a zero point, set $z^l:=\min\Set{z\in  [u^o,u^r)\, \Big |\, f^l(z)=s}$
and $z^r:=\min\Set{z\in  [z^l,u^r]\, \Big| \, f^r(z)=s}$ (since
$f^r(u^r)=s$, $z^r$ is well defined).
By construction and by \eqref{eq:ClosureViscGerm}, we have

--- $u^l\leq z^l$, $f^l(u^l)=f^l(z^l)=s$, and $f^l\geq s$ on
$[u^l,z^l]$;

--- $z^l\leq z^r$, $f^l(z^l)=f^r(z^r)=s$, and $f^r>s$ on the
interval $(z^l,z^r)$;

--- $z^r\leq u^r$, $f^r(z^r)=f^r(u^r)=s$, and $f^r\geq s$ on
$[z^r,u^r]$.

This means that $(u^l,z^l)$ (resp., $(z^r,u^r)$) is a left contact
shock (resp., a right contact shock), and $(z^l,z^r)\in
\mathcal{G}^{o}_{VV}$.

Now consider the situation in which $f^l(\cdot)-s$ has no zero point on the
interval $[u^o,u^r)$. Then $f^l\geq s$ on $[u^l,u^r]$. In the case
$f^l(u^r)=s$, $(u^l,u^r)$ is a left contact shock and
$(u^l,u^r)\in \mathcal{G}^{o}_{VV}$.
Otherwise, we define $z^l:=\max\Set{z\in [u^l,u^r)\, \Big |\, f^l(z)=s}$. In this case,
$(u^l,z^l)$ is a left contact shock. In addition, $(z^l,u^r)\in
\mathcal{G}^{o}_{VV}$ because $z^l<u^r$,
$f^l(z^l)=f^r(u^r)=s$, and $f^l>s$ on $(z^l,u^r]$.

In all the cases, by definition of the closure, we conclude
that $(u^l,u^r)\in \overline{\mathcal{G}^{o}_{VV}}$.

\smallskip
\noindent \underline{Step 2}. It suffices to show that if $f^l(u^l)=f^r(u^r)$
but $(u^l,u^r)\notin \mathcal{G}_{VV}$, then
there exists $(c^l,c^r)\in \mathcal{G}^{o}_{VV}$  such that
\begin{equation}\label{eq:Anti-L1D}
    q^l(u^l,c^l)<q^r(u^r,c^r).
\end{equation}
Set $s:=f^{l,r}(u^{l,r})$. As before, it suffices to consider the
case $u^l<u^r$. Define
$$
z^l:=\sup\Set{z\in [u^l,u^r]\, \Big |\, f^l\geq s \text{~on~} [u^l,z]}, \;
z^r:=\inf\Set{z\in [u^l,u^r]\, \Big |\, f^r\geq s \text{~on~} [z,u^r]}.
$$
If $z^l\geq z^r$, then \eqref{eq:ClosureViscGerm} holds with $u^o=z^l$, so that
$(u^l,u^r)$ lies in $\mathcal{G}_{VV}$. Thus, $z^l<z^r$.
Now, there are three cases to be investigated:

\noindent (a) $f^l$ and $f^r$ have a crossing point $z^o$ in the interval
$(z^l,z^r)$ such that $f^{l,r}(z^o)<s$;

\noindent (b) $f^l$ and $f^r$ have a crossing point $z^o$ in the interval
$(z^l,z^r)$ such that $f^{l,r}(z^o)\geq s$;

\noindent (c) either $f^l<f^r$ on the interval $(z^l,z^r)$, or
$f^r<f^l$ on the interval $(z^l,z^r)$.

In case (a), setting $(c^l,c^r):=(z^o,z^o)$ we obtain
\eqref{eq:Anti-L1D}, because $u^l<c^l$, $c^r<u^r$, and
$f^{l,r}(u^{l,r})>f^{l,r}(z^o)$.
Notice that we do have $(z^o,z^o)\in \mathcal{G}^{o}_{VV}$.

In case (b), by definition of $z^{l,r}$,
there exists $\hat s<s$ such that $\hat s$ belongs
to $f^l\left([z^l,z^o])\cap f^r([z^o,z^r]\right)$.
In this case, set $c^l:=\max\Set{z\in [z^l,z^o]\, \Big |\,f^l(z)=\hat s}$ and
$c^r:=\min\Set{z\in [z^o,z^r]\, \Big |\,f^r(z)=\hat s}$.
We then have \eqref{eq:Anti-L1D}
for the very same reasons as in case (a). In addition,
$(c^l,c^r)\in \mathcal{G}^{o}_{VV}$,
because $c^l<c^r$ and $f^l>\hat s$
on $(c^l,z^o]$, $f^r>\hat s$ on $[z^o,c^r)$.

In case (c), the two situations are similar. Consider, e.g., the
case $f^l<f^r$ on  $(z^l,z^r)$. Choose for $c^r$ the point
of $[z^l,z^r]$ where $f^r$ attains its minimum
value over $[z^l,z^r]$ and which is the closest one to $z^l$.
By definition of $z^r$, $\hat s:=f^r(c^r)$ is smaller than $s$.
Because $f^r(z^l)\geq f^l(z^l)=s>\hat s$,
we have $c^r> z^l$; in turn, this yields $f^l(c^r)<f^r(c^r)=\hat s$.
Since $f^l(z^l)=s>\hat s$, there exists $c^l$ in the interval
$(z^l,z^r)$ such that $f^l(c^l)=\hat s$. The pair $(c^l,c^r)$
fulfills \eqref{eq:ClosureViscGerm}. In addition, by the
definition of $c^r$ we have $f^r\geq \hat s$ on $[z^l,z^r]\supset
[c^l,c^r]$; thus $(c^l,c^r)\in \mathcal{G}^{o}_{VV}$.

In all cases, we have constructed $(c^l,c^r)\in \mathcal{G}^{o}_{VV}$
with property \eqref{eq:Anti-L1D}.

The contradiction shows that $\Bigl(\mathcal{G}^{o}_{VV}\Bigr)^*
\subset \mathcal{G}_{VV}$. Thus  (ii) follows.

(iii) We first show the inclusion  $\mathcal{G}^{o}_{VV}
\subset \mathcal{G}^{s}_{VV}$. Let $(u^l,u^r)$
satisfy \eqref{eq:ViscGerm}. In the case $u^l=u^r$, the standing-wave profile
$W$ can be chosen constant on $\R$. The four other cases are
symmetric. For instance, in the case $u^l<u^r$ and $f^l(z)>s$ for
all $z\in (u^l,u^r]$, $W$ is a continuous function that is
constant (equal to $u^r$) on $[0,\infty)$. On the interval
$(-\infty,0]$, $W$ is constructed as a solution of the autonomous ODE
$$
W'=f^{l}(W)-f^l(u^l), \qquad  W(0)=u^r.
$$
Indeed, because $f^{l}(W)-f^l(u^l)=f^l(W)-s>0$ for $W$ taking
values in $(u^l,u^r]$, $W$ is non-decreasing. If the graph of $W$
crosses the line $W=u^l$, we extend the solution by the constant
value $u^l$ on the left from the crossing point. Otherwise, $W$ is
defined on the whole interval $(-\infty,0]$, and there exists
$d:=\lim_{\xi\to-\infty} W(\xi)\in [u^l,u^r]$. In this case,
$f^l(d)-s=0$, which yields $d=u^l$. Therefore \eqref{eq:VV-germ} holds.

Next, we show that $\mathcal{G}^{s}_{VV}\subset
\overline{\mathcal{G}^{o}_{VV}}$. Let $(u^l,u^r)\in\mathcal
G^{s}_{VV}$. The case $u^l=u^r$ is trivial; by
a symmetry argument, we can assume $u^l<u^r$. The standing-wave
profile $W$ in \eqref{eq:VV-germ} is a continuous on $\R$ function
which is monotone on $\R_-$ and on $\R_+$. Moreover, replacing $W$
by $\max\Set{u^l,\min\Set{W,u^r}}$, we still have a standing-wave
profile with the same properties. Then $u^o:=W(0)$ lies within
$[u^l,u^r]$. The monotonicity of $W$ guarantees that $f^l-s\geq 0$
on $[u^l,u^o]$, and $f^r-s\geq 0$ on $[u^o,u^r]$. Thus $(u^l,u^r)$
fits the definition \eqref{eq:ClosureViscGerm}.
By (ii), $(u^l,u^r)\in \overline{\mathcal{G}^{o}_{VV}}$.

(iv) The claim follows readily from (ii), (iii), and
Proposition \ref{prop:Closure} (iii), (iv).
\end{proof}

Now we compare $\mathcal{G}_{VV}$ with the known germs studied in
Section \ref{sec:Examples} and
in \cite{GimseRisebro1,GimseRisebro2,Diehl1,Diehl2,Diehl3,Diehl2008}.
Let us emphasize that in many cases studied in these papers,
$\mathcal{G}_{VV}$ turns out to be  complete. The completeness of
$\mathcal{G}_{VV}$ will be further studied in
Theorem \ref{th:VV-1D-converges} (see Section \ref{sec:Existence}).

\begin{rem}\label{rem:VVgerm-and-others}
~

(i) With the help of the characterizations
\eqref{EqKarlsenRisebroTowersGerm}, \eqref{eq:ViscGerm},
\eqref{eq:ClosureViscGerm} of $\mathcal{G}_{KRT}$, $\mathcal{G}^{o}_{VV}$,
and $G_{VV}$, respectively, we easily  see that 
$\mathcal{G}^{o}_{VV}\subset \mathcal{G}_{KRT}
= \mathcal{G}_{VV}$  whenever the crossing
condition \eqref{eq:BasicFluxesGoodCrossingCase} holds.
In this case, the  $\mathcal{G}_{VV}$
and $\mathcal{G}_{KRT}$ entropy solutions coincide.

(ii) Consider the case of bell-shaped fluxes
\eqref{eq:Bell-shaped-fluxes} treated  in
Section \ref{ssec:Bell-shaped}.

--- First assume $u^r_o\leq u^l_o$ and that there
exists a crossing point $u_\chi\in [u^r_o,u^l_o]$ between
$f^l$ and $f^r$ (in this case, the crossing condition
\eqref{eq:BasicFluxesGoodCrossingCase} holds). Then $\mathcal
G_{VV}$ contains the definite germ
$\mathcal{G}_{(u_\chi,u_\chi)}$. Thus the
$\mathcal{G}_{(u_\chi,u_\chi)}$, $\mathcal{G}_{VV}$, and
$\mathcal{G}_{KRT}$ entropy solutions coincide.
In addition, since $\left(\mathcal{G}_{(u_\chi,u_\chi)}\right)^*$
is complete, $\mathcal{G}_{VV}$ is complete.

--- If, on the other hand, a crossing point $u_\chi$ exists but
$u^l_o\leq u_\chi \leq u^r_o$ (this is the ``bad" crossing case),
then the $\mathcal{G}_{VV}$-entropy solutions coincide
with the entropy solutions for the germ
$\mathcal{G}_{(A,B)}$ with $(A,B)$ corresponding
to $s_{(A,B)}=s_{\max}$, where
\begin{equation}\label{eq:smax}
    s_{\max}:=
    \max\Set{\text{$s\in \R_+$ $|$ $\exists
    (A,B)\in [u^l_o,1]\times[0,u^r_o]$ s.t.~$f^l(A)=f^r(B)=s$}}.
\end{equation}
Also in the case where the fluxes do not cross within $(0,1)$,
$\mathcal{G}_{VV}=\mathcal{G}_{(A,B)}$
with $s_{(A,B)}=s_{\max}$. In both
cases, $\mathcal{G}_{VV}$ is complete because
$\left(\mathcal{G}_{(A,B)}\right)^*$ is
complete.  Let us stress   that in the non-crossing case,
$\mathcal{G}_{KRT}$ coincides with $\mathcal{G}_{VV}$,
whereas in the case of a ``bad" crossing,
$\mathcal{G}_{KRT}\setminus \mathcal{G}_{VV}\neq \emptyset$.
\end{rem}

In conclusion, following  Diehl
\cite{Diehl1,Diehl2,Diehl3} and B\"urger, Karlsen, Towers \cite{BurgerKarlsenTowers},
let us mention that the germ $\overline{\mathcal{G}^{o}_{VV}}$
also arises from the more general regularization-viscosity approach
\begin{equation}\label{eq:RegulViscoApprox}
    u_t+\mathfrak{f}^\delta (u)_x
    = \eps  u_{xx} \quad
    \text{in $\mathcal{D}'(\R_+\times\R)$,}
\end{equation}
where $\mathfrak{f}^\delta$ is a continuous (in $x$) approximation to
$\mathfrak{f}$. 
\begin{prop}\label{prop:reg-visc-approach}
Suppose $f^{l,r}$ are locally Lipschitz continuous.
Let $(\mathfrak f^\delta(\cdot,\cdot))_{\delta>0}$
be a family of continuous functions approximating
$\mathfrak f$ in \eqref{eq:ModelProb} in the following sense:
\begin{itemize}
\item[--]$\mathfrak f^\delta(x,z)$ is of the form $F(x/\delta,z)$ 
with some fixed function $F$;

\item[--]for all $y\in\R$, there exists $\alpha_y\in[0,1]$ such that
$F(y,\cdot)=\alpha_y f^l(\cdot)+ (1-\alpha_y)f^r(\cdot)$;

\item[--]for $y\leq -1$, $\alpha_y=1$ (i.e., $F(y,\cdot)\equiv f^l(\cdot)$);
similarly, for $y\geq 1$, $F(y,\cdot)\equiv f^r(\cdot)$.
\end{itemize}
Let $(u^l,u^r)\in \mathcal{G}^{o}_{VV}$. Denote by $L$ the Lipschitz
constant of $f^{l,r}$ on the interval between $u^l$ and $u^r$.

Then, whenever $\delta/\eps\leq 2/L$, there exists a stationary
solution $u(t,x)=u^{\eps,\delta}(x)$ of \eqref{eq:RegulViscoApprox}
satisfying $\lim_{x\to -\infty} u^{\eps,\delta}(x)=u^l$,
$\lim_{x\to \infty} u^{\eps,\delta}(x)=u^r$.
In addition, as $\eps\downarrow 0$, with $0<\delta \leq \frac 2L \eps$,
$u^{\eps,\delta}$ converges a.e.~and in $L^1_{\loc}(\R)$
to the elementary stationary solution $u^l\char_{\Set{x<0}}
+u^r\char_{\Set{x>0}}$ of \eqref{eq:ModelProb}.
\end{prop}

\begin{proof}
Let $(u^l,u^r)\in \mathcal{G}^{o}_{VV}$. The case
$u^l=u^r$ is trivial since constants solve \eqref{eq:RegulViscoApprox}.
The other cases are symmetric, so let us only
consider the one with $u^l<u^r$ and $f^l>s$
on the interval $(u^l,u^r]$.

Setting $y=x/\delta$, $u^{\eps,\delta}(x)=U(x/\delta)$, we are
reduced to the problem
\begin{equation}\label{eq:ODEprob}
    \eps\,U'(y)=\delta\,(F(y,U(y))-s),
    \quad U(-\infty)=u^l,\; U(\infty)=u^r.
\end{equation}
In view of the assumptions on $F$ and also the ratio $\delta/\eps$,
the (constant) function $m_+:y\in(-\infty,1]\mapsto u^r$  is a
subsolution of the ODE $U'=\delta/\eps(F(y,U)-s)$; and the
function $m_-:y\in (-\infty,1]\mapsto \min \{u^l,u^r+L\delta/\eps\,(y-1)\}$
is a supersolution of the same ODE. By the classical ODE existence
and comparison result, there exists a solution
of the ODE on $(-\infty,1]$ with $U(1)=u^r$, and we have
$m_-\leq U\leq m_+$. In particular, $U(\cdot)$ takes values in
$(u^l,u^r]$ and satisfies the autonomous equation
$U'=\delta/\eps(f^l(U)-s)$ on $(-\infty,-1)$. Since by assumption,
$f^l>s$ on $(u^l,u^r]$, $U(\cdot)$ admits a limit as $y\to -\infty$
which necessarily equals $u^l$. Finally, extending $U$ by
the constant value $u^r$ on $[1,\infty)$, we obtain a solution to
problem \eqref{eq:ODEprob}.

For the proof of convergence of the solutions $u^{\eps,\delta}(x)=U(x/\delta)$, we 
replace $m_+(\cdot)$ on $(-\infty,-1]$ by the solution $V(\cdot)$ 
of $ V'=\delta/\eps(f^l(V)-s)$ with the initial datum $V(-1)=u^r$;
with the same arguments as above, such a solution exists, it tends 
to $u^l$ as $y\to -\infty$, and it bounds
$U(\cdot)$ from above. By assumptions, we have 
$\delta\leq 2/L\eps\to 0$ as $\eps\to 0$. For all $x<0$, for all
sufficiently small $\eps$, the quantity $|u^{\eps,\delta}(x)-u^l|$ is 
upper bounded by $|V(x/\delta)-u^l|$ which
converges to zero. For $x>0$, we simply have $u^{\eps,\delta}(x)=u^r$ 
whenever $\delta\leq x$. We conclude the
pointwise convergence of $u^{\eps,\delta}$ to the 
profile \eqref{eq:PiecewiseC(x)}; since all these functions
take values within $[u^l,u^r]$, the $L^1_{\loc}$ convergence follows.
\end{proof}

\section{Some existence and convergence results }\label{sec:Existence}
In this section, for a fixed germ $\mathcal{G}$, we
study the existence of $\mathcal{G}$-entropy
solutions and convergence of various vanishing viscosity and
numerical approximations.

The first result concerns the vanishing viscosity
germ $\mathcal{G}_{VV}$, which
was described in Section \ref{sec:SVVGerm}.
The main assumption for the existence result is the
availability of a uniform $L^\infty(\R_+\times\R)$
estimate for solutions $u^\eps$ of \eqref{eq:ModelProb-Visco} with a
given initial datum $u_0\in L^\infty(\R)$.
We also impose some additional nonlinearity and Lipschitz
continuity assumptions on the fluxes $f^{l,r}$.
For the sake of simplicity, we restrict our attention to fluxes
satisfying $f^{l,r}(0)=0=f^{l,r}(1)$ (cf.~the bell-shaped
fluxes \eqref{eq:Bell-shaped-fluxes}) and an initial
condition $u_0$ with values in $[0,1]$; thus solutions are
automatically bounded, because they take values in $[0,1]$.

Recall that in the setting \eqref{eq:Bell-shaped-fluxes} of bell-shaped fluxes 
with $u^l_o\leq u_\chi\leq u^r_o$, the germ $\mathcal{G}_{VV}$ 
coincides with $\left(\mathcal{G}_{(u_\chi,u_\chi)}\right)^*$; the other
maximal $L^1D$ germs $\left(\mathcal{G}_{(A,B)}\right)^*$ are 
corresponding to other choices of the connection
$(A,B)$. For an arbitrary but fixed connection $(A,B)$, using a 
specially fabricated artificial viscosity we
prove convergence of viscosity approximations 
to the $\mathcal{G}_{(A,B)}$-entropy solution.

Then we look at a general complete $L^1D$ germ $\mathcal{G}$ for
a pair of locally Lipschitz continuous functions $f^{l,r}$.
We construct solutions by showing convergence of
a suitably adapted monotone three-point finite volume scheme.
For these results, the uniform $BV_{\loc}$ estimates away
from the interface $\Set{x=0}$ are of importance
(see \cite{BurgerGarciaKarlsenTowers,BurgerKarlsenTowers}).
As previously, we need a uniform $L^\infty$ bound on the approximations,
which actually comes for free from the completeness
of $\mathcal{G}$ (cf.~Proposition \ref{prop:LinftyEstimate}).

Finally, we prove convergence of uniformly $L^\infty$ bounded
(thus weakly compact)  sequences of
approximate $\mathcal{G}$-entropy solutions, without
utilizing $BV$ type a priori estimates and
assumptions that ensure strong compactness.
For this purpose, we have to assume that
$\mathcal{G}$ is a maximal $L^1D$ germ for which existence
is already known (cf.~Section \ref{ssec:EntropyProcessAndExistence}).
Then, assuming that the approximation procedure is
compatible with the germ $\mathcal{G}$ (i.e., that the
elementary stationary solutions selected
by the germ $\mathcal{G}$ are obtained as limits
of the approximation) and that an appropriate
Kato inequality holds for the approximate solutions, we deduce the
convergence of the approximating procedure.

\subsection{The standard vanishing viscosity approach}\label{ssec:VV-in-1D}
To be specific, let us work with the vanishing viscosity germ $\mathcal{G}_{VV}$ in
Definition \ref{def:VV-germ}, and note that according to
Proposition \ref{prop:G-VV-versus-G-VV-0}(iv), we can replace
$\mathcal{G}_{VV}$ by $\mathcal{G}^{s}_{VV}$ or
by $\mathcal{G}^{o}_{VV}$.

\begin{theo}\label{th:VV-1D-converges}
Suppose $f^{l,r}:[0,1]\longrightarrow \R$ are
Lipschitz continuous functions
that are not affine on any interval $I\subset [0,1]$.
Moreover, assume $f^{l,r}(0)=0=f^{l,r}(1)$. Then for all $u_0$
which is  measurable and takes values in $[0,1]$, there
exists a unique $\mathcal{G}_{VV}$-entropy
solution to problem \eqref{eq:ModelProb},\eqref{eq:InitialCond}. In
particular, the germ $\mathcal{G}_{VV}$ is definite and complete.
In addition, the $\mathcal{G}_{VV}$-entropy solutions are
the vanishing viscosity limits in the following sense.
For each $\eps>0$, there exists a weak solution
$u^\eps\in L^2_{\loc}(\R_+;H^1_{\loc}(\R))$ of the
viscous problem \eqref{eq:ModelProb-Visco}
with any measurable initial data $u^\eps|_{t=0}=u^\eps_0$
taking values in $[0,1]$.  If we moreover
assume $u^\eps_0\to u_0$ a.e.~on $\R$,
then $u^\eps$ converge a.e.~on $\R_+\times\R$ to the unique
$\mathcal{G}_{VV}$-entropy solution of
\eqref{eq:ModelProb},\eqref{eq:InitialCond}.
\end{theo}

A multi-dimensional analogue of Theorem \ref{th:VV-1D-converges} holds, 
see \cite{AKR-ParisNote,AKR-II}.

\begin{rem}\label{rem:th-visco-assumptions}
~

(i) The assumptions $f^{l,r}(0)=0=f^{l,r}(1)$ and $u^\eps_0(x)\in [0,1]$ 
are put forward to ensure that the
family $(u^\eps)_{\eps>0}$ is bounded in $L^\infty$. 
Indeed, Theorem~\ref{th:VV-1D-converges}  remains true if the
$L^\infty$ bound is provided by some other means. 
This boundedness assumption is important. For example, in
situations where the graphs of $f^l$ and $f^r$ do not intersect, 
it is then clear that \eqref{eq:ModelProb} has
no solution and consequently $\|u^\eps\|_\infty\to\infty$  as $\eps\to 0$.

(ii) On the contrary, the Lipschitz continuity assumption on $f^{l,r}$
is only needed to establish the existence of a solution semigroup for
viscous problem \eqref{eq:ModelProb-Visco} that satisfies the Kato
inequality. This assumption can be bypassed, thanks to a
regularization of $f^{l,r}$. We then establish the existence of a
(possibly non-unique) solution $u^\eps$ to
\eqref{eq:ModelProb-Visco},\eqref{eq:InitialCond} with merely
continuous functions $f^{l,r}$, and the convergence of the sequence
$(u^\eps)_{\eps>0}$ to the unique $\mathcal{G}_{VV}$-entropy solution
of problem \eqref{eq:ModelProb},\eqref{eq:InitialCond}.

(iii) As pointed out in Corollary
\ref{cor:viscosity-convergence-via-process}, the assumption of
non-degeneracy of $f^{l,r}$ imposed in Theorem
\ref{th:VV-1D-converges} can be dropped; but the proof becomes much
more indirect, involving a beforehand justification of the existence
of $\mathcal{G}_{VV}$-entropy solutions (for this, we use a numerical
scheme, see Theorem \ref{th:exist-by-numerics}). Moreover, in
Corollary \ref{cor:viscosity-convergence-via-process} we do not prove
that the germ $\mathcal{G}_{VV}$ is complete, but we take completeness
as a hypothesis.
\end{rem}

\begin{proof}[Proof of Theorem \ref{th:VV-1D-converges}]
Extend $f^{l,r}$ by zero outside $[0,1]$.
Let $T$ denote a generic positive number.
The following a priori estimate is easily obtained:
\begin{equation}\label{eq:estim-H1-u}
    \begin{split}
        & \int_\R \frac 12 \left(u^\eps(T,\cdot)\right)^2
        + \eps\int_0^{T} \int_\R |u^\eps_x|^2
        \\ & \leq \int_\R \frac 12 \left(u^\eps(0,\cdot)\right)^2
        \\ & \qquad +\int_0^{T} \int_{-\infty}^0
        \left( \int_0^{u^\eps} f^l(z)\,dz\right)_x
        +\int_0^{T} \int^{\infty}_0
        \left(\int_0^{u^\eps} f^r(z)\,dz\right)_x
        \\ & \leq C(T,\|u^\eps_0\|_2,f^{l,r}),
    \end{split}
\end{equation}
where $u^\eps\in C(\R_+;H^1(\R))$ is
a solution of \eqref{eq:ModelProb-Visco}.
Thanks to this estimate, it can be shown by the
classical Galerkin technique (see, e.g., Lions \cite{Lions-book})
that for any $\eps>0$ and $u^\eps_0\in L^2(\R)$  there exists a solution
$u^\eps\in C(\R_+;H^1(\R))$ to the Cauchy problem for the
parabolic equation \eqref{eq:ModelProb-Visco}.
What's more, the Kato inequality holds for \eqref{eq:ModelProb-Visco}.
Indeed, assuming $(u_0^\eps-\hat u_0^\eps)^+$ belongs to $\in L^1(\R)$,
let $u^\eps,\hat u^\eps\in  L^2_{\loc}(\R_+;H^1_{\loc}(\R))$
be two weak solutions of \eqref{eq:ModelProb-Visco} corresponding to the
initial data $u^\eps_0,\hat u_0^\eps\in L^\infty(\R)$, respectively.
For $\alpha<0$, let us introduce a Lipschitz
continuous approximation $H_\alpha$ of the $\sign^+(\cdot)$ function:
\begin{equation}\label{eq:Halpha}
    H_\alpha(z):=\min\Set{\frac {z^+}{\alpha},1}, \qquad z\in \R,
\end{equation}
and use the test function $H_\alpha(u^\eps-\hat u^\eps)\, \xi$,
$\xi\in \mathcal{D}([0,\infty)\times\R)$, in the weak
formulation of \eqref{eq:ModelProb-Visco}.
Notice that, due to the Lipschitz assumption on $f^l$,
\begin{multline*}
    \abs{\int_{\R_+} \int_{\R_-}
    \left(f^l(u^\eps)-f^l(\hat u^\eps)\right)
    \left(H_\alpha)'(u^\eps-\hat u^\eps\right)
    \left(u^\eps-\hat u^\eps\right)_x \xi} \\
    \leq \|(f^l)'\|_{L^\infty} \|\xi\|_{L^\infty}
    \iint_{\Set{0<u^\eps-\hat u^\eps<\alpha}}
    \abs{\left(u^\eps-\hat u^\eps\right)_x}
    \longrightarrow 0 \quad\text{as $\alpha\to 0$},
\end{multline*}
thanks to \eqref{eq:estim-H1-u}; and the same holds
with $\R_-,f^l$ replaced with $\R_+,f^r$.
Hence as $\alpha\to 0$, we deduce the Kato inequality:
for all $\xi\in \mathcal{D}([0,\infty)\times\R)$, $\xi\geq 0$,
\begin{equation}\label{eq:Kato-for-standardvisco}
    \begin{split}
        & \int_{\R_+} \int_{\R}
        \Bigl( (u^\eps-\hat u^\eps)^+\xi_t
        +\mathfrak{q}_+(x,u^\eps,\hat u^\eps) \xi_x \Bigr)
        \\ & \quad
        +\int_{\R} (u^\eps_0-\hat u^\eps_0)^+ \xi(0,\cdot)
        \geq \eps \int_{\R_+} \int_{\R}
        \left[ (u^\eps-\hat u^\eps)^+ \right]_x\xi_x;
    \end{split}
\end{equation}
where we have written $q_+$ for the
semi-Kruzhkov entropy flux:
$$
\mathfrak{q}_+(x,z,k) :=\sign^+(z-k)
\Bigl(
(f^l(z)-f^l(k)) \char_{\Set{x<0}}
+(f^r(z)-f^r(k))\char_{\Set{x>0}}
\Bigr).
$$
In this inequality, using, e.g.,~the technique of Maliki and
Tour\'e \cite{MalikiToure} we can let $\xi$ converge to the
characteristic function of $[0,T)\times\R$ in such a way that the
terms with $\xi_x$ vanish as  $\eps\to 0$. Then we get the $L^1$
contraction and comparison inequality
\begin{equation}\label{eq:L1CC-ineq-standvisc}
    \text{for a.e.~$t>0$,}
    \quad \int_\R (u^\eps-\hat u^\eps)^+(t)
    \leq \int_{\R} (u^\eps_0-\hat u^\eps_0)^+.
\end{equation}

Inequality \eqref{eq:L1CC-ineq-standvisc} ensures the uniqueness of
a weak solution in $L^2_{\loc}(\R_+;H^1_{\loc}(\R))$ of the
Cauchy problem for the parabolic equation
\eqref{eq:ModelProb-Visco}. It also yields  the comparison principle.
Keeping in mind that, under our assumptions, the constants
$0$ and $1$ are evident solutions of \eqref{eq:ModelProb-Visco}, we 
derive the maximum principle: for data $u_0^\eps$ taking values in $[0,1]$, $0\leq
u^\eps\leq 1$ holds a.e.~on $\R_+\times \R$.

Now we are in a position to extend the existence result 
for \eqref{eq:ModelProb-Visco} to a general initial
function $u_0$ taking values in $[0,1]$. 
Indeed, following Ammar and Wittbold~\cite{AmmarWittbold} we can take
the bi-monotone approximation  of $u_0\in L^\infty(\R)$ 
by bounded compactly supported functions, namely,
$$
(u_0)_{m,n}:=\min\left\{(u_0)^+,n\right\}\char_{\{|x|<n\}} 
-\min\left\{(u_0)^-,m\right\}\char_{\{|x|<m\}}.
$$
We deduce the existence of a weak solution to \eqref{eq:ModelProb-Visco}. 
Moreover, thanks to the monotone
convergence theorem, the obtained solutions 
still fulfill \eqref{eq:Kato-for-standardvisco} and
\eqref{eq:L1CC-ineq-standvisc}. Finally, the 
solutions belong to $L^2_{\loc}(\R_+;H^1_{\loc}(\R))$, thus the
uniqueness result applies.

Now, let us justify the convergence of $u^\eps$ to a 
$\mathcal{G}_{VV}$-entropy solution. 
Take a sequence of data $u_0^\eps$ as in the statement of the theorem. 
By the non-degeneracy assumption on $f^{l,r}$ and strong
precompactness results  of \cite{Panov-precomp-first,Panov-precomp-ARMA}, 
applied separately in the domains
$\{x>0\}$ and $\{x<0\}$, we deduce that there 
exists a (not labelled) sequence $\eps\downarrow 0$ such that
$u^\eps$ converge to some limit $u$ in $L^1_{\loc}(\R_+\times\R)$. 
By the dominated convergence theorem, we also
have the $L^1_{\loc}$ convergence of $u_0^\eps$ to $u_0$. 
Passing to the limit in the distributional formulation
of the Cauchy problem for \eqref{eq:ModelProb-Visco}, we find 
that $u$ is a weak solution of \eqref{eq:ModelProb},\eqref{eq:InitialCond}. 
Following \cite{Kruzhkov}, we also deduce the Kruzhkov entropy
inequalities in the domains $\{x>0\}$ and $\{x<0\}$. 
In particular, the existence of strong traces
$\gamma^{l,r}u$ on $\Set{x=0}$ follows.

Finally, notice that for any pair $u,\hat u$ obtained as the vanishing
viscosity limit with the same extracted subsequence $\eps\to 0$, we
can pass to the limit in the Kato inequality
\eqref{eq:Kato-for-standardvisco}, where the right-hand side vanishes
due to a uniform bound on $\eps |u^\eps_x|^2$ in
$L^1_{\loc}(\R_+\times\R)$.  To conclude the proof, it remains to
notice that, by the definition \eqref{eq:VV-germ} of
$\mathcal{G}^{s}_{VV}$, the elementary solutions $\hat
u=c^l\char_{\Set{x<0}}+c^r\char_{\Set{x>0}}$, $(c^l,c^r)\in
\mathcal{G}^{s}_{VV}$, are obtained as vanishing viscosity
limits. Passing to the limit in the corresponding Kato inequalities
written for $\eps>0$, we infer the entropy inequalities
\eqref{eq:EntropySolDefi} with $(c^l,c^r)\in \mathcal G^{s}_{VV}$ and
zero remainder term $\mathcal R_{\mathcal{G}}$.  According to
Proposition \ref{prop:Carrillo-type-defs} and Definition
\ref{def:AdmIntegral}, $u$ is a $\mathcal{G}^{s}_{VV}$-entropy
solution of \eqref{eq:ModelProb}, \eqref{eq:InitialCond}.  By
Proposition \ref{prop:G-VV-versus-G-VV-0}(iv),
$\left(\mathcal{G}^{s}_{VV}\right)^*$ coincides with the maximal
$L^1D$ germ $\mathcal{G}_{VV}$, thus we conclude that $u$ is the
unique $\mathcal{G}_{VV}$-entropy solution of \eqref{eq:ModelProb},
\eqref{eq:InitialCond}.  The uniqueness of the accumulation point
ensures that all sequences converge to the same limit $u$.

It remains to notice that we have obtained a solution
for every measurable initial function $u_0$ with values in $U=[0,1]$.
In particular, for all Riemann initial data $u_0$ in \eqref{eq:RiemannInitialCond},
there exists a $\mathcal{G}_{VV}$-entropy solution. According to
Remark \ref{rem:RPb-G-solutions}, the
germ $\mathcal{G}_{VV}$ is complete.
\end{proof}

\subsection{The vanishing viscosity approach 
adapted to $(A,B)$-connections}\label{ssec:VV-for-connections}
In this section we work with bell-shaped fluxes
as defined in \eqref{eq:Bell-shaped-fluxes}.
Recall that any complete maximal $L^1D$ germ is of the form
$\left(\mathcal{G}_{(A,B)}\right)^*$, where $\mathcal{G}_{(A,B)}:=\Set{(A,B)}$
is a definite germ and the connection $(A,B)$ is a pair satisfying
\begin{equation}\label{eq:AB-assumptions}
    \text{$(A,B)\in [u^l_o,1]\times[0,u^r_o]$ and
    $f^l(A)=f^r(B)=:s_{(A,B)}$}.
\end{equation}
The explicit description of $\left(G_{(A,B)}\right)^*$ is
given by formula \eqref{eq:AB-germ}.

Our goal is to construct $\mathcal{G}_{(A,B)}$-entropy
solutions by the vanishing viscosity method; clearly, the choice of
viscosity must be adapted to the connection $(A,B)$.
For the Buckley-Leverett equation with a flux that is discontinuous at $x=0$,
Kaasschietter \cite{Kaasschietter} gives a physically motivated
viscosity of the form
$$
\eps \Bigl(\lambda(x,u) p_c(u)_x  \Bigr)_x,
$$
where $p_c(\cdot)$ is the capillary pressure and
$\lambda(x,\cdot)=\lambda^l(\cdot)\char_{\Set{x<0}}
+\lambda^r(\cdot)\char_{\Set{x>0}}$ is the
mobility function, discontinuous at $\Set{x=0}$. This choice corresponds
to the particular connection $(A,B)$ such that
$s_{(A,B)}=s_{\max}$, cf.~\eqref{eq:smax}.

Our approach is more academic. We fix an arbitrary
connection $(A,B)$ and construct an artificial ``adapted viscosity''
of the form
$$
\eps \Bigl(a(x,u)\Bigr)_{xx}
$$
such that the stationary solution $c(x):=A \char_{\Set{x<0}}+ B\char_{\Set{x>0}}$
of the limit equation \eqref{eq:ModelProb}
is also a solution of the viscous problem
\begin{equation}\label{eq:adapted-viscosity}
    \begin{split}
        & u_t + \mathfrak{f}(x,u)_x=\eps a(x,u)_{xx},
        \\
        & \mathfrak{f}(x,z)
        =
        \begin{cases}
            f^l(z), & x<0,\\
            f^r(z), & x>0.
        \end{cases}
        \quad
        a(x,z)
        =
        \begin{cases}
            a^l(z), & x<0,\\
            a^r(z), & x>0.
        \end{cases}
    \end{split}
\end{equation}
In the sequel, by a weak solution of \eqref{eq:adapted-viscosity} we
mean a function $u\in L^\infty(\R_+\times\R)$ satisfying
\eqref{eq:adapted-viscosity} in the sense of distributions and such
that $w(\cdot):=a(\cdot,u(\cdot))$ belongs to
$L^2_{\loc}(\R_+;H^1_{\loc}(\R))$.  For \eqref{eq:adapted-viscosity}
to be parabolic, the functions $a^{l,r}$ should be strictly increasing
on $[0,1]$; and it is convenient to ask that
$\inf_{z\in[0,1]}\left(a^{l,r}\right)'(z)>0$.  Because
$f^{l,r}(0)=0=f^{l,r}(1)$ by assumption, it is convenient to require
the constants $u\equiv 0$ and $u\equiv 1$ to be solutions of
\eqref{eq:adapted-viscosity} for all $\eps>0$.

For example, any continuous functions $a^{l,r}$ such that
\begin{equation}\label{eq:choice-of-a}
    \text{$a^{l,r}$ are strictly monotone on $[0,1]$, and}\,
    \begin{cases}
        a^l(0)=0=a^r(0), \\
        a^l(A)=\kappa = a^r(B), \\
        a^l(1)=1= a^r(1),
     \end{cases}
\end{equation}
with $\kappa \in (0,1)$, do satisfy the above
requirements, except with $(A,B)=(1,0)$, in which case
we take $\kappa=0$. For example, we can interpolate
the values \eqref{eq:choice-of-a} to produce piecewise
affine and strictly increasing examples of 
such functions $a^{l,r}$ defined on $[0,1]$. Notice that, by choosing
$\kappa$ appropriately, for the connection $(u_\chi,u_\chi)$ we
obtain the standard viscosity $\eps u_{xx}$ studied
in the previous section; under assumption
\eqref{eq:Bell-shaped-fluxes}, we actually have
$\left(\mathcal{G}_{(u_\chi,u_\chi)}\right)^*=\mathcal{G}_{VV}$.

Now we will prove that the $\mathcal{G}_{(A,B)}$-entropy
solutions of \eqref{eq:ModelProb} can be obtained
as the limit of the vanishing adapted viscosity
approximations defined by \eqref{eq:adapted-viscosity}.
More precisely, we have the following analogue
of Theorem \ref{th:VV-1D-converges}, which
represents a reinterpretation and improvement of the results of
B\"urger, Karlsen, Towers \cite{BurgerKarlsenTowers}.

\begin{theo}\label{th:AB-connections-VV-1D-converges}
Assume that $f^{l,r}$ are Lipschitz continuous
functions of the form \eqref{eq:Bell-shaped-fluxes}.
In addition, assume that $f^l$ and $f^r$ are not affine on any interval
$I\subset [0,1]$. Let $(A,B)$ be a given
connection of the form \eqref{eq:AB-assumptions}.
Then, for each measurable initial function $u_0$ taking
values in $[0,1]$, there exists a unique
$\mathcal{G}_{(A,B)}$-entropy solution of
\eqref{eq:ModelProb},\eqref{eq:InitialCond}.
Moreover, consider $a^{l,r}$ satisfying \eqref{eq:choice-of-a} and set
$$
a(x,\cdot)= a^l(\cdot)\char_{\Set{x<0}}+ a^r(\cdot)\char_{\Set{x>0}};
$$
assume in addition that $\left(a^{l,r}\right)'\geq
\mathrm{const} >0$ a.e.~on $[0,1]$.
Let $(u^\eps_0)_{\eps>0}$ be a family of
measurable functions taking values in $[0,1]$ such that
$u^\eps_0\to u_0$ a.e.~on $\R$. For each $\eps>0$, there
exists a weak solution $u^\eps$ of the adapted viscosity
regularized problem \eqref{eq:adapted-viscosity}
with initial data $u^\eps\Big |_{t=0}=u^\eps_0$. The sequence
$(u^\eps)_{\eps>0}$ converges a.e.~on $\R_+\times \R$ to
the unique $\mathcal{G}_{(A,B)}$-entropy
solution of \eqref{eq:ModelProb},\eqref{eq:InitialCond}.
\end{theo}

\begin{proof}
We will focus on the existence of solutions to the adapted
viscous problems and the convergence of these
solutions as the viscosity parameter tends to zero. The proof is very
similar to the proof of Theorem \ref{th:VV-1D-converges}.
We will argue in terms of the unknown $w(t,x):=a(x,u(t,x))$. Then
$$
u(t,x)=b(x,w(t,x)), \quad
b(x,\cdot)= \left(a^l\right)^{-1}(\cdot)\char_{\Set{x<0}}
+ \left(a^r\right)^{-1}(\cdot)\char_{\Set{x>0}},
$$
and the Cauchy problem for \eqref{eq:adapted-viscosity} is equivalent to
\begin{equation}\label{eq:in-terms-of-w}
    b(x,w)_t + \mathfrak{f}(x,b(x,w))_x=\eps w_{xx}, \quad
    w(0,x)=w_0^\eps:=a(x,u^\eps_0(x)).
\end{equation}
Existence for \eqref{eq:in-terms-of-w} is obtained by proving
convergence of Galerkin approximations
along the lines of Alt and Luckhaus \cite{AltLuckhaus}.
Indeed, set $\mathcal B(x,r):=\int_0^r b(x,z)\,dz$.
In the same way as for \eqref{eq:estim-H1-u},
assuming that $\mathcal B\circ w_0^\eps(\cdot):=\mathcal B(\cdot,w_0^\eps(\cdot))$
belongs to $L^1(\R)$ and using the weak
chain rule \cite{AltLuckhaus}, we obtain the {\it a priori} estimate
\begin{equation*}%\label{eq:estim-H1-w-AB}
    \begin{split}
        &\int_\R  \mathcal B(\cdot,w^\eps(T,\cdot))
        + \eps\int_0^{T}\int_\R |w^\eps_x|^2
        \\ & \leq \int_\R \mathcal B(\cdot,u^\eps(0,\cdot))
        \\ & \qquad + \int_0^{T} \int_{-\infty}^0 \left(\int_0^{u^\eps} f^l(z)a^l(z)\,dz\right)_x
        +\int_0^{T}\int^{\infty}_0\left(\int_0^{u^\eps} f^r(z)a^r(z)\,dz\right)_x
        \\ & \leq C\left(T,\norm{\mathcal B\circ w^\eps_0}_{L^1(\R)},f^{l,r},a^{l,r}\right).
    \end{split}
\end{equation*}
This estimate and the compactness technique
of \cite{AltLuckhaus} ensure the convergence of
the Galerkin approximations to a weak solution
$w^\eps$ of \eqref{eq:in-terms-of-w} with the properties
\begin{align*}
    & \mathcal B\circ w^\eps\in L^\infty(\R_+;L^1(\R)),\quad
    w^\eps\in L^2_{\loc}(\R_+;H^1(\R)),\\
    & (b\circ w^\eps)_t\in L^2_{\loc}(\R_+;H^{-1}(\R)), \quad
    (b\circ w^\eps)(0,\cdot)=b\circ w^\eps_0.
\end{align*}
As in the proof of Theorem \ref{th:VV-1D-converges}, the assumption
$\mathcal B\circ w_0^\eps\in L^1(\R)$ can be dropped,
thanks to the bi-monotone approximation of $w_0^\eps$ by
compactly supported bounded functions
$$
(w_0^\eps)_{m,n}:=\min\left\{(w_0^\eps)^+,n\right\}\char_{\{|x|<n\}}
-\min\left\{(w_0^\eps)^-,m\right\}\char_{\{|x|<m\}}
$$
(see \cite{AmmarWittbold}) and thanks to the comparison principle that
we now justify.  Indeed, let $\hat w^\eps$ be a solution of the same
equation corresponding to the initial datum $\hat w^\eps_0$.  Take
$\xi\in \mathcal{D}([0,{\infty})\times\R)$, $\xi\geq 0$, and
$H_\alpha$ given by \eqref{eq:Halpha}. Utilizing the test function
$H_\alpha (w^\eps-\hat w^\eps) \xi$, using the doubling-of-variables
technique in time only, as in Otto \cite{Otto:parab-ellipt}, and
taking into account that $f^{l,r}$ and $b^{l,r}$ are Lipschitz
continuous (cf.~the proof of Theorem \ref{th:VV-1D-converges}), we
deduce the Kato inequality
\begin{equation*}%\label{eq:Kato-for-adaptedvisco}
    \begin{split}
        & \int_{\R_+} \int_{\R} \left( (b\circ w^\eps-b\circ \hat w^\eps)^+\xi_t
        +\mathfrak{q}_+(x;b\circ w^\eps,b\circ\hat w^\eps)\, \xi_x \right)\\
        & \qquad\qquad
        +\int_{\R} (b\circ w^\eps_0-b\circ \hat u^\eps_0)^+
        \xi(0,\cdot) \geq \eps
        \int_{\R_+}\int_{\R} (w^\eps_x-\hat w^\eps_x)\,\xi_x;
    \end{split}
\end{equation*}
here $\mathfrak{q}_+$ has the same meaning as in \eqref{eq:Kato-for-standardvisco}.
With the technique of \cite{MalikiToure}, we can let $\xi$ converge to the
characteristic function of $[0,T)\times\R$ and derive the
contraction and comparison inequality
\begin{equation*}%\label{eq:L1CC-ineq-standvisc-bis}
    \text{for a.e.~$t>0$,} \quad
    \int_\R (u^\eps- \hat  u^\eps)^+(t)
    \leq \int_{\R} (u^\eps_0-\hat u^\eps_0)^+.
\end{equation*}
By the definition of $(A,B)$ and assumptions
\eqref{eq:Bell-shaped-fluxes}, \eqref{eq:choice-of-a},
the constants $0$ and $1$ are evident
solutions of \eqref{eq:ModelProb-Visco}.
Hence, the following maximum principle holds:
for data $u_0^\eps$ taking values in $[0,1]$, we have
$0\leq u^\eps\leq 1$ a.e.~on $\R_+\times \R$. Also the function
$\hat u(x):=A\,\char_{\Set{x<0}}+ B\,\char_{\Set{x>0}}$ is an apparent solution of
\eqref{eq:ModelProb-Visco}. As in the proof of Theorem \ref{th:VV-1D-converges},
we use uniform estimates on $u^\eps$ and compactness
arguments to pass to the limit in the Kato inequality
written for $u^\eps$ and for $\hat u^\eps\equiv \hat u$.
What we obtain is the entropy inequality \eqref{eq:EntropySolDefi} with $(c^l,c^r)=(A,B)$
and $\mathcal R_{\mathcal{G}}=0$ (this is exactly the entropy
inequality of  B\"urger, Karlsen, Towers \cite{BurgerKarlsenTowers}).
In addition, we see that $u$ is a weak solution
of \eqref{eq:ModelProb}, \eqref{eq:InitialCond} and it is a Kruzhkov
entropy solution in the domains $\{x<0\}$ and $\{x>0\}$.
We conclude either using the uniqueness result of \cite{BurgerKarlsenTowers}, or
using Proposition \ref{prop:Carrillo-type-defs} and
the fact that $\mathcal{G}_{(A,B)}$ is definite.
\end{proof}

\subsection{Existence for complete germs through
the discretization approach}\label{ssec:Numerics}

We now establish the existence of a $\mathcal{G}$-entropy solution
for any complete maximal $L^1D$ germ $\mathcal{G}$.
For the sake of simplicity, let us take $U=\R$.

\begin{theo}\label{th:exist-by-numerics}
Let $\mathcal{G}$ be a complete maximal $L^1D$ germ.
Assume that the functions $f^{l,r}$ are locally Lipschitz continuous on $\R$.
Then for any initial function $u_0\in L^\infty(\R)$
there exists a unique $\mathcal{G}$-entropy solution
of problem \eqref{eq:ModelProb},\eqref{eq:InitialCond}.
\end{theo}

\begin{proof}
The proof is a combination of the well-known
finite volume method (see Eymard, Gallou\"et, Herbin \cite{EGH}) with
a careful treatment of the interface $\Set{x=0}$ using the Godunov scheme;
the proof of compactness of the family of discrete solutions
is based upon the $BV_{\loc}$ estimate of
Burger, Karlsen, Towers \cite{BurgerGarciaKarlsenTowers,BurgerKarlsenTowers}
and an $L^\infty$ bound similar to the
one of Proposition \ref{prop:LinftyEstimate}.

Recall that by the definition of a complete germ,
there exists a Riemann solver $\mathcal{RS}^{\mathcal{G}}$ at $\Set{x=0}$,
fully determined by $\mathcal{G}$, defined for all
Riemann data \eqref{eq:RiemannInitialCond}.
We set $m:=\text{ess}\inf_\R u_0$
(resp., $M:=\text{ess}\sup_\R u_0$) and denote
by $c^{l,r}$ (resp., by $C^{l,r}$) the one-sided
traces $\gamma^{l,r} \mathcal{RS}^{\mathcal{G}}m$
(resp., $\gamma^{l,r} \mathcal{RS}^{\mathcal{G}}M$) of the solution
of the Riemann problem at $\Set{x=0}$ with the constant datum
$u^\pm=m$ (resp., with the constant datum $u^\pm=M$).
Assign
\begin{equation}\label{eq:bB-def}
    b^{l,r}:=\min\Set{c^{l,r},m}, \qquad
    B^{l,r}:=\max\Set{C^{l,r},M}.
\end{equation}

We consider an explicit finite volume scheme based
on the uniform spatial mesh
$$
\bigcup\nolimits_{i\in \Z} (i\Delta x,(i+1)\Delta x), \qquad \Delta x>0,
$$
and a time step $\Delta t$ satisfying the CFL condition
$$
\lambda:=\frac{\Delta t}{\Delta x}\leq 1/{2L}, \quad
L:=\max\Set{\max\limits_{z\in [b^l,B^l]}\abs{(f^l)'(z)},
\max\limits_{z\in [b^r,B^r]}\abs{(f^r)'(z)}}.
$$
A lower bound for $L$ can be obtained from Proposition \ref{prop:LinftyEstimate}.

For $i\neq 0$, we can utilize any consistent monotone flux
$g_{i+1/2}(\cdot,\cdot)$ (see e.g.~\cite{EGH}); but at the interface $i=0$ we
take $g_{i+1/2}(\cdot,\cdot)$ to be the Godunov flux based
on the exact $\mathcal{G}$-Riemann solver.
Given the numerical flux $g_{i+1/2}(\cdot,\cdot)$, the
difference scheme is defined as
\begin{align*}
    & u^n_i=u_i^{n-1}
    - \lambda \left( g_{i+1/2}\left(u_{i+1}^{n-1},u_i^{n-1}\right)
    -g_{i-1/2}\left(u_i^{n-1},u_{i-1}^{n-1}\right)\right),
    \quad \forall n\in \NN, \;\forall i\in \Z.
\end{align*}
where the iteration is initialized by
\begin{align*}%\label{eq:IC-discr}
    & u^0_i:=\frac{1}{\Delta x} \int_{i\Delta x}^{(i+1)\Delta x} u_0,
    \qquad \forall i\in \Z.
\end{align*}
The monotonicity of the numerical flux
$g_{i+1/2}(\cdot,\cdot)$ and the CFL condition ensure
that the scheme can be written as
\begin{equation*}%\label{eq:scheme-with-H_i}
    u_i^n=H_i\left(u_{i-1}^{n-1},u_i^{n-1},u_{i+1}^n\right),
    \qquad \forall n\in \NN, \;\forall i\in \Z,
\end{equation*}
for some functions $H_i(\cdot,\cdot,\cdot)$ that are monotone
in each of the three arguments.  We identify the
discrete solution $(u_i^n)_{n\in\NN,i\in\Z}$ with the
piecewise constant function
$$
\mathcal{S}^hu_0:=\sum_{n\in\NN,\, i\in \Z} u_i^n
\char_{\Bigl((n-1)\Delta t,n\Delta t)\times (i\Delta x,(i+1)\Delta x)\Bigr)};
$$
where $h$ is a collective symbol for the
discretization parameters $\Delta x, \Delta t$.

By construction (exact $\mathcal{G}$-Riemann solver
at the interface $\Set{x=0}$),
\begin{equation}\label{eq:Godunov-preserves}
    \begin{split}
        & \text{the difference scheme preserves all
        stationary solutions of \eqref{eq:ModelProb}}\\
        & \text{of the form $c^l\,\char_{\Set{x<0}}+c^r\,\char_{\Set{x>0}}$
        with $(c^l,c^r)\in \mathcal{G}$.}
    \end{split}
\end{equation}

First, let us establish the following uniform
$L^\infty$ bound on $\mathcal{S}^hu_0$:
\begin{equation}\label{eq:Linfty-for-scheme}
    \begin{split}
        & b(x):=b^l\,\char_{\Set{x<0}} + b^r\,\char_{\Set{x>0}}
        \\ &\qquad
        \leq (\mathcal{S}^hu_0)(t,x)
        \\ & \qquad
        \leq B^l\,\char_{\Set{x<0}} + B^r\,\char_{\Set{x>0}}=:B(x),
    \end{split}
\end{equation}
for a.e.~$(t,x)\in (0,\infty)\times \R$, where $b^{l,r},B^{l,r}$ are 
defined in \eqref{eq:bB-def}. To prove \eqref{eq:Linfty-for-scheme}, we 
take the Riemann initial datum $b(\cdot)$ 
and look at the corresponding
discrete solution $(\mathcal{S}^hb)(\Delta t,\cdot)$ after the first time step. 
Denote by $b_i^n$ the values
taken by $(\mathcal{S}^hb)(\Delta t,\cdot)$. Notice that, 
by the choice of $b^{l,r}$, the exact solution
$\mathcal{RS}^{\mathcal{G}}b$ of the Riemann problem has 
the one-sided traces $c^{l,r}$ at $\Set{x=0}$.
Therefore, the Godunov flux $g_{1/2}(b^n_0,b^n_1)=g_{1/2}(b^l,b^r)$ 
at the interface at time level $n=1$ takes
the value $f^{l,r}(c^{l,r})$ (the two values being equal). From the 
definition of the scheme, it is clear that $b^1_i=b^0_i$ for 
all $i\neq 0,1$.  Moreover, from the definition of the scheme we have
$$
b^1_0=b^l-\lambda\left(f^l(c^l)-f^l(b^l)\right),
\qquad
b^1_1=b^r-\lambda\left(\,f^r(b^r)-f^r(c^r)\right).
$$
Because $b^{l}\leq c^{l}$ by construction, and because the state
$b^l$ at $x<0$ can be joined to the state $c^l$
at $x=0-$ with waves of negative speed for the flux $f^l$, it follows that
$f^l(c^l)\leq f^l(b^l)$. Similarly, we have $f^r(c^r)\geq f^r(b^r)$.
Combining the above information, we deduce that for all
$i\in \Z$, $b^1_i\geq b^0_i$.
Thus we have
$$
(\mathcal{S}^hb)(\Delta t,\cdot)\geq b(\cdot).
$$
Hence by induction, using the monotonicity of the scheme, we deduce
$$
(\mathcal{S}^hb)(t,x)\geq b(x) \;\; \text{for a.e.~$(t,x)\in (0,\infty)\times \R$}.
$$
Because $u_0\geq b$ a.e.~on $\R$ and again in view of the
monotonicity of the scheme,
$$
(\mathcal{S}^h u_0)(t,x) \geq (\mathcal{S}^h b)(t,x)
\geq b(x) \quad
\text{for a.e.~$(t,x)\in (0,\infty)\times \R$},
$$
which establishes the desired lower
bound in \eqref{eq:Linfty-for-scheme}.
The proof of the upper bound is entirely similar.

Now we assume that $u_0$ is compactly supported and belongs to $BV(\R)$.
Then for all $T>0$ and $l>0$, we apply the uniform
$BV\biggl((0,T)\times \Bigl(\R\setminus (-l,l)\Bigr)\biggr)$
estimate of \cite{BurgerGarciaKarlsenTowers,BurgerKarlsenTowers}
to the family $(\mathcal{S}^hu_0)_{h>0}$.
To establish this estimate, we first combine the discrete $L^1$ contraction
property, obtained from the Crandall-Tartar lemma \cite{CrandallTartar}, with
the $L^1$ Lipschitz continuity in time $t$ of the solver $\mathcal{S}^h$.
The result is an estimate of the time variation of $\mathcal{S}^hu_0$
on $(0,T)\times\R$ in terms of the variation $\text{Var}\,u_0$
of the initial data $u_0$ on $\R$. Then, using the mean value theorem, we pick
$l^h\in (0,l)$ such that the variation of $(\mathcal{S}^hu_0)(\cdot,\pm l^h)$
does not exceed $\text{Const}\left(\text{Var}\,u_0\right)/l$.
Subsequently, we consider $(\mathcal{S}^hu_0)|_{\R\setminus (-l^h,l^h)}$ as
originating from the finite volume discretization of two Cauchy-Dirichlet problems.
For instance on $(0,T)\times(l^h,\infty)$, both the initial condition
$u_0|_{(l^h,\infty)}$ and the boundary condition
$u_b(\cdot):=(\mathcal{S}^hu_0)(\cdot,l^h)$ have the variation
controlled in terms of $\text{Var}\,u_0$ and of $\frac 1l$.
It follows that the space-time variation of $\mathcal{S}^hu_0$
on $(0,T)\times\Bigl(\R\setminus (-l,l)\Bigr)\subset (0,T)
\times\Bigl(\R\setminus (-l^h,l^h)\Bigr)$ is bounded
uniformly in $h$; we refer to \cite[Lemma 4.2]{BurgerGarciaKarlsenTowers},
\cite[Lemmas 5.3, 5.4]{BurgerKarlsenTowers} for the details.
Using in addition the $L^\infty$ bound, with the help
of a diagonal extraction argument, we get convergence as $h\to 0$ of a (not labelled)
sequence $(\mathcal{S}^hu_0)_h$ to some limit $u$.

Now we justify that the limit $u$ is the unique $\mathcal{G}$-entropy 
solution of \eqref{eq:ModelProb}, \eqref{eq:InitialCond}. 
The standard ``weak $BV$'' estimates for 
monotone finite volume schemes (see \cite{EGH})
permit us to get an approximate weak formulation and pass to the limit as $h\to 0$. 
It follows that $u$ is a weak solution 
to \eqref{eq:ModelProb},\eqref{eq:InitialCond}. In the same 
way, the technique of \cite{EGH}
allows us to get the Kruzhkov entropy inequalities 
for $u$ in the domains $\{x<0\}$ and $\{x>0\}$.

According to Definition \ref{def:AdmIntegral} and
Proposition \ref{prop:Carrillo-type-defs}, what remains is to prove
the entropy inequality \eqref{eq:EntropySolDefi}
with an arbitrary non-negative test function $\xi$ (not necessarily
zero on the interface $\Set{x=0}$) and with any pair
$(c^l,c^r)\in \mathcal{G}$. Let us give a proof
using \eqref{eq:Godunov-preserves}.
To this end, notice that a key feature of the scheme,
thanks to the use the Godunov flux of the
exact $\mathcal{G}$-Riemann solver
at the interface, is that the discrete solutions
$u^h:=\mathcal{S}^hu_0$ take the values
 \begin{align*}
    &(\gamma^{l} u^h)(\cdot)
    \equiv \sum\nolimits_{n\in\NN}
    u_0^n\,\char_{((n-1)\Delta t,n\Delta t)}, \\
    & (\gamma^{r} u^h)(\cdot)
    \equiv \sum\nolimits_{n\in\NN}
    u_1^n\,\char_{((n-1)\Delta t,n\Delta t)}
 \end{align*}
at the interface (strong one-sided traces at the interface $\Set{x=0}$), which satisfy
\begin{equation*}\label{eq:traces-FV}
    \Bigl( (\gamma^lu^h)(t),(\gamma^r u^h)(t)\Bigr)\in \mathcal{G},
    \quad \text{for a.e.~$t>0$}.
\end{equation*}
Because $\mathcal{G}$ is an $L^1D$ germ, it follows that
\begin{equation}\label{eq:interface-dissip-ineq}
    q^l\left((\gamma^lu^h)(t),c^l\right)
    \geq q^r\left((\gamma^ru^h)(t),c^r\right),
    \qquad \forall (c^l,c^r)\in \mathcal{G},
\end{equation}
for a.e.~$t>0$. Therefore with the same
arguments as in \cite{EGH}, using \eqref{eq:Godunov-preserves}
and using in addition the interface dissipation
inequality \eqref{eq:interface-dissip-ineq}, we get the
required entropy inequalities \eqref{eq:EntropySolDefi}
with $(c^l,c^r)\in \mathcal{G}$ and zero remainder term.

Summarizing our findings, a $\mathcal{G}$-entropy
solution has been constructed for any compactly
supported initial function $u_0$ in $BV(\R)$. It
is easy to generalize this result so that it covers all
$u_0\in L^\infty(\R)$. Indeed, because of the
Lipschitz continuity assumption on $f^{l,r}$,
the contraction principle of Theorem \ref{th:Contraction-Comparison} can
be localized (as in the original work of Kruzhkov \cite{Kruzhkov});
using this contraction principle, by truncation and
regularization of $u_0$ we construct a strongly compact sequence
of approximations $u^\eps$. Then we pass to the
limit in this sequence of approximations
using the formulation of Definition \ref{def:AdmIntegral}.
This concludes the proof of the theorem.
\end{proof}

\subsection{On convergence of approximate
solutions without $BV$ estimates}\label{ssec:Numerics-bis}

In this section, we assume that the existence of a 
$\mathcal{G}$-entropy solution to the problem
\eqref{eq:ModelProb},\eqref{eq:InitialCond} is already known.
For  existence results, we refer, e.g., to 
Sections \ref{ssec:VV-in-1D}, \ref{ssec:VV-for-connections},
\ref{ssec:Numerics} and to many of the references at the end of the paper.
Under this assumption, we have the following general convergence result.

\begin{theo}\label{th:conv-via-process}
Assume that the fluxes $f^{l,r}$ are merely continuous and the
associated maximal $L^1D$ germ $\mathcal{G}$ is chosen in such a
way that there exists a solution to the problem
\eqref{eq:ModelProb},\eqref{eq:InitialCond} for all
bounded initial functions $u_0$.

Suppose that for any $\eps>0$ we are given a map
$S^\eps:L^\infty(\R)\mapsto L^\infty(\R_+\times\R)$
with the following properties:
\begin{itemize}
\itemsep=4pt
\item[(B1)] for each $u_0\in L^\infty(\R)$, the
    family $(S^\eps u_0)_{\eps>0}$ is
    bounded in $L^\infty(\R_+\times\R)$;

\item[(B2)] if $\hat u_0(x)=c^l \char_{\Set{x<0}}+ c^r\char_{\Set{x>0}}$
    with $(c^l,c^r)\in \mathcal{G}$,\\ then $S^\eps \hat u_0$
    converges to $\hat u_0$ a.e.~on $\R_+\times \R$, as $\eps\to 0$;

\item[(B3)] for all $u_0,\hat u_0\in L^\infty(\R)$ and for all
    nonnegative $\xi\in \mathcal{D}(\R_+\times\R)$, there
    exists $r_1=r_1(u_0,\hat u_0,\xi,\eps)$, with $r_1\to 0$
    as $\eps\to 0$, such that the following
    approximate Kato inequality holds:
    $$
    \int_{\R_+} \int_{\R} \Bigl( \abs{S^\eps u_0-S^\eps \hat u_0}\xi_t
    +\mathfrak{q}(x,S^\eps u_0,S^\eps \hat u_0) \xi_x \Bigr)
    +\int_{\R} \abs{u_0-\hat u_0} \xi(0,\cdot) \geq - r_1(u_0,\hat u_0,\xi,\eps).
    $$

\item[(B4)] for all $u_0\in L^\infty(\R)$ and for all
    $\xi\in \mathcal{D}(\R_+\times\R)$, there
    exists $r_2=r_2(u_0,\xi,\eps)$ with $r_2\to 0$ as
    $\eps\to 0$, such that the following
    approximate weak formulation holds:
    $$
    \int_{\R_+} \int_{\R} \Bigl( S^\eps u_0 \xi_t
    +\mathfrak{f}(x,S^\eps u_0) \xi_x \Bigr)
    +\int_{\R} u_0 \xi(0,\cdot)=r_2(u_0,\xi,\eps).
    $$

\item[(B5)] for all $u_0\in L^\infty(\R)$, for all
    nonnegative $\xi\in \mathcal{D}\Bigl(\R_+
    \times\left(\R\setminus\Set{x=0}\right)\Bigr)$,
    and for all $k\in\R$, there exists $r_3=r_3(u_0,k,\xi,\eps)$,
    with $r_3\to 0$ as $\eps\to 0$, such that the following approximate
    Kruzhkov formulation holds away from the interface $\Set{x=0}$:
    $$
    \int_{\R_+} \int_{\R} \Bigl( \abs{S^\eps u_0-k}\xi_t
    +\mathfrak{q}(x,S^\eps u_0,k) \xi_x \Bigr)
    +\int_{\R}\abs{u_0-k} \xi(0,\cdot)
    \geq-r_3(u_0,k,\xi,\eps).
    $$
\end{itemize}
Then, as $\eps\to 0$, $S^\eps u_0$ converges a.e.~on $\R_+ \times \R$ to
the unique $\mathcal{G}$ entropy solution of
\eqref{eq:ModelProb},\eqref{eq:InitialCond}.
\end{theo}

The applications we have in mind are mainly related to numerical approximations.
Yet it should be noted that the standard vanishing viscosity
approximations of Sections \ref{ssec:VV-in-1D} fulfill assumptions (B1)-(B5).

In this context, the property (B2) is somewhat delicate:
it clearly holds for $(c^l,c^r)\in \mathcal{G}_{VV}^{s}$,
then it can be extended to $(c^l,c^r)\in \mathcal{G}_{VV}$, using
the fact that $\mathcal{G}_{VV}$ is the closure
of $\mathcal{G}_{VV}^{s}$ (see Proposition \ref{prop:G-VV-versus-G-VV-0})
and using the continuity with respect to the initial function $u_0$
of the vanishing viscosity solver $S^\eps$.

Thanks to Theorem \ref{th:conv-via-process}, from the
existence result of Theorem \ref{th:exist-by-numerics} we deduce

\begin{cor}\label{cor:viscosity-convergence-via-process}
If the vanishing viscosity germ $\mathcal{G}_{VV}$
is complete, then the following assumption can be dropped
from Theorem \ref{th:VV-1D-converges}:
\begin{equation}\label{eq:non-degeneracy}
    \text{$f^{l}$ and $f^{r}$ are not affine on any interval $I\subset [0,1]$}.
\end{equation}
\end{cor}

The germ $\mathcal{G}_{VV}$ is known to be complete in many situations
(see \cite{GimseRisebro1,GimseRisebro2,Diehl1,Diehl2,Diehl3,Diehl2008}).
We expect that the vanishing viscosity germ
$G_{VV}$ is complete whenever $f^{l,r}$ are compactly supported, and also
in the case  there exist two sequences of elementary $G_{VV}$-entropy solutions
$(b_k(\cdot))_{k>1}$ and $(B_k(\cdot))_{k>1}$ such
that as $k\to \infty$, $b_k\to -\infty$ and $B_k\to \infty$.

In general, assumption (B2) appears to be too restrictive.
Certainly it holds for numerical schemes using either the
exact Riemann solver on the interface $\Set{x=0}$, or the associated
Godunov flux, but it may become difficult to justify in other situations.
For instance, we cannot simply combine
Theorems \ref{th:exist-by-numerics} and \ref{th:conv-via-process}
in order to drop assumption \eqref{eq:non-degeneracy} from
the statement of Theorem \ref{th:AB-connections-VV-1D-converges};
this is so because solely the connection
solution $A\,\char_{\Set{x<0}}+B\,\char_{\Set{x>0}}$ and the
two constant states $0$ and $1$ are preserved by the
adapted viscosity approximation \eqref{eq:adapted-viscosity}. However,
the investigation of the approximate solutions
for Riemann initial data of the form $u_0=c^l \char_{\Set{x<0}}
+ c^r \char_{\Set{x>0}}$, with $(c^l,c^r) \in \mathcal{G}_{(A,B)}
\setminus \Bigl\{(A,B),(0,0),(1,1) \Bigr\}$, would
require subtler arguments.

Note that whenever the strong compactness property is 
available, the schemes that only preserve
some definite germ $\mathcal{G}_0$ strictly smaller 
than $\mathcal{G}$ do converge, thanks to
the formulation of Proposition \ref{prop:Carrillo-type-defs}
(see \cite{KarlsenRisebroTowers2002b}, \cite{BurgerKarlsenTowers},
\cite{AndrGoatinSeguin} for a few particular cases).

\begin{proof}[Proof of Theorem \ref{th:conv-via-process}]
First, in view of (B1), we have nonlinear weak-$\star$ compactness
of the sequence $(S^\eps u_0)_{\eps>0}$.  Combining it with the
strong convergence of $S^\eps \hat u_0$, with
$\hat u^0$ being any stationary solution from (B2), we
pass to the limit as $\eps\to 0$ (up to extraction
of a subsequence) in the formulations (B3), (B4), and (B5).
Then the nonlinear weak-$\star$ limit $\mu$ of (a subsequence of)
$S^\eps u_0$ is a $\mathcal{G}$-entropy process
solution, according to Remark \ref{rem:simpler-entropyproc-def}.
Applying Theorem \ref{th:UniqProcessSol},
we conclude that $\mu$ is identified with the
unique $\mathcal{G}$-entropy solution $u$
of \eqref{eq:ModelProb},\eqref{eq:InitialCond}, and $u$ is
therefore the unique accumulation point of $(S^\eps u_0)_{\eps>0}$
in $L^1_{\loc}(\R_+\times \R)$.
This concludes the proof.
\end{proof}

\section*{Acknowledgement}
This paper was written as part of the research program on Nonlinear Partial
Differential Equations at the Centre for Advanced Study at the
Norwegian Academy of Science and Letters in Oslo, which took place during the
academic year 2008-09. The authors thank S. Mishra, D. Mitrovi\v{c},
E. Panov, N. Seguin, and J. Towers for interesting discussions.

\end{document}